\documentclass[a4paper,reqno]{amsart}
\usepackage{amsmath,amssymb,amsthm,enumerate}

\usepackage{graphics}
\usepackage{epsfig}
\usepackage{amsfonts}
\usepackage{amssymb}
\usepackage{color}
\usepackage{xcolor}
\usepackage[colorlinks=true]{hyperref}
\hypersetup{linkcolor=violet,urlcolor=cyan,citecolor=cyan}

\usepackage{tikz}
\usetikzlibrary{arrows, calc}
\usetikzlibrary{calc,quotes,angles}
\usepackage{pgfplots}
\usepackage{mathtools}
\usetikzlibrary{intersections}

\usepackage{mathrsfs}
\usepackage{mathtools}

\usepackage[margin=1in]{geometry}
\def\ifl{\iffalse }

\def\bc{\begin{center}} \def\ec{\end{center}}
\def\ba{\begin{array}} \def\ea{\end{array}}
\def\bea{\begin{eqnarray}} \def\eea{\end{eqnarray}}
\def\beaa{\begin{eqnarray*}} \def\eeaa{\end{eqnarray*}}

\theoremstyle{definition}
\newtheorem{thm}{Theorem}[section]
\newtheorem{prop}[thm]{Proposition}
\newtheorem{lem}[thm]{Lemma}
\theoremstyle{remark}
\newtheorem{rem}[thm]{Remark}
\newtheorem*{rem*}{Remark}

\numberwithin{equation}{section}

\newcommand{\R}{\mathbb{R}}

\newcommand{\supp}{\mathop{\mathrm{supp}}}

\newcommand{\pa}{\partial}
\newcommand{\na}{\nabla}
\newcommand{\al}{\alpha}
\newcommand{\be}{\beta}
\newcommand{\de}{\delta}

\newcommand{\Ga}{\Gamma}
\newcommand{\td}{\tilde}
\newcommand{\wtd}{\widetilde}

\newcommand{\Lg}{\langle}
\newcommand{\Rg}{\rangle}

\newcommand{\ls}{\lesssim}

\title[GWP of wave-KG]{Global well-posedness of a two dimensional wave-Klein-Gordon system with small non-compactly supported data}

\author[X. Cheng]{Xinyu Cheng}
 \address{X. Cheng, Research Institute of Intelligent Complex Systems, Fudan University, Shanghai, P.R. China}
\email{xycheng@fudan.edu.cn}

\begin{document}

\maketitle
\begin{abstract}
   In this paper we are interested in the coupled wave and Klein-Gordon equations in $\R^+\times\R^2$. We want to establish the global well-posedness of such system by showing the uniform boundedness of the energy for the global solution without any compactness assumptions on the initial data. In addition, we also demonstrate the pointwise asymptotic behavior of the solution pair. In order to achieve that we apply a modified Alinhac’s ghost weight method together with a newly developed
normal-form framework to remedy the lack of the space-time scaling vector field. Finally we show the global solutions scatter linearly strongly for the Klein-Gordon field $\phi$ and weakly for the wave field $n$ as $t\to+\infty$. To our best knowledge such scattering phenomenon is novel in the literature.
\end{abstract}

\section{introduction}
\subsection{Introduction to the model and main results}
In this work we are interested in certain coupled wave-Klein-Gordon system (wKG) in $\R^+\times\R^2$ that results from the study of differential geometry and plasma physics. More specifically speaking, we first focus on the following two model systems:
\begin{subequations}\label{eq:kg}   
\begin{equation}\label{eq:kg1}
\begin{cases}
&\Box n = A^{\alpha}\partial_{\alpha}(|\phi|^2),
\\
&\Box \phi + \phi = B^{\alpha}\phi\partial_{\alpha}n ;
\end{cases}
\end{equation}  
\begin{equation}\label{eq:kg2}
\begin{cases}
&\Box n = A^{\alpha\beta}\partial_{\alpha\beta}(|\phi|^2),
\\
&\Box \phi + \phi = B n \phi.
\end{cases}
\end{equation}
\end{subequations}
Here the initial data is given by $(\phi,\partial_t \phi, n,\pa_{t}n)|_{t=0}=(\Phi_0, \Phi_1, N_0, N_1),$ $\square = \partial_{tt} - \Delta$ is the standard d'Alembertian and $\Delta$ is the usual Laplacian. We denote $\pa=(\pa_t,\pa_1,\pa_2)$, see Section~\ref{sec:nota} for more information on the notation. The Einstein convention is adopted with the Greek letters running from 0 to 2 and Latin letters running from 1 to 2; $A^{\alpha}, B^{\alpha}, A^{\alpha\beta}, B$ are given constants. We emphasize here that the constants $A^\al,B^\al, A^{\al\be}$ need not obey any null conditions. The unknowns $\phi,n$ take values in $\R^2$ and $\R$ respectively.

 The impetus behind this coupled wKG system \eqref{eq:kg} stems from the integration of two research areas, differential geometry and plasma physics, each offering unique perspectives and insights from different vantage points. The  geometric  problem  concerns  the  stability  of  a  totally  geodesic  wave  map (cf. \cite{AC21,DM10}) from Minkowski space $\R^+\times \R^d$ to a $n-$dimensional manifold $\mathcal{M}^n$. To be more explicit, in \cite{AC21} the following system was formulated:
\begin{equation}\label{eq1-wave-map}
\left\{\begin{aligned}
&-\Box \phi^1 = -2\sum_{k=2}^n\phi^k \pa_\al\phi^k + \text{h.o.t}
\\
&-\Box \phi^k - \phi^k = 2\phi^k \pa_\al \phi^1 + \text{h.o.t.}, \quad k=2,\cdots, n
\end{aligned}\right.
\end{equation}
where $\phi^1$ and $\phi^k$ are scalar functions defined in $\R^+\times\R^d $. Based on this formulation, global stability results on the wave maps in $\R^3$ and higher dimensions have been established in \cite{AC21} ($d\ge 3$), while the 2 dimensional case $(d=2)$ was suggested to be open due to its rather slow decay-in-time. Later \cite{DM10} gave a preliminary answer to the 2 dimensional case, where a vector field method on hyperboloids
dating back to Klainerman and H\"ormander was adopted to show the global well-posedness of \eqref{eq:kg} with small compactly supported initial data. In this article, one of our modest goals is to adopt the usual vector field bootstrap method (armed with the newly-developed normal form strategy for nonlinear wave systems in \cite{Li21}) and eventually remove the compact support requirement. For more discussions on wave maps, we refer the readers to \cite{SS98,Kri07,V70}.

 The significance of such wave map models \eqref{eq1-wave-map} extends beyond its immediate context of differential geometry, finding relevance in both the fields of general relativity and plasma physics (the nonlinear $\sigma$-model including the Skyrme model \cite{S61,LDuke21} for example), where it describes the interaction between Langmuir waves and ion acoustic
waves in plasma via ion density fluctuation and the electric field (cf. \cite{Plasma}). To be clearer, the global well-posedness of \eqref{eq1-wave-map} is profoundly related to the global stability problem of the Klein-Gordon-Zakharov system:
\begin{equation}\label{eq1-Zakharov}\left\{
\aligned
& \Box E^a + E^a = -nE^a,\quad a = 1,2,
\\
& \Box n = \Delta \big(|E^1|^2 + |E^2|^2\big),
\endaligned\right.
\end{equation}
where the ion density $n$ is a scalar function and the electronic field $E=(E^1,E^2)$ take values in $\R^2$. \eqref{eq1-Zakharov} can be dated back to the original Zakharov equation that was firstly introduced in \cite{Zak72} to describe certain type of oscillation of a plasma. To be accurate, it is known from the physical point of view that for strong Langmuir turbulence, the Langmuir phase velocity in the
Klein-Gordon equation in \eqref{eq1-Zakharov} is different (usually about one thousand times as large, cf. \cite{OTT99,Plasma}) from the
ion acoustic phase velocity in the wave equation. {Such phenomenon results from the classic two-fluid, two-timescale theory, taking the huge difference between the mass of an electron and that of an ion into account.} In fact there is no scaling transformation that can make these two speeds
equal (unlike the original Zakharov system as in \cite{Zak72}). Despite all this, such Klein-Gordon-Zakharov system is still regarded as an incredibly interesting model from the mathematical perspective due to its complexity. To be more specific such Klein-Gordon-Zakharov model \eqref{eq1-Zakharov} can be viewed as a nonlinear wave-Klein-Gordon coupled system, therefore numerous techniques have been attempted to tackle the complicated nonlinear coupling structure through the development of wave-related methods, which we put the literature review on hold for the time being (see Section~\ref{Section:HR} below). It is worthwhile pointing out here that the Klein-Gordon-Zakharov system \eqref{eq1-Zakharov} is a special case of our model problem \eqref{eq:kg2} by choosing $A^{00}=A^{\al\be}=0$ for $\al\neq \be$ and $A^{11}=A^{22}=-B^{\al} =1$ for $\al=0,1,2$. 

In this paper, we are interested in the small data global well-posedness (GWP) of the model problem \eqref{eq:kg} and the long-time energy/pointwise asymptotic behavior of the global solutions. Furthermore, we carefully explore the scattering phenomenon of \eqref{eq:kg}. Our first main result states below:

 \begin{thm}\label{thm}
Consider the coupled wave-Klein-Gordon system as in \eqref{eq:kg1} and let $K$ be an integer no less than $12$. For any fixed small parameter $0<\de<\frac12$, there exists small $\varepsilon_0>0$ such that for all $\varepsilon\in (0,\varepsilon_0)$ and all initial data $(\Phi_0,\Phi_1,N_0,N_1)$ satisfying the smallness condition below: 
\begin{align}\label{ID2}
    \|(\Phi_0,\Phi_1)\|_{\text{ID1}}+\|(N_0,N_1)\|_{\text{ID2}}<\varepsilon,
\end{align}
where the initial-data-norms of $(\Phi_0,\Phi_1)$ and $(N_0,N_1)$ are given by{
\begin{equation}\label{ID3}
\begin{split}
   \|(\Phi_0,\Phi_1)\|_{\text{ID1}}= &\|\Lg x \Rg^{K}\Lg \na\Rg^{K+1}\Phi_0\|_2+\|\Lg x \Rg^{K}\Lg \na\Rg^{K}\Phi_1\|_2+\|\Lg x\Rg^{K+1}\log(1+\Lg x\Rg) \Lg \na \Rg^{K}\Phi_0\|_2\\ &+\|\Lg x\Rg^{K+1}\log(1+\Lg x\Rg) \Lg \na \Rg^{K-1}\Phi_1\|_2,
    \\ \|(N_0,N_1)\|_{\text{ID2}}= &\sum_{j=0}^{K-5}\Big(\|\Lg x\Rg^{j+\frac52} \na^{j+3}(\na N_0,N_1)\|_1+\|\Lg x\Rg^{j+\frac52} \na^{j}(\na N_0,N_1)\|_1\Big)\\
    &+\sum_{j=0}^K\|\Lg x\Rg^{j+1}\na^j(\na N_0,N_1)\|_2.
    \end{split}
\end{equation}
}
%   \begin{align}\label{ID1}
% \sum_{j=0}^{K}\|\langle x\rangle^{j+1}\nabla^j(\nabla \Phi_0,\Phi_1)\|_{L^2}
% +\sum_{j=0}^{k_1+3}\|\langle x\rangle^{j+3}\nabla^j(\nabla \Phi_0,\Phi_1)\|_{L^2}
% +\sum_{j=0}^{K}\|\langle x\rangle^{j+1}\nabla^j(\nabla n_0,n_1)\|_{L^2}\le \varepsilon
% \end{align}
% or 
% \begin{align}
% \sum_{j=0}^{K+1}\|\langle x\rangle^{j+2}\nabla^j \Phi_0\|_{L^2}
% +\sum_{j=0}^{K}\|\langle x\rangle^{j+3}\nabla^j \Phi_1\|_{L^2}
% +\sum_{j=0}^{K}\|\langle x\rangle^{j}\nabla^j n_0\|_{L^2}
% +\sum_{j=0}^{K-1}\|\langle x\rangle^{j+1}\nabla^j n_1\|_{L^2}\le \varepsilon.
% \end{align}
We thus can conclude the following:

\texttt{(i).} The Cauchy problem for the \eqref{eq:kg1} system admits a couple of global solutions $(\phi,n)$ in time with the following energy of polynomial growth-in-time:  
\begin{align}\label{ue}
\| \partial \Gamma^{\le K}\phi\|_{2}+\| \Gamma^{\le K}\phi\|_{2}+\|\pa \Gamma^{\le K}n\|_{2}+\|\Ga^{\le K}n\|_2\le C\varepsilon \Lg t\Rg^{\de}
\end{align}
for some constant $C>0$ and $\Ga$ are the usual vector fields excluding the space-time scaling operator $L_0$ (see Section~\ref{sec:nota} for definition). Moreover, we have the following uniform boundedness of the lower-order energy for some constant $C>0$:
\begin{align}\label{ue2}
\| \partial \Gamma^{\le K-1}\phi\|_{2}+\| \Gamma^{\le K-1}\phi\|_{2}+\|\pa\Gamma^{\le K-2}n\|_{2}\le C\varepsilon.
\end{align}

\texttt{(ii).} Such global solutions satisfy the following asymptotic pointwise decay estimate: 
\begin{align}\label{op}
    |\phi(t,x)|\le C_1 \varepsilon \Lg t\Rg^{-1},\quad |\pa n(t,x)|\le C_2 \varepsilon \Lg t\Rg^{-\frac12} \Lg t-|x|\Rg^{-\frac{1+\delta}{2}},
\end{align}
for some positive constants $C_1,C_2>0$. The notation $\Lg t\Rg$ denotes the usual Japanese bracket that we refer to Section~\ref{sec:nota} for more details of the notation.
\end{thm}

\begin{rem}
There exist a few global well-posedeness results for different wave-Klein-Gordon coupled system (including the totally geodesic wave map model) with
small initial data (cf. \cite{AC21,DW20, Kata12, OTT95, D,DM21,DM10}) in different dimensions. For clarification our main contribution to this problem is to prove the global stability of system \eqref{eq:kg} by obtaining the uniform-in-time
boundedness of the energy with small non-compactly supported initial data in $\R^+\times\R^2$. In order to overcome the difficulties including the slow decay-in-time resulting from the low dimension and the compact support requirement, our approach is mainly based on a newly developed normal-type estimate
where we manage to handle the lack of the space-time scaling operator $L_0$ using vector field method on the usual flat space-time. We also provide a proof of a strong linear scattering (see Theorem~\ref{thm1} below) for the Klein-Gordon field $\phi$ and a weak linear scattering for the wave field $n$ as $t\to+\infty$. Another novelty of our work is in giving a clear description of the smallness conditions of non-compactly supported initial data (see Remark~\ref{rem1.2}).
\end{rem}

\begin{rem}
As will be clear later, our approach to prove the global well-posedness of \eqref{eq:kg1} is to separate $n=n_0+A^\al \pa_\al \td{n}$, where
    \begin{align}
\begin{cases}
\Box n_0=0,\\
(n_0,\partial_t n_0)|_{t=0} =(N_0,N_1),
\end{cases}
\qquad \text{and}\qquad\
\begin{cases}
\Box \td{n}=|\phi|^2,\\
(\td{n},\partial_t \td{n}) |_{t=0}=(0,0).
\end{cases}
\end{align}
In effect the global well-posedness of \eqref{eq:kg2} can be dealt with via the same strategy: we decompose $n=n_0+A^{\al\be} \pa_{\al\be} N$ where
\begin{align}
\begin{cases}
\Box n_0=0,\\
(n_0,\partial_t n_0)|_{t=0} =(N_0,N_1),
\end{cases}
\qquad \text{and}\qquad\
\begin{cases}
\Box N=|\phi|^2,\\
(N,\partial_t N) |_{t=0}=(0,0).
\end{cases}
\end{align}
Our energy bootstrap framework holds in such similar situation while the initial condition is slightly different and we will discuss more about the \eqref{eq:kg2} case in Section~\ref{sec:bcase}.
\end{rem}

\begin{rem}\label{rem1.1}
It is worth mentioning here that usually the best pointwise decay-in-time of $(\phi,n)$ one can expect (ignoring the nonlinear structure) is 
\begin{align}\label{1.8}
    |\phi(t,x)|\ls \Lg t\Rg^{-1},\quad |\pa n(t,x)|\ls \Lg t\Rg^{-\frac12}\Lg t-|x|\Rg^{-\frac12},
\end{align}
 which is obviously weaker than \eqref{op}. To see \eqref{1.8} is the best one can expect, we consider the homogeneous linear wave and Klein-Gordon equation in $\R^+\times\R^2$:
\begin{align}
     \Box u=0,\quad \Box v+v=0,
 \end{align}
with sufficiently nice initial data. Then it follows from the usual Klainerman-Sobolev inequality (cf. Lemma~\ref{KSineq}) and the $L^\infty$ estimate (cf. Lemma~\ref{kgt32}). We refer the readers to \cite{H97} for the wave equation and \cite{G1992} for the Klein-Gordon equation. Such pointwise decay-in-time \eqref{1.8} brings major difficulties in showing the global well-posedness (and linear scattering for the wave component), which will be explained in more details in Section~\ref{sec:mainD}. Fortunately such situation can be improved due to the special nonlinear structure, we pause our discussion in this direction for now and will continue in Section~\ref{sec:ol}.
\end{rem}

\begin{rem}
We remark here that the requirement \eqref{ID2} and \eqref{ID3} on the initial condition is of slightly high regularity {(around 12 times of differentiablity in the Sobolev sense)}. In fact in the existing literature there are many applicable techniques such as multilinear space-time estimates and frequency
decomposition that have been used intensively to reduce
regularity assumptions (cf. \cite{GNW14,GNW14b}); however we do not dig deeper in this direction. Our main contribution in this work is to prove the global stability with non-compactly supported initial data via our flat space-time normal form strategy.
\end{rem}

\begin{rem}\label{rem1.2}
 Another point we would like to emphasize is that the initial assumption \eqref{ID2} is essentially different from the ones in nonlinear waves (cf. \cite{CX22}). More specifically speaking, one can observe from \eqref{ID2} that $\|\Lg x\Rg^K \Lg\na\Rg^{K+1}\Phi_0\|_2<\varepsilon$ is yet strong since even $\Phi_0$ has to obey certain spatial decay properties: $\|\Lg x\Rg^K \Phi_0\|_2<\varepsilon$. Unfortunately such assumption is necessary. This can be explained through a crucial observation that when performing energy estimates with vector field $\Omega_{i0}$, one may run into the quantity $\|\Omega_{i0}^4 \phi|_{t=0}\|_2$. With no loss it suffices to consider $\|\Lg x\Rg^4 \partial_t^4\phi|_{t=0}\|_2$. Note that by applying the Klein-Gordon equation twice, one arrive that
 \begin{align}
     \|\Lg x\Rg^4 \partial_t^4\phi|_{t=0}\|_2\lesssim&  \|\Lg x\Rg^4 \partial_t^2\Delta\phi|_{t=0}\|_2+ \|\Lg x\Rg^4 \partial_t^2\phi|_{t=0}\|_2+ \|\Lg x\Rg^4 \partial_t^2(\pa\phi n)|_{t=0}\|_2\\
     \lesssim& \|\Lg x\Rg^4 \Delta^2\phi|_{t=0}\|_2+ \|\Lg x\Rg^4 \Delta \phi|_{t=0}\|_2+ \|\Lg x\Rg^4 \phi|_{t=0}\|_2+\text{nonlinear terms}.
 \end{align}
As shown above, one already needs to control terms such as $\|\Lg x\Rg^4 \phi|_{t=0}\|_2=\|\Lg x\Rg^4 \Phi_0\|_2$. For more discussion on this we refer the readers to \cite{CX22} and Proposition~\ref{prop} later. 
 
\end{rem}

Our next result states that the couple of solutions $(\phi,n)$ scatters linearly:

\begin{thm}\label{thm1}
 Assume the same setting as in Theorem~\ref{thm} is adopted and in addition we denote the energy space {$\mathcal{X}_{K}=H^{K}\times H^{K-1}$ and $\mathcal{Y}_{K}=\dot{H}^{K}\times \dot{H}^{K-1}\cap\dot{H}^1\times L^2$ }where $H^m, \dot{H}^m$ are standard Sobolev spaces in $\R^2$ that are defined in \eqref{eq:hsdot}. We also denote
$D=(D_i)_{i=1}^2$ to be the usual differential operator with $D_1=\pa_1,D_2=\pa_2$. Then the following statements hold:
 
 \texttt{(i)}. The solution $\phi$ scatters to a free solution in $\mathcal{X}_K$ as $t\to+\infty$, namely, there exists $(\Phi_{l_0},\Phi_{l_1})\in \mathcal{X}_K$ such that 
   \begin{align}\label{Phiscat}
    \lim_{t\to+ \infty}\|(\phi,\partial_t\phi)-(\phi_l,\partial_t\phi_l)\|_{\mathcal{X}_K}= 0,
\end{align}
where $\phi_l$ is the linear homogeneous solution to the Klein-Gordon equation with initial data $(\Phi_{l_0},\Phi_{l_1})$.

\texttt{(ii)}. The solution $n$ scatters to a free solution weakly in $\mathcal{Y}_K$ as $t\to+ \infty$, i.e. there exists $(N_{l_0},N_{l_1})\in \mathcal{Y}_K$ such that for any $(F_0,F_1)\in \mathcal{Y}_K$ we have
\begin{align}\label{Nscat1}
    &\lim_{t\to + \infty}\int_{\R^2} \na D^{K-1} (n-n_l)\cdot\na D^{K-1} f+\pa_t \na D^{K-1}(n-n_l)\cdot \pa_t\na D^{K-1} f=0,\\ \label{Nscat2}
\mbox{and}\ &\lim_{t\to + \infty}\int_{\R^2}\na(n-n_l)\cdot \na f+\pa_t(n-n_l)\pa_tf=0,
\end{align}
where $n_l$ is the linear homogeneous solution to the wave equation with initial data $(N_{l_0},N_{l_1})$ and $f$ is the linear homogeneous solution to the wave equation with initial data $(F_0,F_1)$.

\end{thm}

\begin{rem}\label{rem1.8}
   We want to emphasize that the scattering result in Theorem~\ref{thm1} is valid under high regularity assumptions on the initial data; we refer the readers to other scattering results in \cite{D,OTT95} for the Klein-Gordon-Zakharov system with initial data of much higher regularity. Difficulties arise when appealing to initial data of low regularity data (cf. \cite{GNW14}). Nevertheless, the strong scattering of the wave field $n$ cannot be shown through our argument due to the critical decay-in-time as 
   \begin{align}\label{1.18a}
       \|\pa(|\phi|^2)\|_{H^m}(s)\ls \Lg s\Rg^{-1}\implies \int_1^t\|\pa(|\phi|^2)\|_{H^m}\ ds\sim \log(t).
   \end{align}
  In fact whether the wave
field $n$ scatters strongly (linearly or nonlinearly) is unknown yet and we leave it open. However, a weak scattering can be shown from the extra dual structure under the crucial observation $\left|\frac{\phi}{\Lg r-t\Rg}\right|\ls \Lg t\Rg^{-2}$, see \eqref{3.52} later; we also prove that energy of wave field $n$ is uniformly bounded \eqref{ue2}. Indeed both results are necessary to the strong linear scattering. Moreover the weak linear scattering leads to the uniform boundedness of the usual energy (without any weight) of $n$ applying the classic Uniform Boundedness Principle (see Remark~\ref{rem4.2} for more information).
\end{rem}

We recall some previous results in the literature that are closely related to the presenting paper here.

\subsection{Historical review}\label{Section:HR}
As mentioned earlier the model we are interested is of a nonlinear wave and Klein-Gordon coupling system, therefore it is of great significance to go through some related results. To start with, nonlinear wave equations have been widely studied: consider the following second order quasilinear wave equation in $[0,\infty)\times\R^d$ $(d\ge2)$:
\begin{equation}\label{eq:wend}
    \Box u=g^{kij}\pa_{k} u\pa_{ij}u,
\end{equation}
where $g^{kij}$ are constants and assume the initial data behave sufficiently well. For the $d\ge 4$ case, small data global solutions of \eqref{eq:wend} were known to exist (cf. \cite{H97}). When $d=3$, GWP of \eqref{eq:wend} was obtained in the pioneering work by Klainerman \cite{K86} and Christodoulou \cite{C} under the null condition ($g^{kij} \omega_k \omega_i \omega_j=0,$ for $\omega_0=-1$ and $(\omega_1,\omega_2,\omega_3)\in \mathbb{S}^2$). On the other hand, the solutions admit finite time blow-up behaviors without the null conditions (cf. \cite{J81}). When it comes to the $d=2$ case, the relatively slow decay-in-time causes major difficulties. To overcome such challenge Alinhac in the seminal work \cite{Alinh01_1,Alinh01_2} shows that \eqref{eq:wend} admits a small data global smooth solution under the null conditions by introducing the well-known ``ghost weight'' technique and on the contrary the solution blows up in finite time if the null conditions are not satisfied. In addition in \cite{Alinh01_1} Alinhac showed that the highest norm of the solution grows at most polynomially in time. Recently many work have successfully improved the previous results and we list several of relation here. In \cite{Li21} the author proved the uniform boundedness of the highest norm of the solution by developing a new normal-form type strategy; in \cite{D21} similar results were obtained by applying vector field method on hyperboloids dating back to Klainerman and H\"ormander. Beyond these, there are studies focusing on dealing with models where the Lorentz invariance is not available. Such systems include non-relativistic wave systems with multiple wave speeds (cf. \cite{ ST01}) and %nonlinear wave equations on non-flat space-time (cf. \cite{Y16}),
exterior domains (cf. \cite{M07}). %and incompressible elastodynamics (cf. \cite{L16}). 
It is also worth pointing out that in the work \cite{Li21}, the author developed a new systematic normal-form framework and proved GWP of 2D quasilinear wave equations without the Lorentz boost $\Omega_{i0}$; in fact this strategy inspired us to handle models where the space-time scaling operator $L_0$ is not applicable. 

% It is also worth mentioning \cite{CLLX} where the authors obtained GWP of the 2D quasilinear wave system with non-compactly supported small initial data by developing a novel $(L^2, L^\infty)$ estimates.

The nonlinear Klein-Gordon equations have been roundly studied as well. In the pioneering work by Klainerman \cite{K85} and Shatah \cite{Shata85}, the Klein-Gordon equations with quadratic nonlinearities were shown to possess small global solutions. Motivated by the large amount of inventive work on nonlinear wave equation and Klein-Gordon equation mentioned above, the coupled wave and Klein-Gordon systems have arouse a great deal of interest for decades. Among which, to our best knowledge, one of the very first result belonging to this area was derived by Bachelot \cite{B88} on
the Dirac-Klein-Gordon equations. Later in \cite{G90} Georgiev proved GWP with strong null nonlinearities. Much more physical models described by the coupled wave and Klein-Gordon systems
have been studied since then. We list a few models of interest here: the Klein-Gordon-Zakharov equations\cite{D,DW20,Kata12,OTT95}, the
Maxwell-Klein-Gordon equations \cite{KWY20,P99,P99b}, the Einstein-Klein-Gordon equations \cite{LM16,Wang20}.%, and Dirac-Klein-Gordon model \cite{DLY22}. 

We also review here some results concerning the wave-Klein-Gordon type coupled system such as the Zakharov system arising in plasma physics. We refer the readers to its physical background in \cite{Plasma,Zak72} and briefly review some related models. We recall that the Euler-Maxwell equations are of fundamental importance in the area of plasma physics. In fact one can derive the Zakharov equations via the Euler equation for the electrons and ions coupled with the Maxwell equation for the electric field (cf. \cite{Te07}) and in particular the Euler-Poisson system in $\R^+\times \R^2$ was proved to admit
global solutions by Li-Wu \cite{LW14}. %The work \cite{IP14} studied the GWP of the quasilinear Klein-Gordon systems with multiple speeds and masses.
There are many other interesting results concerning different aspects of other closely related models of the wave-Klein-Gordon system. We hereby list a few of them: GWP of the Klein-Gordon-Zakharov equations with multiple propagation speeds \cite{OTT99}, the convergence of the Klein-Gordon-Zakharov equations to the Schr\"odinger
equation (also the Zakharov system) as certain parameters go to $+\infty$ (cf. \cite{MN08}), the finite time blow-up for Klein-Gordon-Zakharov with rough data \cite{SW19}, the long-term pointwise decay of the small global solution (cf. \cite{D_KG21,DM10,GY95}), scattering (cf. \cite{GNW14, GNW14b, HPS13}). Other related developments with different strategies can be found in the papers  \cite{DM21,DW20,M222,WY14,Y22,Z22}.
 %([26, 34] from Dong-kgz).}

The presenting paper was mainly motivated by several papers \cite{CX22,D,DM21,DM10}. In the work \cite{D}, the authors studied the Klein-Gordon-Zakharov equations on $\R^+\times \R^3$ where they showed the energy of the global solutions pair is uniformly bounded with initial data therein having rather high regularity (around 15 times of differentiability in the Sobolev sense). Their proof relies on certain $L^\infty-L^\infty$ estimates of both wave and Klein-Gordon components. The paper \cite{CX22} studied the 3D Klein-Gordon-Zakharov as well where the regularity assumption of the initial data was reduced to around 8 times of differentiability with the usual energy bootstrap method (armed with the newly-developed normal form strategy introduced in \cite{Li21}). When it comes to the 2D case, the work \cite{DM21} studied the Klein-Gordon-Zakharov equations on $\R^+\times \R^2$ and showed such system admits global small solutions with non-compactly supported initial data. However, they assumed the wave field $n$ is of the form $\Delta n^\Delta$ for some function $n^\Delta$; this in return required the initial data $N_0$ to be in $H^{-2}(\R^2)$. Actually such transformation makes partial sense due to a key energy conservation law; we refer the readers to \cite{GNW14} and the references therein, where the energy was defined as
\begin{align}\label{1.19}
    E=\int |\phi|^2+|\pa \phi|^2+\frac{|D^{-1} \pa_t n|^2+|n|^2}{2}-n|\phi|^2\ dx.
\end{align}
Last but not least, the work \cite{DM10} studied the exact wave-Klein-Gordon system as in \eqref{eq:kg}. The authors obtained global small solutions with compactly supported solutions where a vector field on hyperboloids strategy dating back to Klainerman and H\"ormander was adopted. Motivated by the existing results on coupled wave and Klein-Gordon systems, one of our modest
goals is to show the global existence result without any compactness assumptions on the initial data through a new energy bootstrap based on the normal-form strategy. In addition we want to give a clear description of the pointwise asymptotic behavior together with showing the linear scattering of these global solutions.

\subsection{Main Difficulties}\label{sec:mainD}
Before we sketch the strategy of our proof, we briefly demonstrate the main difficulties of this model problem here. To begin with, as we mentioned earlier in Remark~\ref{rem1.1}, the best pointwise decay of $(\phi,n)$ one can expect (ignoring the constant) is \eqref{1.8} (instead of \eqref{op}):
\begin{align}\label{1.10}
    |\phi(t,x)|\lesssim \Lg t\Rg^{-1},\quad |\pa n(t,x)|\lesssim \Lg t\Rg^{-\frac12} \Lg t-|x|\Rg^{-\frac12}.
\end{align}
Recall that the Klein-Gordon component
\begin{align}\label{1.11}
    \Box \phi+\phi =B^\al\phi \pa_\al n.
\end{align}
The usual energy estimate (by taking $L^2$ inner product  with $\partial_t \phi$ on \eqref{1.11}) yields 
\begin{align}
    \frac{d}{dt}(\| \partial \phi\|_2^2+\|\phi\|_2^2)=\int_{\R^2}B^\al \phi \pa_\al n \cdot\partial_t\phi \ dx,
\end{align}
where $\partial=(\partial_t,\nabla)$ (see Section 2 for the notation). It is then very natural to integrate in time and estimate the following quantity:
\begin{align}
    \int_0^t \|\phi \pa n\|_2(s) \ ds\lesssim\int_0^t \|\phi\|_{\infty} \|\pa n\|_2 \ ds +\int_0^t \|\pa n\|_{\infty} \|\phi\|_2 \ ds.
\end{align}
Taking \eqref{1.10} into consideration and naively assuming the $L^2$ norms are uniformly bounded in time we arrive at:
\begin{align}
    \int_0^t \|\phi \pa n\|_2 \ ds\lesssim \int_0^t (\Lg s\Rg^{-1}+\Lg s\Rg^{-\frac12})\ ds \lesssim \log(1+t)+\Lg t\Rg^{\frac12}, 
\end{align}
which even implies a blow-up of the energy as $t$ goes to $+\infty$! We believe there are two reasons causing such terrible blow-up and holding us back. 
\begin{itemize}
    \item First we observe that the decay $\Lg t-r\Rg^{-\frac12}$ has not been made fully use of while taking supremum in space. To our understanding, one of the main advantages of the aforementioned hyperbolic change of variable (vector field method on hyperboloids) is that one can gain better control of the so-called ``conformal energy'' thanks to the extra integrability in the hyperbolic time $s=\sqrt{t^2-r^2}$ (cf. \cite{D_KG21}). 
    \item Second is that we did not pay proper attention to the crucial machinery of the Klein-Gordon field $\phi$ that extra decay in $\Lg t+r\Rg$ can be gained with a trade-off losing $\Lg t-r\Rg$. For a more detailed discussion, we refer the readers to Remark~\ref{rem2.12} or Lemma~\ref{lem3.5} and Lemma~\ref{lem3.7a} alternatively.
\end{itemize}
To handle the above challenges, we follow the similar idea as in \cite{CX22}. Briefly speaking, we are about to apply the well-adapted Alinhac's ghost weight method to absorb extra $\Lg t-r\Rg$ weight and thereby generate better decay-in-time. Afterwards we will take good care of the nonlinear structure together with the key machinery of the Klein-Gordon field $\phi$ mentioned earlier to derive even better decay-in-time. Another main difficulty is that the space-time scaling is not invariant in the KGZ system therefore the $L_0$ vector field is not applicable. In the presenting paper, we develop a new normal-form strategy that is motivated by the work \cite{Li21} where the Lorentz transformation is not invariant. Lastly, another question is whether one can prove the scattering (linearly or nonlinearly) of the wave field $n$. Unlike the Klein-Gordon field $\phi$ is known to scatter linearly, scattering of the wave field $n$ is still an open question in $\R^+\times \R^2$ in the Klein-Gordon-Zakharov context. Such problem results from the critical decay-in-time due to the low dimension as explained earlier in Remark~\ref{rem1.8}. In other words, when approaching to show the linear scattering of $n$ in the energy space we usually argue by semi-group techniques. As a result, it seems inevitable to estimate terms such as $\|\pa(|\phi|^2)\|_{H^m}$, which decays as $\Lg t\Rg^{-1}$ at best. Therefore the space-time integral is not uniformly bounded \eqref{1.18a}:
$$ \int_1^t \|\pa(|\phi|^2)\|_{H^m}\ ds\sim \log(t).$$
Although it seems unlikely to show the wave field $n$ scatters linearly, we manage to prove a weak scattering result from the extra dual structure under the crucial machinery mentioned earlier $\left|\frac{\phi}{\Lg r-t\Rg}\right|\ls \Lg t\Rg^{-2}$, see Lemma~\ref{lem3.7a} later. In addition we also prove that energy of wave field $n$ is uniformly bounded \eqref{ue2}. It is worth emphasizing that both the uniform boundedness of the energy and the weak linear scattering are necessary to the strong linear scattering.

\subsection{Main steps of the proof.} \label{sec:ol}
We hereby outline the key steps of our proof. To illustrate the idea clearly, we fix any multi-index $\alpha$ and denote $\Phi = \Ga^\al \phi$ and $V=\pa\Ga^\al n$ with $|\al|=K$ (we suppress the dependence on $\alpha$ to ease the notation). We point out here $\Ga$ cannot be $L_0$ due to the lack of space-time scaling invariance; for the exact definition we refer to \eqref{def_Gamma0}--\eqref{def_Gamma}. 

\texttt{Step 1.} {Energy estimate.} To start with, we make the following $a\ priori$ hypothesis:
\begin{align}
% &\|\partial \Gamma^{\le K}\phi\|_{L^2(\R^3)}^2+\| \Gamma^{\le K}\phi\|_{L^2(\R^3)}^2+\int_{t_0}^t\int_{\R^3}\frac{|\Gamma^{\le K}\phi|^2}{\langle r-t\rangle^{1+\delta}} dxds\le (C\varepsilon)^2,\label{A1}\\
% &\| \Gamma^{\le K}n\|_{L^2(\R^3)}^2 \le (C\varepsilon)^2,\label{A11}\\
|\Gamma^{\le k} n|\le C\varepsilon\langle  t\rangle^{-\frac12}\langle r-t\rangle^{-\frac{1+\de}{2}},\quad |\Gamma^{\le k} \phi|\le C\varepsilon \langle t+r\rangle^{-1},\label{Ab}
\end{align}
where positive constant $C$ is chosen later and $k$ is around half size of $K$. We first apply Alinhac's ghost weight method by choosing $p(r,t)=q(r-t)$ with $q'(s)$ scaling as $\Lg s\Rg^{-1-\de}$. In light of such choice of $q$, we remark here that it is already the best as far as we know; in fact it seems totally non-trivial to find a ``ghost weight'' with $q(s)$ scaling exactly as $\Lg s\Rg^{-1}$. As a matter of fact, the $\Lg s\Rg^{-1-\de}$ scale brings technical difficulties to both the energy estimate and the pointwise asymptotic analysis. To see this, we briefly discuss about performing energy estimats below. Let us focus on the Klein-Gordon equation after applying $\Ga^\alpha$:
\begin{align}\label{1.16}
    \Box \Phi+\Phi=\Ga^\al(B^\be\phi \pa_{\be}n).
\end{align}
Then by taking the $L^2$ inner product on \eqref{1.16} with $\partial_t \Phi\cdot \exp(p)(r,t)$ we can obtain that
{\small\begin{align}
\frac12\frac{d}{dt}(\| e^\frac{p}{2}\partial \Phi\|_{2}^2+\| e^\frac{p}{2} \Phi\|_{2}^2) +\frac 12
   \int e^p q^{\prime} (|T \Phi|^2+|\Phi|^2) =\sum_{I_1+I_2=I}C^\al_{I; I_1,I_2} \int e^p\pa_\al\Gamma^{I_1}n \Gamma^{I_2}\phi \partial_t \Phi
\end{align}}
for some constant $C^\al_{I; I_1,I_2}$. For the sake of simplicity, we focus on the terms with the highest index, i.e. $\int e^p (\Phi \pa n +V\phi)\partial_t \Phi\ dx$. At first glance, clearly we can control
\begin{align}\label{1.18}
    \left|\int V\phi e^p\partial_t\Phi\ dx\right|\lesssim\|V\|_2\|\phi\|_\infty\|\partial_t\Phi\|_2\lesssim\Lg t\Rg^{-1}\| V\|_2\|\partial_t\Phi\|_2.
\end{align}
The other term is handled by applying the usual Cauchy-Schwartz inequality as below:
\begin{equation}
\left|\int \Phi \pa n e^p\partial_t\Phi\ dx\right|\le \frac14  \int e^p q'\cdot |\Phi|^2\ dx+ \mbox{constant} \cdot \int |\pa \Phi|^2 \cdot\frac{|n|^2 }{q'} e^p\ dx,
\end{equation}
which leads to the energy estimate of the Klein-Gordon field {(here  $\frac{|n|^2 }{q'}\sim t^{-1}$ by \eqref{Ab} )}. Similar energy estimates on the wave component yields the same result and both lead to a polynomially growth of the energy, namely, $\| \pa \Phi\|_2+\|\Phi\|_2\ls \Lg t\Rg^{\de_1}$ for arbitrarily small $\de_1>0$. In order to close the bootstrap scheme by showing the pointwise asymptotic decay-in-time, it is necessary to prove the lower-order energy is uniformly bounded for both wave field $n$ and Klein-Gordon field $\phi$. To proceed, we notice from the previous work \cite{CX22} that the wave field $\pa n$ indeed is able to absorb extra $\Lg r-t\Rg$ while in return the crucial machinery of the Klein-Gordon $\frac{\phi}{\Lg r-t\Rg}$ can gain better decay $\Lg t\Rg^{-2}$. More specifically, our approach is to decompose $n$ into two parts: $n = n_0 + A^\al \pa_\al \td{n}$, where 
\begin{align}
\begin{cases}
\Box n_0=0,\\
(n_0,\partial_t n_0)|_{t=0} =(N_0,N_1),
\end{cases}
\qquad \text{and}\qquad\
\begin{cases}
\Box \td{n}=|\phi|^2,\\
(\td{n},\pa_t\td{n}) |_{t=0}=(0,0).
\end{cases}
\end{align}
This kind of reformulation is dated back to \cite{Kata12}, where the author considered nonlinear wave equations with
nonlinearities of divergence form. One way to understand this decomposition is to treat $\td{n}$ as a small perturbation of order $O(\varepsilon^2)$ (the ``leading'' term $n_0$ is of order $O(\varepsilon)$). We refer the readers to Remark~\ref{rem3.3}. For the details of the proof of the uniform energy boundedness, we refer the readers to Section~\ref{sec:energy}.

\texttt{Step 2.} {Pointwise decay-in-time and scattering.} We then close the energy bootstrap by showing the asymptotic pointwise decay \eqref{op}. The decay of the Klein-Gordon component $\phi$ follows from a result by Georgiev in \cite{G1992} (cf. Lemma~\ref{kgt32}). In short words, it suffices for us to bound
\begin{align}
    \|\Lg t+|\cdot|\Rg\Ga^{\le k+4}(\phi \pa n)\|_2.
\end{align}
We notice that Lemma~\ref{lem2.7} plays an important role in dealing with terms such as
\begin{align}
    \|\Lg t+|\cdot|\Rg(\Ga^{\le k+4}\phi)\pa n \|_2+ \|\Lg t+|\cdot|\Rg(\Ga^{\le k+4}\pa n)\phi\|_2.
\end{align}
As explained earlier we indeed view Lemma~\ref{lem2.7} as an iteration process: the Klein-Gordon component $\phi$ can gain better $\Lg t+r\Rg$ decay as long as it can absorb more $\Lg t-r\Rg$ weight. For the wave component $n=n_0+A^\al \pa_\al \td{n}$. $n_0$ is a free wave solution therefore we handle it by developing a $L^\infty-L^1$ estimate (see Lemma~\ref{lemhom} and Lemma~\ref{lemB2}). On the other hand, to estimate the $\td{n}$ part and remedy the lack of $L_0$ vector fields, we employ $L^{\infty}$ and $L^2$ estimates involving the weight-factor $\langle
r-t\rangle$. The key observation (see Lemma~\ref{lem2.6}) is that one can control $\Lg t-r\Rg\partial^2u$. At the expense of the smallness of the energy, we obtain the desired asymptotic decay-in-time. These in turn lead to the closure of the energy bootstrap with a careful choice of the constants. In the end, we show the Klein-Gordon field $\phi$ scatters linearly as $t\to+\infty$. In fact this results from a semi-group method and the uniform boundedness of the energy in time. Although it seems unlikely to show the wave field $n$ scatters linearly by our method, we still manage to prove a weak scattering result from the extra dual structure (see Theorem~\ref{thm1}). In addition we also prove that energy of wave field $n$ is uniformly bounded \eqref{ue2}. We highlight here that both the uniform boundedness of the energy and the weak linear scattering are necessary to the strong linear scattering of the wave field $n$.

\subsection{Organization of the present paper}The rest of this work is organized as follows. In Section 2 we collect the notation and some
preliminaries including a few useful lemmas. In Section~\ref{sec:energy} we show the uniform boundedness of the  energy. The pointwise decay in time and scattering are included in Section~\ref{sec:close}, therefore the proof of Theorem~\ref{thm} and Theorem~\ref{thm1} is complete (for the \eqref{eq:kg1} case). In Section~\ref{sec:bcase}, we briefly discuss about \eqref{eq:kg2} case. We leave some extra analysis related to the initial data in the Appendix.

\section{Preliminaries and notation}

\subsection{Notation}\label{sec:nota}
We shall use the Japanese bracket notation: $ \langle x \rangle = \sqrt{1+|x|^2}$, for $x \in \mathbb R^2$.  We denote  $\partial_0 = \partial_t$,
$\partial_i = \partial_{x_i}$, $i=1,2$ and  
\begin{align}
& \partial = (\partial_i)_{i=0}^3, \; \Omega_{ij}=x_i\partial_j-x_j\partial_i, 1\le i<j\le 2;\; \Omega_{i0}=t\pa_{i}+x_{i}\pa_{t}, 1\le i\le 2;
\; \Omega = r \partial_t + t \partial_r;\notag\\
& \Gamma= (\Gamma_i )_{i=1}^{6}, \quad\text{where } \Gamma_1 =\partial_t, \Gamma_2=\partial_1,
\Gamma_3= \partial_2, \Gamma_4=\Omega_{10},\Gamma_5=\Omega_{20},
\label{def_Gamma0} \\& \Gamma_6=\Omega_{12}=\partial_\theta,
\label{def_Gamma}\\
& \Gamma^{\alpha} =\Gamma_1^{\alpha_1} \Gamma_2^{\alpha_2}
\cdots\Gamma_{6}^{\alpha_{6}}, 
\qquad \text{$\alpha=(\alpha_1,\cdots, \alpha_{6})$ is a multi-index};  \notag\\
%& \partial_+ =\partial_t + \partial_r, \qquad \partial_- =\partial_t - \partial_r;  \notag\\
&\widetilde\Gamma=(\widetilde\Gamma_i )_{i=1}^{7}=(L_0,\Gamma),\qquad L_0=t\partial_t+r\partial_r,\qquad r=|x|; \label{def-tdGamma}\\
& \widetilde\Gamma^{\beta} =\Gamma_1^{\beta_1} \Gamma_2^{\beta_2}
\cdots\Gamma_{7}^{\beta_{7}}, 
\qquad \text{$\beta=(\beta_1,\cdots, \beta_{7})$ is a multi-index};  \notag\\
%&\widetilde\Gamma\in\{\Gamma,L_0\}=\{\partial_{t},\partial_{1},\partial_2,\partial_3,\Omega_{10},\Omega_{20},\Omega_{30},\Omega_{12},\Omega_{13},\Omega_{23}\};
& T_i = \omega_i \partial_t + \partial_i, \; \omega_0=-1, \; \omega_i=x_i/r, \, i=1,2. \label{DefT}
\end{align}
Note that $T_0=0$.
For simplicity of notation,
we define for any integer $k\ge 1$,  $\Gamma^k = (\Gamma^{\alpha})_{|\alpha|=k}$,
$\Gamma^{\le k} =(\Gamma^{\alpha})_{|\alpha|\le k}$.
In particular
\begin{align}
|\Gamma^{\le k} u | = \left(\sum_{|\alpha|\le k} |\Gamma^{\alpha} u |^2\right)^{\frac 12}.
\end{align}
Informally speaking, it is useful to think of  $\Gamma^{\le k} $ as any one of the vector
fields $ \Gamma^{\alpha}$ with $|\alpha| \le k$. We also denote
$D=(D_i)_{i=1}^2$ to be the usual differential operator with $D_1=\pa_1,D_2=\pa_2$. We shall need the following convention for multi-indices: for $\beta=(\beta_1,\cdots,\beta_{6})$ and $\alpha=(\alpha_1,\cdots,\alpha_{6})$, we denote $\beta<\alpha$ if $\beta_i\le\alpha_i $ for $i=1,\cdots,6$ and $|\beta|<|\alpha|$ (Here $|\alpha|=\sum_{i=1}^{6}\alpha_i$). Similarly we denote $\beta\le \alpha$ if $\beta_i\le\alpha_i $ for $i=1,\cdots,6$. 
The commutator of the vector fields is given below and the proof is standard (cf. \cite{H97}).
\begin{lem}[Commutator]
For any given multi-index $\al=(\al_{1},\cdots,\al_{\iota})$, we have 
\begin{align}
    &[\pa_{i}, \Gamma^{\al}]
    =\sum_{|\be|\leq |\al|-1}a_{\al\be}\Gamma^{\be} \pa
    =\sum_{|\be|\leq |\al|-1}\tilde{a}_{\al\be}\pa\Gamma^{\be},\\
   &[\Box, \Ga^\al]= [\Box+1,\Ga^\al]=0,
\end{align}
where $\Gamma^{\al}=(\Gamma^{\al_{1}},\cdots,\Gamma^{\al_{\iota}})$,  $a_{\al\be}$, $\td{a}_{\al\be}$ are constants and $[A,B]=AB-BA$ is the usual commutator.
\end{lem}

For a real-valued function $u:\R^2\to \R$ we denote its usual Lebesgue $L^p$-norm by
\begin{align}
    \|u\|_{p}=\|u\|_{L^p(\R^2)}=\begin{cases}
       & \left(\int_{\R^2} |u|^p\ dx\right)^{\frac{1}{p}},\quad  1\le p<\infty;\\
       & \operatorname{esssup}_{x\in\R^2}|u(x)|,\quad p=\infty.
    \end{cases}
\end{align}
We use the following convention for the Fourier transform pair:
\begin{align}
    \hat{f}(\xi)=\int_{\R^2}f(x)e^{-2\pi i \xi\cdot x}\ dx,\qquad f(x)=\int_{\R^2}\hat{f}(\xi) e^{2\pi i \xi\cdot x}\ d\xi.
\end{align}
We thereby denote the usual Sobolev norm for $0\le s<\infty $ as follows:
\begin{subequations}\label{eq:hsdot}   
\begin{equation}
    \|u\|_{\dot{H}^s}=\|u\|_{\dot{H}^s(\R^2)}=\||\na|^s u\|_2=\|(2\pi |\xi|)^s\hat{f}(\xi)\|_2,
\end{equation}
\begin{equation}
    \|u\|_{{H}^s}=\sqrt{\|u\|_2+\|u\|_{\dot{H}^s}}=\|\Lg 2\pi |\xi|\Rg^s\hat{f}(\xi)\|_2.
\end{equation}
\end{subequations}
For integer $J\ge 3$, we shall denote 
\begin{align}
E_J = E_J(u(t,\cdot)) = \| (\partial \Gamma^{\le J} u)(t,\cdot) \|_{L_x^2(\mathbb R^2)}^2.
\end{align}

For any two quantities $A$, $B\ge 0$, we write  $A\lesssim B$ if $A\le CB$ for some unimportant constant $C>0$ and such $C$ may vary from line to line if not specified.
%We write $A\lesssim_{Z_1,\cdots,Z_k} B$ if $A\le CB$ where $C>0$ depends on
%the parameters ($Z_1,\cdots, Z_k$).
We write $A\sim B$ if $A\lesssim B$ and $B\lesssim A$. We write $A\ll B$ if
$A\le c B$ and $c>0$ is a sufficiently small constant. The needed smallness is clear from the context.

{Throughout this work we assume $t\ge 2$ if it is not specified.}

It is convenient to adopt the following notation for a smooth cut-off function $\chi(s)$ in this work:
\begin{align}\label{chi}
    \chi(s)\coloneqq \begin{cases}
        &0, \quad s\le 1,\\
        &1, \quad s\ge 2.
    \end{cases}
\end{align}

\subsection{Decay estimates}
We collect some decay estimates in the following lemmas.

\begin{lem}[Klainerman--Sobolev]\label{KSineq}
Let $h(t,x)\in C^\infty([0,\infty)\times \mathbb{R}^2)$ and $ h(t,x)\in \mathcal{S}(\R^2)$  for every $t>0$. Then
\begin{align}\label{ksineqa}
\Lg t+|x|\Rg^{\frac12}\Lg t-|x|\Rg^{\frac12}|h(t,x)|\lesssim  \Vert \widetilde\Gamma^{\leq 2} h(t,\cdot) \Vert_{2}\qquad \forall t>0, \ x\in\mathbb{R}^2.
\end{align}

\end{lem}
\begin{proof}
    The proof is standard, we refer the readers to \cite{H97} for example. %One can find it in Lemma 2.3 of \cite{CLLX} for example. 
    We also emphasize that \eqref{ksineqa} involves $L_0$ that may lead to trouble when applied to the wKG system, so that one must be carefully examining $L_0$ terms. 
\end{proof}

\begin{lem}
For $x\in \R^2$ with $|x|\le\frac t2$, $t\ge 1$, we have for $u\in \mathcal{S}(\R^2)$ 
\begin{align}\label{t12}
t^{\frac 12}|u(t,x)|\lesssim \|\partial^{\le 1}{\Gamma^{\le 2}} u\|_{L_x^2(\R^2)}
\end{align}
where $\pa=\pa_t,\pa_1,\pa_2$. Moreover, for the case $|x|\ge \frac{t}{2}$ with $t\ge1$, we have
\begin{align}\label{t12b}
|x|^{\frac12}|u(t,x)|\lesssim \|\na^{\le 1}\pa_{\theta}^{\le 1} u\|_{L_x^2(\R^2)}.
\end{align}

\end{lem}
\begin{proof}
We refer the readers to Lemma 2.4 in \cite{G1992} for the proof of \eqref{t12} and \eqref{t12b} follows from a simple Sobolev inequality: assuming $x=(r\cos\theta,r\sin\theta)$ we get
\begin{align}
    r|u(t,r,\theta)|^2\le& r\int_r^\infty |\partial_\rho (u(t,\rho,\theta))^2|\ d\rho\\
    \lesssim& \int_r^\infty |u(t,\rho,\theta)||\partial_\rho u(t,\rho,\theta)| \rho d\rho\\
    \lesssim& \| \partial^{\le 1}_{\theta} u\|_2^2+\| \partial^{\le 1}_{\theta} \partial_\rho u\|_2^2\qquad (\text{by }|u|\lesssim \|\partial^{\le 1}_\theta u\|_{L^2_\theta})\\
    \lesssim& \|\na^{\le 1}\pa_\theta^{\le 1}u\|_2^2.
\end{align}
\end{proof}

\begin{lem}
Assume $|x|\ge \frac t2$ and $t\ge 1$. Let $\de>0$ be an arbitrarily small parameter, then the following estimate holds true any for $u\in \mathcal{S}(\R^2)$:
\begin{align}\label{t1}
\langle t-|x|\rangle^{\frac{1+\de}{2}}\langle t+|x|\rangle^{\frac12}|u(t,x)|\lesssim\left\|\chi(\frac{4|\cdot|}{t})\langle t-|\cdot|\rangle^{\frac{1+\de}{2}}\pa_r^{\le 1} \pa_\theta^{\le 1}u\right\|_2.
\end{align}
\end{lem}

\begin{proof}
This lemma is a 2D version of the one in \cite{CX22}.
For the 3D case we refer the details to \cite{CX22} and here we provide a sketch for the sake of completeness. We first write $u(t,x)$ in the polar coordinates as $u(t,r,\theta)$ (abusing the notation in $u$) with $x=(r\cos\theta,r\sin\theta)$, note that we have 
$|u|\lesssim \|\partial_{\theta}^{\le 1} u\|_{L_{\theta}^2}$.
If $|t-r|\le 1$, it follows that 
{\begin{align}
\langle t-r\rangle^{1+\de} t|u|^2\lesssim t|u(t,x)|^2
\lesssim r|u(t,x)|^2\lesssim \left\|\chi(\frac{4|\cdot|}{t})\na^{\le 1}\Ga^{\le 1}u\right\|_2^2\qquad(\text{by } r\ge\frac{t}{2} \mbox{ and } \eqref{t12b}).
\end{align}}
Therefore we can assume $|t-r|\ge 1$. Using the standard Sobolev inequality we have for any $h\in \mathcal S(\R)$ and for $\frac t2\le r< t$,
\begin{align}
(t-r)^{1+\de}r|h(r)|^2\lesssim &\int_{r}^t|\partial_\rho \left((t-\rho)^{1+\de}\rho|h(\rho)|^2\right)|d\rho
\\ \lesssim&\int_{r}^t|h(\rho)|^2 (t-\rho)^\de\rho d\rho+\int_{r}^t|t-\rho|^{1+\de}|h(\rho)|^2d\rho+\int_{r}^t |t-\rho|^{1+\de}\rho|\partial_\rho h(\rho)||h(\rho)|d\rho
\\ \lesssim&\int_{r}^t\Lg t-\rho \Rg^{1+\de}|h(\rho)|^2 \rho d\rho+\int_{r}^t\Lg t-\rho\Rg^{1+\de}|\partial_\rho h(\rho)|^2 \rho d\rho.
\end{align}
For $r>t$, we can obtain that
\begin{align}
(r-t)^{1+\de}r|h(r)|^2\lesssim &\int_{t}^r|\partial_\rho \left((\rho-t)^{1+\de}\rho|h(\rho)|^2\right)|d\rho
\\ \lesssim&\int_{t}^r(\rho-t)^{\de}|h(\rho)|^2 \rho d\rho+\int_{t}^r|t-\rho|^{1+\de}|h(\rho)|^2d\rho+\int_{t}^r |t-\rho|^{1+\de}\rho|\partial_\rho h(\rho)||h(\rho)|d\rho
\\ \lesssim&\int_{t}^r\Lg \rho-t\Rg^{1+\de}|h(\rho)|^2 \rho d\rho+\int_{t}^r\Lg \rho-t\Rg^{1+\de}|\partial_\rho h(\rho)|^2 \rho d\rho.
\end{align}
% For \cmtr{$r\ge2t$}, we use the following Sobolev inequality
% \begin{align}
% r^3|h(r)|^2\lesssim \int_{r}^\infty|\partial_\rho \left(\rho^3|h(\rho)|^2\right)|d\rho
% \lesssim&\int_{r}^\infty|h(\rho)|^2 \rho^2d\rho+\int_{r}^\infty \rho^3|\partial_\rho h(\rho)||h(\rho)|d\rho
% \\ \lesssim&\int_{r}^\infty|h(\rho)|^2 \rho^2d\rho+\cmtr{\int_{r}^\infty(\rho-t)^2|\partial_\rho h(\rho)|^2 \rho^2d\rho}.
% \end{align}
These lead to \eqref{t1} after applying $|u|\lesssim \|u\|_{H^1(\mathbb{S}^1)}$.
%\begin{align}
%\langle t-r\rangle^{\frac12} \langle t+r\rangle^{\frac12}|u(t,x)|\lesssim \|\Gamma^{\le 1}u\|_{2}+\|\langle t-|\cdot|\rangle\nabla \Gamma^{\le 1}u\|_{2}.
%\end{align}
\end{proof}

\begin{lem}\label{lem2.6}
Suppose $u=u(t,x)$ has continuous second order derivatives. Then for  $t>0$ and $r=|x|$ we have
\begin{align}\label{2.260}
&|\langle t-r\rangle\partial_{tt} u|+|\langle t-r\rangle\nabla \partial_t u|+|\langle t-r\rangle\nabla^2 u|
 \lesssim  |\partial \Gamma^{\le 1} u|+\langle t+r\rangle|\Box u|.
\end{align}
\end{lem}

\begin{proof}
It is clear that \eqref{2.260} holds for $|t-r|\le 1$. We then assume $|t-r|\ge 1$.

\texttt{Case 1. $r\le 2t$.} %(Use the form $\partial_i =\frac{\Omega_{i0}}t-\frac{x_i}t\partial_t$).
Recalling $\Omega_{i0}u=x_i\partial_t u+t\partial_i u$, we have for $x\in\R^2$
\begin{equation}
  \begin{aligned}\label{2.290}
&\partial_t\Omega_{i0}u= x_i\partial_{tt}u+\partial_i u+t\partial_t\partial_i u,\\
&\partial_i\Omega_{i0}u= 2\partial_{t}u+r\partial_r\partial_t u+t\Delta u.  
\end{aligned}  
\end{equation}
(Note we used Einstein summation convention here).
This implies that 
\begin{align}
&(t^2-r^2)\Delta u =r^2\Box u+r\partial_r u-2t\partial_t u-x_i\partial_t\Omega_{i0}u+t\partial_i\Omega_{i0}u\qquad(\text{by }\Box =\partial_{tt}-\Delta), \label{2.300}\\
\implies & |(t-r) \Delta u| \lesssim r|\Box u|+|\partial\Gamma^{\le 1} u|,\\
\implies &|( t-r)\partial_{tt} u| \le |( t-r)(\Box u+\Delta u)|
\lesssim  (t+r)|\Box u|+|\partial\Gamma^{\le 1} u|.\label{2.330}
\end{align}
Then we focus on $\partial_{t}\nabla u$ and $\nabla^2 u$.  From \eqref{2.290} we know that for $i=1,2$
\begin{align}
&\partial_i \partial_tu=\frac{1}{t}\partial_t \Omega_{i0}u-\frac1t\partial_i u-\frac{x_i}t\partial_{tt} u,\\
\implies &|(t-r)\partial_i \partial_tu|\lesssim|\partial_t \Omega_{i0}u|+|\partial_i u|+|(t-r)\partial_{tt} u|\lesssim |\partial\Gamma^{\le 1}u|+(t+r)|\Box u| \qquad (\text{by \eqref{2.330}}).\label{2.350}
\end{align}
By the definition of $\Omega_{i0}u=x_i\partial_{t}u+t\partial_i u$, we have for $1\le i,j\le 2$
\begin{align}
\partial_{j}\Omega_{i0}u=&\ \delta_{ij}\partial_{t}u+x_{i}\partial_{t}\partial_{j}u+t\partial_{ij} u,
\\ \implies |(t-r)\partial_{ij}u|=&\frac {(t-r)}t( \partial_{j}\Omega_{i0}u-\delta_{ij}\partial_{t}u-x_{i}\partial_{t}\partial_{j}u)\lesssim |\partial \Gamma^{\le 1}u|+|(t-r)\partial_{t}\partial_ju|
\\ \lesssim &|\partial \Gamma^{\le 1}u|+(t+r)|\Box u|\qquad(\text{by }\eqref{2.350}).
\end{align}
\texttt{Case 2. $r\ge \frac t2$}.  Recall that $\Omega=t\partial_r+r\partial_t$. Therefore we have
%\frac{\widetilde\Omega\cdot \widetilde\Omega}{r^2}$, $\widetilde\Omega=x\wedge \nabla)$.}
%Recalling $\Omega=t\partial_r+r\partial_t$, we have 
\begin{align}
\Omega \partial_{t}u=t\partial_r\partial_{t}u+r\partial_{tt}u,\qquad%=t\partial_r\partial_{t}u+r\Delta u+r\Box u\\
\Omega \partial_{r}u=t\partial_{rr}u+r\partial_t\partial_{r}u.
\end{align}
Then we have 
\begin{align}
(t^2-r^2)\partial_t\partial_r u=&\ t\Omega\partial_t u-r\Omega\partial_r u-tr(\partial_{tt}u-\partial_{rr}u)
\\=&\ t\Omega\partial_t u-r\Omega\partial_r u-tr(\Box u+\frac 1 r\partial_{r}u+\frac1{r^2}\Delta_{\mathbb{S}^1} u),
\\ \implies |(t-r)\partial_{t}\partial_r u|\lesssim& |\partial\Gamma^{\le 1} u|+(t+r)|\Box u|,
\end{align}
recalling $\Delta =\partial_{rr}+\frac 1r\partial_r+\frac{\Delta_{\mathbb{S}^1}}{r^2}$. Here $\Delta_{\mathbb{S}^1}$ is the Laplace-Beltrami operator $\partial_{\theta\theta}$.
{Therefore we get $|( t-r)\partial_{t}\nabla u|\lesssim |\partial\Gamma^{\le 1} u|+(t+r)|\Box u|$.}
On the other hand, we have 
\begin{align}
&t^2\partial_{rr}u-r^2\partial_{tt}u=t\Omega\partial_r u-r\Omega\partial_t u,
\\ \implies& (t^2-r^2)\partial_{rr}u=2r\partial_r u+\Delta_{\mathbb S^1}u+r^2\Box u+t\Omega\partial_r u-r\Omega\partial_t u,
\\ \implies&|( t-r)\partial_{rr}u|\lesssim |\partial\Gamma^{\le 1} u|+(t+r)|\Box u|. \label{2.420}
% \\ \implies& (t^2-r^2)\Delta u-\frac{2t^2}{r}\partial_{r}u-\frac{t^2}{r^2}\widetilde\Omega\cdot\widetilde\Omega u-r^2\Box u=t\Omega\partial_r u-r\Omega\partial_t u, 
% \\ \implies&\langle t-r\rangle|\Delta u|\lesssim |\partial\Gamma^{\le 1}u|+(t+r)|\Box u|,
% \\ \implies &\langle t-r\rangle (|\partial_{tt} u|+|\partial_{rr} u|\lesssim |\partial\Gamma^{\le 1} u|+(t+r)|\Box u|.
\end{align}
This implies that $\langle t-r\rangle|\Delta u|$ and $\langle t-r\rangle|\nabla^2 u|$ can be controlled by RHS of \eqref{2.420}. Furthermore we use the identity $\partial_{tt}u =\Delta u+\Box u$ and deduce that $\langle t-r\rangle|\partial_{tt} u|\lesssim |\partial\Gamma^{\le 1} u|+(t+r)|\Box u|$. Thus we complete the proof.

\end{proof}

\begin{lem}
Assume $h\in \mathcal{S}(\R^2)$. Then
\begin{align}\label{S-2}
  \|h\|_{2}\lesssim \|\Lg x\Rg\nabla h\|_{2}.
\end{align}
\end{lem}

\begin{proof}We refer the readers to the proof of a 3D version to \cite{CX22}; the 2D case is very similar therefore we omit the details.
\end{proof}

\subsection{Pointwise estimates for the wave component}
We here collect some pointwise estimates for the wave components.

Let $v(t,x)$ be a solution of the following system:
\begin{align*}
 \begin{cases}
 \Box v=F(t,x);\\
 (v,\pa_{t}v)|_{t=0}=(v_{0},v_1).
 \end{cases}
\end{align*}
It is known that the solution to such system can be written as the following mild form:
\begin{align*}
    v(t,x)=&(\cos t|\na|\ v_0)(x)+(\frac{\sin t|\na|}{|\na|}v_1)(x)+\int_{0}^{t}(\frac{\sin (t-s)|\na|}{|\na|}F)(s,x)ds\\ \coloneqq & v_{\text{hom}}+v_{\text{inh}},
   % \\=&\frac1{2\pi}\pa_{t}\int_{|x-y|<t}\frac{v_{0}(y)dy}{\sqrt{t^2-|x-y|^2}}+\frac1{2\pi}\int_{|x-y|<t}\frac{v_{1}(y)dy}{\sqrt{t^2-|x-y|^2}}
  %  \\&+\frac1{2\pi}\int_{0}^{t}\int_{|x-y|<t-s}\frac{F(s,y)dyds}{\sqrt{(t-s)^2-|x-y|^2}},
\end{align*}
where $|\na|$ is the differential operator that is corresponding to the Fourier multiplier $|\xi|$; moreover $v_{\text{hom}}=\cos t|\na|v_0+\frac{\sin t|\na|}{|\na|}v_1$ and $v_{\text{inh}}=\int_{0}^{t}(\frac{\sin (t-s)|\na|}{|\na|}F)(s,x)ds$. It is well known that $v_{\text{hom}}$ is the solution to the following homogeneous wave system
\begin{align}\label{vhom}
    \begin{cases}
    \Box v_{\text{hom}}=0,\\
(v_{\text{hom}},\pa_{t}v_{\text{hom}})|_{t=0}=(v_{0},v_1)
    \end{cases}
\end{align}
and in particular
\begin{align}\label{vhom1}
    \begin{cases}
       & (\cos t|\na| v_0)(x)=\frac1{2\pi}\pa_{t}\int_{|x-y|<t}\frac{v_{0}(y)dy}{\sqrt{t^2-|x-y|^2}}\\
       & (\frac{\sin t|\na|}{|\na|}v_1)(x)=\frac1{2\pi}\int_{|x-y|<t}\frac{v_{1}(y)dy}{\sqrt{t^2-|x-y|^2}}.
    \end{cases}
\end{align}

\begin{lem}[Homogeneous estimate] \label{lemhom}Assume $v_{\text{hom}}$ is defined above in \eqref{vhom} and assume $v_0, v_1\in \mathcal{S}(\R^2)$, then the following estimate holds: for any $t\ge 2$ and $x\in\R^2$,
 \begin{align}\label{BB1}
   | v_{\text{hom}}(t,x)|\lesssim t^{-\frac12}\Lg t-|x|\Rg^{-{\frac12}}(\|\Lg \cdot\Rg^{\frac12}\Lg\nabla\Rg v_{1}\|_{L^1}+\|\Lg \cdot\Rg^{\frac12}\Lg\nabla\Rg^{ 2}v_{0}\|_{L^1}).
 \end{align}
\end{lem}
\begin{proof}
We postpone the proof to the Appendix (see Lemma~\ref{lemBB2}).
\end{proof}

\begin{lem}[Extra homogeneous estimate.] \label{lemB2}Assume $v_{\text{hom}}$ is defined as in \eqref{vhom} above and assume $v_0, v_1\in \mathcal{S}(\R^2)$. Then for any $t\ge 2$ and $x\in\R^2$, there exists some constant $C>0$ such that the following holds: for $r=|x|$ we have
 \begin{align}\label{BB3}
   | \pa v_{\text{hom}}(t,x)|\le C \Lg t\Rg^{-\frac12}\Lg t-|x|\Rg^{-\frac32}(\|\Lg \cdot\Rg^{\frac32}\Lg\nabla\Rg^{3}v_{0}\|_{L^1})+\|\Lg \cdot\Rg^{\frac32}\Lg\nabla\Rg^{2} v_{1}\|_{L^1}).
 \end{align}
\end{lem}
\begin{proof}
    We first note that by a simple observation:
    \begin{align}
    (t^2-r^2)(\pa_t,\pa_1,\pa_2)=(tL_0-x_i\Omega_{i0}, -x_1L_0+t \Omega_{10}+x_{2}\pa_{\theta}, -x_2L_0+t \Omega_{20}-x_{1}\pa_{\theta}). \label{eq:pagaid}
     \end{align}
% &(t^2-r^2)\pa_t=tL_0-x_i\Omega_{i0};\
% (t^2-r^2)\pa_1= -x_1L_0+t \Omega_{10}+x_{2}\pa_{\theta};\ \
% (t^2-r^2)\pa_2= -x_2L_0+t \Omega_{20}-x_{1}\pa_{\theta};
It then follows from \eqref{eq:pagaid} that for $t>0$
 \begin{align}\label{DE-1}
       &\Lg t-|x| \Rg|\pa   v_{\text{hom}}|\lesssim  |\widetilde{\Ga} v_{\text{hom}} (t,x)|, 
\end{align}
where $\widetilde{\Ga}$ is defined as in \eqref{def-tdGamma}. Note that
\begin{align}\label{2.47}
    [\Ga,\Box]=0\ ,\quad [L_0,\Box]=-2\Box,
\end{align}
where $[\cdot,\cdot]$ is the usual commutator, then in effect $\widetilde{\Ga}v_{\text{hom}}$ is again a solution to the system \eqref{vhom} with certain initial data. Then combining \eqref{BB1}, \eqref{DE-1} and the initial condition estimate (cf.\cite{CX22}). We thus conclude \eqref{BB3}.
\end{proof}

%\begin{lem}[Inhomogeneous estimate]\label{cor:gradinh}
%Suppose $v_{\text{inh}}$ is defined as above in \eqref{vinh}. 
%where with no loss we assume $x=(r,0)$ and $y=\la(\cos \psi, \sin \psi)$,
%Assume $t\ge2$ and $0<\mu<\frac12$, it then follows that 
%\begin{align}
% \Lg |x|\Rg^{\frac12} \Lg |x|-t\Rg|(\na v)_{\mbox{inh}}(t,x)|\lesssim\sup_{\{(s,y):s\in(0,t),|x-y|<t-s,|y|\ge\frac{s}{2}\}}\{|y|^{\frac12}\Lg |y|+s\Rg^{1+\mu}\Lg |y|-s\Rg(|\na^{\le 1}F(s,y)|+|\pa_\theta F(s,y)|)\}\\+\sup_{\{(s,y):s\in(0,t),|x-y|<t-s,|y|<\frac{s}{2}\}}\{|y|^{\frac12}\Lg |y|+s\Rg^{1+\mu}\Lg |y|\Rg(|\na^{\le 1}F(s,y)|+|\pa_\theta F(s,y)|)\}. 
% \end{align}
%\end{lem}

%Let $\varphi=\frac{r}{\Lg r\Rg}$. Then
%\begin{align*}
%&0\leq \int (\Lg r\Rg \pa_{r}h+\varphi h)^2\ dx=\int \left(\Lg r\Rg^2(\pa_{r}h)^2+\Lg r\Rg\varphi\pa_{r}(h^2)+\varphi^2h^2\right) dx,
%\\ \implies&\int \Lg r\Rg^2 |\pa_rh|^2\ dx \geq \int \left(\frac{1}{r^2}\pa_{r}(r^2\Lg r\Rg \varphi)-\varphi^2\right)h^2 \ dx\geq \int (3-(\frac{r}{\Lg r\Rg})^2)h^2\ dx\gtrsim \int |h|^2\ dx.
%\end{align*}

%\begin{proof}
%We leave the proof in the Appendix.
%\end{proof}

\subsection{Pointwise estimates for the Klein-Gordon component}
We here collect some previously known pointwise estimates for the Klein-Gordon components, one can refer to \cite{G1992} (also see \cite{D,DM21}).
Let $\{p_j\}_{0}^\infty$ be the usual Littlewood-Paley partition of the unity
\[\sum_{j\ge 0} p_j(s)=1,\qquad s\ge 0,\]
which satisfies 
\begin{align}
0\le p_j\le 1,\qquad p_j\in C_0^\infty(\R)\qquad \text{for all }j\ge 0, 
\end{align}
and 
\begin{align}
\supp\ p_0\subset(-\infty, 2],\qquad \supp\ p_j\subset[2^{j-1},2^{j+1}]\qquad \text{for all }j\ge1.
\end{align}
\begin{lem}[Georgiev]\label{kgt32}
Let $w$ be the solution of the Klein-Gordon equation
\begin{align}
\begin{cases}
\Box w+w=f,\\
 (w,\partial_t w)|_{t=0} =(w_0,w_1),
\end{cases}
\end{align}
with $f=f(t,x)$ a sufficiently nice function. Then for all $t\ge 0$, it holds
\begin{align}
\langle t+|x|\rangle|w(t,x)|\lesssim \sum_{j\ge 0}(\sup_{0< s\le t}p_j(s)\|\langle s+|\cdot|\rangle\Gamma^{\le 4}f(s,\cdot)\|_{2}+\|\langle |\cdot|\rangle p_j(|\cdot|)\Gamma^{\le 5}w(0,\cdot)\|_{2}).
\end{align}

\end{lem}
In this article we shall use the following variation as in \cite{D,DM21}.
\begin{lem}[Dong et al.]
With the same settings as in Lemma \ref{kgt32}, letting $\delta'>0$ and assuming
\begin{align} 
 \sum_{|I|\le 4}\|\Lg s+|\cdot|\Rg \Ga^I f(s,\cdot)\|_2\le C_f\Lg s\Rg^{-\delta'},
\end{align}
then for all $t\ge 0$ we have
\begin{align}\label{d-1}
\Lg t+|x|\Rg |w(t,x)|\lesssim \frac{C_f}{1-2^{-\delta'}}+\sum_{|I|\le 5}\|\Lg |\cdot|\Rg \log \Lg \cdot\Rg\Ga^I w(0,\cdot) \|_2.
%\langle t+|x|\rangle|w(t,x)|\lesssim \sup_{0< s\le t}\|\langle s+|\cdot|\rangle^{1+\delta}\Gamma^{\le 4}f(s,\cdot)\|_{L^2} +\|{\langle |\cdot|\rangle^{\frac32+\delta }}\Gamma^{\le 4}w(0,\cdot)\|_{L^2}.
\end{align}
%Here $0<\delta\ll1$.

\end{lem}

\begin{lem}\label{lem2.7}
Suppose $u=u(t,x)$ is a smooth solution to $\Box u+u=F$. Then for $t>0$ we have 
\begin{align}\label{kgw}
\Big|\frac{\langle t+r\rangle}{\langle t-r\rangle}u\Big|
\lesssim&|\partial\Gamma^{\le 1}u|+\Big|\frac{\langle t+r\rangle}{\langle t-r\rangle}F\Big|.
\end{align}
\end{lem}

\begin{proof}
If $r< \frac t2$ or $r> 2t$, we have $\langle t+r\rangle \sim\langle t-r\rangle$. This implies that 
\[\Big|\frac{\langle t+r\rangle}{\langle t-r\rangle}u\Big|
\lesssim| u|\lesssim |\Box u|+|F|
\lesssim |\partial\Gamma^{\le 1}u| +|F|.\]
If $\frac t2\le r\le 2t$, we recall that \eqref{2.300} gives that
\begin{align}
(t^2-r^2)\Delta u =&r^2\Box u+r\partial_r u-2t\partial_t u-x_i\partial_t\Omega_{i0}u+t\partial_i\Omega_{i0}u
\\ =& -r^2  u+r^2F+r\partial_r u-2t\partial_t u-x_i\partial_t\Omega_{i0}u+t\partial_i\Omega_{i0}u\qquad (\text{by }\Box u+u=F)
\\ \implies u=&-\frac{t^2-r^2}{r^2}\Delta u+F+\frac1r\partial_r u-\frac {2t}{r^2}\partial_t u-\frac1{r^2}(x_i\partial_t\Omega_{i0}u-t\partial_i\Omega_{i0}u).
% \\ \implies \|\frac{\langle t+r\rangle}{\langle t-r\rangle}u\|_{2} &
% \lesssim\|\Delta u\|_{2}+\|\frac{\langle t+r\rangle}{\langle t-r\rangle}F\|_2+\|\partial\Gamma^{\le 1}u\|_{2}.
\end{align}
This implies \eqref{kgw} holds for $\frac t2\le r\le 2t$. 
\end{proof}
\begin{rem}\label{rem2.12}
  Note that one can iterate \eqref{kgw} as long as the nonlinear term $F$ is sufficiently nice. A typical example is $F=u^2$ then $F$ can easily absorb the weight $\frac{\Lg t+r\Rg}{\Lg t-r\Rg}$ and eventually we can derive the decay of $|u|\lesssim \Lg t+r\Rg^{-M}\Lg t-r\Rg^M$ for arbitrarily large $M$. Such property of $u$ can be considered as a machinery that one can gain any decay in $\Lg t+r\Rg$ with a trade-off of losing $\Lg t-r\Rg$ decay. 
\end{rem}

\begin{lem}\label{lem2.13}
    Suppose $u=u(t,x)$ is a smooth solution to $\Box u+u=F$. Then for $t\ge 2$ and $r\le 3t$ we have
\begin{align}\label{2.62}
    |u|\ls \frac{\Lg r-t\Rg }{ t}|\pa^2 u|+ \frac{1}{t}|\pa\Ga^{\le 1}u|+|F|.
\end{align}
\end{lem}

\begin{proof}
    The proof is very similar to Lemma~\ref{lem2.6}. Recall \eqref{2.300} that
    \begin{align}
        &(t^2-r^2)\Delta u =r^2\Box u+r\partial_r u-2t\partial_t u-x_i\partial_t\Omega_{i0}u+t\partial_i\Omega_{i0}u\\ \implies &(t^2-r^2)\pa_{tt} u+(r^2-t^2)\Box u=r^2\Box u+r\partial_r u-2t\partial_t u-x_i\partial_t\Omega_{i0}u+t\partial_i\Omega_{i0}u \\ \implies & t^2 u= (r^2-t^2)\pa_{tt} u+t^2F+r^2\Box u+r\partial_r u-2t\partial_t u-x_i\partial_t\Omega_{i0}u+t\partial_i\Omega_{i0}u .
    \end{align}
We thus conclude \eqref{2.62} from here.
    
\end{proof}

\section{Estimates of the energy}\label{sec:energy}
To prove Theorem~\ref{thm},
we first make an {\it a priori} hypothesis: for any $K \ge 12$ we assume
\begin{align}
&|\pa\Gamma^{\le K-5} n_0|+|\pa^2\Ga^{\le K-5}\td{n}|\le C_1\varepsilon{\langle t\rangle^{-\frac12}}\langle r-t\rangle^{-\frac{1+\de}{2}} ,\label{A2}
\\&|\Gamma^{\le K-5} \phi|\le C_2\varepsilon \langle t+r\rangle^{-1},\label{A3}
\end{align}
where the positive constants $C_1,\ C_2$ will be chosen later and $\de>0$ is a small fixed parameter as in Theorem~\ref{thm}. 

\begin{rem}
    It is worth mentioning that here our {\it a priori} assumption on wave field $n$ is only of the form $\Lg t\Rg^{-\frac12}\Lg r-t\Rg^{-\frac{1+\de}{2}}$ yet in fact we will show such decay can be the usual $\Lg t+r\Rg^{-\frac12}\Lg r-t\Rg^{-\frac{1+\de}{2}}$ to close the bootstrap scheme (see Section~\ref{sec:close}).
\end{rem}
As mentioned earlier, we consider the initial conditions on $\phi$ and $n$: 
\begin{align}\label{3.333}
\begin{cases}
(\phi,\partial_t \phi)|_{t=0}=(\Phi_0,\Phi_1),\\
(n,\partial_t n) |_{t=0}=( N_0,N_1).
\end{cases}
\end{align}
Recall that we separate $n=n_0+A^\al \pa_\al\td{n}$, where
\begin{align}\label{n0}
\begin{cases}
    &\Box n_0=0,\\
    &(n_0,\pa_t n_0)|_{t=0}=(N_0,N_1),
\end{cases}
\end{align}
and
\begin{align}\label{tdn}
    \begin{cases}
    & \Box \td{n}=|\phi|^2,\\
        &(\td{n},\pa_t\td{n})|_{t=0}=(0,0).
    \end{cases}
\end{align}
To begin with, we first show the top-order energy is ``almost'' bounded that is it grows polynomially as $\Lg t\Rg^{\delta_1}$ where the power $\delta_1$ can be made arbitrarily small.
\begin{lem}[Polynomial growth of the top-order energy]\label{lem3.1}
Let $(\phi,n)$ be a couple of solutions to \eqref{eq:kg}. Assume \eqref{A2} and \eqref{A3} hold. %Let \cmtr{$K\ge k+5$}.
Then we have for some constant $C>0$,
\begin{align}\label{3.32}
\| \partial \Gamma^{\le K}\phi\|_{2}+\| \Gamma^{\le K}\phi\|_{2}+\|\pa\Gamma^{\le K}n\|_{2}\le C\varepsilon t^{\de_1}
\end{align}
where $\de_1$ satisfying $0<\de_1<\frac{\de}{4}$ is an arbitrarily small constant. In particular, we have
\begin{align}\label{3.44e}
        \|\pa \Ga^{\le K} \td{n}\|_2 \le C\varepsilon^2 \Lg t\Rg^{\de_1}.
    \end{align}
%and 
%\begin{align}\label{3.34}
%\|\pa\Gamma^{\le k+4} \td{n}\|_{2} \le C\varepsilon.
%\end{align}

\end{lem}

\begin{proof} To begin with we first consider the wave component $\|\pa\Ga^{\le K}n\|_2$.

\texttt{Step 1, the wave component $n$.}

Note that since $n=n_0+A^\al\pa_\al \td{n}$, we have for any multi-index $I$,
\begin{align}\label{3.31}
|\Ga^I n|\lesssim |\Ga^In_0|+|\Gamma^{I}\pa \td{n}| \lesssim |\Ga^I n_0|+|\partial\Gamma^{\le |I|} \td{n}|.
\end{align}
As a result for $|I|\le K$, 
\begin{align}
    \|\pa\Ga^{\le K}n\|_2\lesssim \|\pa \Ga^{\le K}n_0\|_2+ \|\pa^2\Ga^{\le K}\td{n}\|_2.
\end{align}
Note that by a simple energy estimate \begin{align}\label{3.9}
\|\pa \Ga^{\le K}n_0\|_2(t)=\|\pa \Ga^{\le K}n_0\|_2(0)\le C_i\varepsilon,
\end{align}where $C_i>0$ is some fixed constant resulting from the initial condition. Therefore it suffices to estimate $\| \pa^2\Ga^{\le K}\td{n}\|_2$. Let $|I|=m$ and $m\le K$ is a running index, we have from \eqref{tdn} and \eqref{n0} that
\begin{align}\label{3.17}
\Box \Gamma^I \td{n}=\sum_{I_1+I_2=I }C_{I;I_1,I_2}\Gamma^{I_1}\phi\Gamma^{I_2}\phi.
\end{align}
%Multiplying both sides of \eqref{3.17} by $\partial_t\Gamma^I\td{n}$, we obtain
%\begin{align}
%\frac{d}{dt}\|\partial\Gamma^I \td{n}\|_{2}^2
%\lesssim& \int|\Gamma^{I_1}\phi\Gamma^{I_2}\phi \partial_t\Gamma^I \td{n}|
% \\ \lesssim &(\|\Gamma^{\le [\frac{K-1}2]}\phi\|_{\infty}\|\Gamma^{\le K}\phi\|_{2}
%              +\|\Gamma^{\le K}\phi\|_{2}\|\Gamma^{\le [\frac{K}2]}\phi\|_{\infty})
%              \| \partial_t\Gamma^I \td{n}\|_{2}
% \\ \lesssim &  \|\Gamma^{\le [\frac{K-1}2]+1}\phi\|_{\infty} (\|\nabla\Gamma^{\le K}\phi\|_{2}+\|\Gamma^{\le K}\phi\|_{2})  \| \partial_t\Gamma^I \td{n}\|_{2}         
% \\ \le & \ C\varepsilon t^{-1}(\|\nabla\Gamma^{\le K}\phi\|_{2}+\|\Gamma^{\le K}\phi\|_{2})\| \partial_t\Gamma^I \td{n}\|_{2} \qquad\qquad (\text{by \eqref{A3}}).
%\end{align}
%Note that here we need \cmt{$[\frac{K-1}{2}]+1\le k\le K-4 \implies K\ge 8$}. 
Moreover, we have 
\begin{align}\label{3.17a}
\Box \partial\Gamma^I \td{n}=\sum_{I_1+I_2=I }\td{C}_{I;I_1,I_2}\pa \Gamma^{I_1}\phi\Gamma^{I_2}\phi.
\end{align}
Multiplying both sides of \eqref{3.17a} by $\partial_t\pa \Gamma^I\td{n}$, we obtain
\begin{equation}\label{3.13}
\begin{aligned}
\frac{d}{dt}\|\partial^2\Gamma^I \td{n}\|_{2}^2
\lesssim& \int|\pa \Gamma^{I_1}\phi\Gamma^{I_2}\phi \partial_t\pa\Gamma^I \td{n}|
 \\ \lesssim &(\|\pa\Gamma^{\le [\frac{K-1}2]}\phi\|_{\infty}\|\Gamma^{\le K}\phi\|_{2}+\|\pa \Gamma^{\le K}\phi\|_{2}\|\Gamma^{\le[\frac{K}2]}\phi\|_{\infty}) \| \partial_t\pa\Gamma^I \td{n}\|_{2}
 \\ \lesssim &  \|\Gamma^{\le [\frac{K-1}2]+1}\phi\|_{\infty} (\|\pa\Gamma^{\le K}\phi\|_{2}+\|\Gamma^{\le K}\phi\|_{2})  \| \partial_t\pa\Gamma^I \td{n}\|_{2}         
 \\ \le & \ C\varepsilon t^{-1}(\|\pa\Gamma^{\le K}\phi\|_{2}+\|\Gamma^{\le K}\phi\|_{2})\| \partial_t\pa\Gamma^I \td{n}\|_{2} \qquad\qquad (\text{by \eqref{A3}}).
\end{aligned}
\end{equation}
Note that here we indeed have {$[\frac{K-1}{2}]+1\le K-5$ since $K \ge 12$}. This implies that 
\begin{align}\label{3.27}
\frac{d}{dt}\|\partial^2\Gamma^I \td{n}\|_{2}^2
\le  \ C\varepsilon t^{-1}(\|\pa\Gamma^{\le K}\phi\|_{2}^2+\|\Gamma^{\le K}\phi\|_{2}^2+\| \partial^2\Gamma^{\le K}\td{n}\|_{2}^2).
 \end{align}

\texttt{Step 2, we estimate the Klein-Gordon component $\|\Ga^{\le K}\phi\|_2, \|\pa\Ga^{\le K}\phi\|_2$.}

Let $I$ be a multi-index and $|I|\le K$. 
Write $\Phi= \Gamma^{I} \phi$. We have
\begin{align}\label{3.0A}
\square \Phi+\Phi=\sum_{I_1+I_2=I}C^\al_{I; I_1,I_2}\pa_\al \Gamma^{I_1}n \Gamma^{I_2}\phi.
\end{align}
Let $p(t,r) = q(r-t)$, where
\begin{align}\label{q}
q(s) = \int_{-\infty}^s \langle \tau\rangle^{-1}  \bigl(\log ( 2+\tau^2) \bigr)^{-2} d\tau.
\end{align}
Clearly
\begin{align}
-\partial_t p = \partial_r p = q^{\prime}(r-t)
= \langle r -t \rangle^{-1} \bigl( \log (2+(r-t)^2)  \bigr)^{-2}.
\end{align}
Multiplying both sides of \eqref{3.0A} by $e^p \partial_t \Phi$, we obtain
\begin{align}
   \text{LHS}&=\frac12\frac{d}{dt}(\| e^\frac{p}{2}\partial \Phi\|_{2}^2+\| e^\frac{p}{2} \Phi\|_{2}^2) +\frac 12
   \int e^p q^{\prime} (|T \Phi|^2+|\Phi|^2),
\end{align}
where $|\partial \Phi|^2=\sum_{i=0}^2|\pa_i \Phi|^2$ and $|T\Phi|^2=\sum_{i=1}^2|T_i \Phi|^2$. We split the RHS into two cases:
\begin{align}
\text{RHS} &= \sum_{I_1+I_2=I}C^\al_{I; I_1,I_2} \int e^p\pa_\al\Gamma^{I_1}n \Gamma^{I_2}\phi \partial_t \Phi.
\end{align}

\texttt{Case 1:} $|I_1|<|I_2|\le |I|$.
\begin{align}
\text{RHS}\le &\frac{1}{10}\int e^p q' |\Gamma^{I_2}\phi|^2+C\int \frac{e^p}{q'}|\pa \Gamma^{\le [\frac{|I|}2]}n|^2|\partial_t \Phi|^2\\
\le& \frac{1}{10}\int e^p q' |\Gamma^{I_2}\phi|^2+C\|\langle r-t\rangle^{\frac{1+\delta}2}\pa \Gamma^{\le [\frac{|I|}2]}n\|_{\infty}^2\|\partial_t\Phi\|_{2}^2
\\\le& \frac{1}{10} \int e^p q' |\Gamma^{I_2}\phi|^2+(C\varepsilon)^2t^{-1}\|\partial_t\Phi\|_{2}^2\qquad (\text{by \eqref{A2}}).
\end{align}
Here the positive constant $C$ may vary from line to line. We also notice that after summing the index $I$ the term $\frac{1}{10} \int e^p q' |\Gamma^{I_2}\phi|^2$ can be absorbed by the LHS (by choosing $K\ge 12$ we indeed have % {\color{red} (Note we need $[\frac{|\alpha|}2]\le[\frac{K}{2}]\le k\le K-4\implies K\ge 7$ when we use \eqref{A2}.)}
{$[\frac{K}{2}]+1\le K-5$}).

\texttt{Case 2:} $|I|\ge|I_1|\ge |I_2|$.
\begin{align}
\text{RHS}\le &C\| \pa\Gamma^{\le|I|}n\|_{2}\|\Gamma^{\le [\frac{|I|}2]}\Phi\|_{\infty}\|\partial_t\Phi\|_{2}
\\ \le &C\varepsilon t^{-1}(\| \pa\Gamma^{\le|I|}\pa \td{n}\|_{2}+\|\pa \Ga^{\le |I|}n_0\|_2)\|\partial_t\Phi\|_{2}\qquad (\text{by \eqref{A3} and $n=n_0+A^\al \partial_\al \td{n}$})
\\ \le &C\varepsilon t^{-1}\| \partial^2\Gamma^{\le |I|} \td{n}\|_{2}\|\partial_t\Phi\|_{2}+C\varepsilon^2t^{-1}\|\pa_t\Phi\|_2.
\end{align}
% {\color{red}Here we need $\|\Gamma^{\le K}n^0\|_{L^2}\lesssim \sum_{j=0}^{K}\|\langle x\rangle^{j}\nabla^j \nabla n_0\|_{L^2}+\sum_{j=0}^{K-1}\|\langle x\rangle^{j+1}\nabla^j n_1\|_{L^2}\lesssim \varepsilon$.}
Then we arrive at 
\begin{equation}\label{3.29}
\begin{split}
&\frac{d}{dt}(\| \partial \Phi\|_{2}^2+\| \Phi\|_{2}^2) +\frac 12
   \int e^p q^{\prime} (|T \Phi|^2+|\Phi|^2)\\\le&\  (C\varepsilon)^2t^{-1}(\| \partial\Phi\|_{2}^2+\|\pa \Phi\|_2)+C\varepsilon t^{-1}\| \partial^2 \Gamma^{\le |I|} \td{n}\|_{2}\|\partial_t\Phi\|_{2}.
\end{split}
\end{equation}
Combining \eqref{3.27} and \eqref{3.29}, it follows that 
\begin{align}
&\frac{d}{dt}(\| \partial \Gamma^{\le K}\phi\|_{2}^2+\| \Gamma^{\le K}\phi\|_{2}^2+\|\partial^2\Gamma^{\le K} \td{n}\|_{2}^2)
\\\le&\  (C\varepsilon)^2t^{-1}(\| \partial \Phi\|_{2}^2+\|\pa \Phi\|_2)+ \ C\varepsilon t^{-1}(\|\pa\Gamma^{\le K}\phi\|_{2}^2+\|\Gamma^{\le K}\phi\|_{2}^2+\| \partial^2\Gamma^{\le K} \td{n}\|_{2}^2).
\end{align}
This immediately implies that 
 % Thus it follows from \eqref{3.21} and \eqref{3.13} that
\begin{align}
\| \partial \Gamma^{\le K}\phi\|_{2}+\| \Gamma^{\le K}\phi\|_{2}+\|\partial^2\Gamma^{\le K} \td{n}\|_{2}
\le C\varepsilon t^{\de_1}.
\end{align}
For the term $\|\pa\Ga^{\le K} \td{n}\|_2$:
let $J$ be a multi-index and $|J|=m\le K$. Then we have
\begin{align}\label{3.35s}
\Box \Gamma^J \td{n}=\sum_{J_1+J_2= J}C_{J;J_1,J_2}\Gamma^{J_1}\phi\Gamma^{J_2}\phi.
\end{align}
 Multiplying both sides of \eqref{3.35s} by $\partial_t\Gamma^J \td{n}$, we obtain
 \begin{equation}\label{3.30}
\begin{aligned}
\frac{d}{dt}\|\partial\Gamma^J \td{n}\|_{2}^2
\lesssim& \sum_{J_1+J_2= J}\int|\Gamma^{J_1}\phi\Gamma^{J_2}\phi \partial_t\Gamma^J \td{n}|
 \\ \lesssim &\|\Gamma^{\le [\frac{K}2]}\phi\|_{\infty}\|\Gamma^{\le K}\phi\|_{2}
              \| \partial_t\Gamma^J \td{n}\|_{2}
\\ \le& C\varepsilon^2 \Lg t\Rg^{-1+\de_1}\|\pa\Ga^J \td{n}\|_2 & (\text{by \eqref{A3} and \eqref{3.32}}),
\end{aligned} 
\end{equation}
which shows %(\cmtr{need $k_2+1\le K$ and $[\frac{k_2+1}{2}]\le k $ }  )
\begin{align}\label{3.44d}
 \|\pa \Ga^{\le K}\td{n}\|_2\le C\varepsilon\Lg t\Rg^{\de_1}  
\end{align}
by taking the initial data estimate into consideration. Here we need the following smallness assumption (also see \eqref{B1},\eqref{B2} and \eqref{B1a}):
\begin{align}\label{3.35a}
    \| (\partial \Gamma^{\le K}\phi)|_{t=0}\|_{2}+
\| (\Gamma^{\le K}\phi)|_{t=0}\|_{2}+\|(\pa \Ga^{\le K}n_0)|_{t=0} \|_2+  \|(\partial^2\Gamma^{\le K} \td{n})|_{t=0}\|_{2}\le C_i\varepsilon,    
\end{align}
where the positive constant $C_i$ is clear from the initial condition {(cf. \eqref{B1}-\eqref{B4})}. Thus \eqref{3.32} holds.

\end{proof}

\begin{rem}\label{rem3.3}
 In fact from the assumptions of the initial data we indeed have 
 \begin{align}\label{3.32b}
      \|\pa\Ga^{\le K}\td{n}\|_2+\|\pa^2\Ga^{\le K}\td{n}\|_2\le C\varepsilon^2 t^{\de_1}.
 \end{align}  
 To see this, it suffices to show for $|\al|\le K+1$,
 \begin{align}
     \|(\pa \Ga^{\al}\td{n})|_{t=0}\|_2\le C\varepsilon^2.
 \end{align}
Since interchanging the vector fields merely generates lower-order terms, we estimate the initial data following similar arguments in \cite{CX22}: 
\begin{align}
\|(\partial\Gamma^{\alpha}\td{n})|_{t=0}\|_{2}
\lesssim\|(\Gamma_{\Omega}^{\alpha_1}\partial_{t}^{\alpha_2}\nabla^{\alpha_3}\td{n})|_{t=0}\|_{2}
\lesssim \sum_{\alpha_1+\alpha_2+\alpha_3\le K+1}\sum_{J=0}^{\alpha_1}\|\langle x\rangle^{J}(\partial^J\partial_{t}^{\alpha_2}\nabla^{\alpha_3}\td{n})|_{t=0}\|_2,
\end{align}
where $\Gamma_{\Omega}=\Omega_{i0},\Omega_{ij}$. Recalling the equation of $\td{n}$ in \eqref{tdn}, we have for any $b\ge2$ and $a\ge0$
\[\partial_t^b\nabla^a \td{n}=\partial_t^{b-2}\nabla^a \Delta \td{n}+\partial_t^{b-2}\nabla^a(|\phi|^2).
\]
Since $(\td{n},\partial_t \td{n})|_{t=0}=\mathbf 0$, one has (by Proposition \ref{prop})
{\small\begin{align}
\sum_{J=0}^{\alpha_1}\|\langle x\rangle^{J}(\partial^J\partial_{t}^{\alpha_2}\nabla^{\alpha_3}\td{n})|_{t=0}\|_2
\lesssim \|\langle x\rangle^{[\frac{K+1}2]+1}\langle \partial \rangle^{ K-1}\phi|_{t=0}\|_2\|\langle x\rangle^{[\frac{K+1}2]}\langle \partial \rangle^{ [\frac{K+1}{2}]}\phi|_{t=0}\|_\infty\le C\varepsilon^2.\label{3.44}
\end{align}}
The cases $b=0,1$ follow immediately. {Hence \eqref{3.34} holds combining the estimates \eqref{3.13}, \eqref{3.30} and \eqref{3.44}. In other words even though the initial data $(\Ga^\al \td{n},\pa_t \Ga^{\al }\td{n})|_{t=0}$ are not zero, it is of the order $O(\varepsilon^2)$ thanks to the nonlinear structure. Therefore one can treat $\td{n}$ as a $O(\varepsilon^2)$ order perturbation of $n_0$. The initial condition of the Klein-Gordon component can be handled similarly and we postpone the analysis in Proposition~\ref{prop}.}
\end{rem}
From now on we fix the choice $\de_1>0$ as in Lemma~\ref{lem3.1}.

\begin{lem}[Polynomial growth of an almost top-order weighted energy of the wave component]\label{lem3.7}
     Let $(\phi,n)$ be a couple of solutions to \eqref{eq:kg}. Assume \eqref{A2} and \eqref{A3} hold. With no loss we assume $t\ge 2$. Then there exists some constant $C>0$ such that
\begin{align}\label{3.59a}
    \|\Lg t-|\cdot|\Rg \pa \Ga^{\le K-1} n \|_2\le (C_i\varepsilon +C\varepsilon^2) \Lg t\Rg^{\de_1}.
\end{align}
\end{lem}
\begin{proof}
    The proof of \eqref{3.59a} consists two parts. Note that $\|\Lg t-|\cdot|\Rg \pa \Ga^{\le K-1} n \|_2\ls  \|\Lg t-|\cdot|\Rg \pa \Ga^{\le K-1} n_0 \|_2+ \|\Lg t-|\cdot|\Rg \pa^2 \Ga^{\le K-1} \td{n} \|_2$. We first consider the $n_0$ part. Notice that for any multi-index $J$ with $|J|\le K-1$, we have
    \begin{align}\label{3.60a}
        \Box \Ga^J n_0=0.
    \end{align}
We can then apply $\wtd{\Ga}$ to \eqref{3.60a} where $\widetilde{\Ga}$ is defined as in \eqref{def-tdGamma}. Recalling \eqref{2.47}, we have
   \begin{align}\label{3.60b}
        \Box \wtd{\Ga}\Ga^J n_0=0.
    \end{align}
It thereby follows that
\begin{align}
    \|\Lg t-|\cdot|\Rg\pa\Ga^J n_0\|_2\le& C\|\wtd{\Ga}\Ga^J n_0\|_2 &\text{( by \eqref{DE-1})}\\
    \le& C_i\varepsilon \log t & \text{( by \eqref{A.1} and \eqref{3.60b})},
\end{align}
{where the initial condition $\| (\wtd{\Ga} \Ga^{K-1}n_0)|_{t=0}\|_2\le C_i\varepsilon$ by \eqref{B2a}}.
For the term $\|\Lg t-|\cdot|\Rg \pa^2 \Ga^{J} \td{n} \|_2$, we observe that
\begin{align}
    \|\Lg t-|\cdot|\Rg \pa^2 \Ga^{\le K-1} \td{n} \|_2\ls \| \pa \Ga^{\le K} \td{n} \|_2+\|\Lg t+|\cdot|\Rg \Ga^{\le K-1}\phi \Ga^{\le [\frac{K-1}{2}]}\phi\|_2\le C_i\varepsilon+ C\varepsilon^2\Lg t\Rg^{\de_1},
\end{align}
which is true by \eqref{A3}, \eqref{3.32} and \eqref{3.44e}.
    
\end{proof}

\begin{lem}[Polynomial growth of an almost top-order weighted energy of the Klein-Gordon component]\label{lem3.5}
    Let $(\phi,n)$ be a couple of solutions to \eqref{eq:kg}. Assume \eqref{A2} and \eqref{A3} hold. Then there exists some constant $C>0$ such that
\begin{align}\label{3.322}
    \left\|\frac{\Lg \rho+t\Rg}{\Lg \rho-t\Rg} \Ga^{\le K-1}\phi \right\|_2\le C\varepsilon \Lg t\Rg^{\de_1}.
\end{align}
\end{lem}
\begin{proof}
    The proof consists two areas. 
    \begin{align}
         \left\|\frac{\Lg \rho+t\Rg}{\Lg \rho-t\Rg} \Ga^{\le K-1}\phi \right\|_2\ls  \left\|\chi^{\frac12}(\rho-2t)\frac{\Lg \rho+t\Rg}{\Lg \rho-t\Rg} \Ga^{\le K-1}\phi \right\|_2+ \left\|(1-\chi^{\frac12}(\rho-2t))  \frac{\Lg \rho+t\Rg}{\Lg \rho-t\Rg} \Ga^{\le K-1}\phi \right\|_2,
    \end{align}
   where the cut-off function $\chi(s)$ is defined in \eqref{chi}. For the region $\rho\ge 2t$, it is clear that $\Lg \rho+t\Rg \sim \Lg \rho-t\Rg$ and therefore 
\begin{align}
    \left\|\chi^{\frac12}(\rho-2t)\frac{\Lg \rho+t\Rg}{\Lg \rho-t\Rg} \Ga^{\le K-1}\phi \right\|_2\ls \|\Ga^{\le K-1}\phi\|_2\ls C\varepsilon t^{\de_1},\quad \text{ by \eqref{3.32}}.
\end{align}
        It then suffices to estimate the second term. In fact it can be observed that when $t\ge 2$, we automatically have $\rho\le 2t+2\le 3t$. As a result, we apply Lemma~\ref{lem2.13} to derive that
        \begin{equation}\label{3.37a}
             \begin{aligned}
            &\left\|(1-\chi^{\frac12}(\rho-2t))  \frac{\Lg \rho+t\Rg}{\Lg \rho-t\Rg} \Ga^{\le K-1}\phi \right\|_2\\ \ls& \|\pa^2\Ga^{\le K-1}\phi\|_2+\|\pa \Ga^{\le K}\phi\|_2 + \sum_{|J_1|+|J_2|\le K-1}\left\|(1-\chi^{\frac12}(\rho-2t))\frac{\Lg \rho+t\Rg}{\Lg \rho-t\Rg} \pa \Ga^{J_1} n\Ga^{J_2}\phi\right\|_2,\quad \text{( by \eqref{2.62})}.
         \end{aligned}
        \end{equation}
    Note that by \eqref{A2}-\eqref{A3}, we have 
    \begin{equation}\label{3.37b}
        \begin{aligned}
       &\sum_{|J_1|+|J_2|\le K-1}\left\|(1-\chi^{\frac12}(\rho-2t))\frac{\Lg \rho+t\Rg}{\Lg \rho-t\Rg} \pa \Ga^{J_1} n\Ga^{J_2}\phi\right\|_2\\ \ls&
       \left \|(1-\chi^{\frac12}(\rho-2t))\frac{\Lg \rho+t\Rg}{\Lg \rho-t\Rg} \pa \Ga^{\le [\frac{K-1}{2}]} n\Ga^{\le K-1}\phi \right\|_2+\left\|(1-\chi^{\frac12}(\rho-2t))\frac{\Lg \rho+t\Rg}{\Lg \rho-t\Rg} \pa \Ga^{\le K-1} n\Ga^{\le [\frac{K-1}{2}]}\phi\right\|_2\\
       \ls& \varepsilon\Lg t\Rg^{-\frac12}\left\|(1-\chi^{\frac12}(r-2t))  \frac{\Lg \rho+t\Rg}{\Lg \rho-t\Rg} \Ga^{\le K-1}\phi \right\|_2+\|\pa \Ga^{\le K-1}n\|_2.
   \end{aligned}
    \end{equation}
   Combining the estimates \eqref{3.37a} and \eqref{3.37b} we then conclude \eqref{3.322} using the smallness assumption.

\end{proof}

\begin{lem}[Polynomial growth of a lower-order weighted energy slightly away from the light cone]\label{lem3.1a}
Let $(\phi,n)$ be a couple of solutions to \eqref{eq:kg}. Assume \eqref{A2} and \eqref{A3} hold. Then we have for some constant $C>0$,
\begin{align}\label{3.33a}
\| \chi^{\frac12}(|\cdot|-t)\Lg |\cdot|-t\Rg \partial \Gamma^{\le K-5}\phi\|_{2}+\| \chi^{\frac12}(|\cdot|-t)\Lg |\cdot|-t\Rg \Gamma^{\le K-5}\phi\|_{2}\le C_i\varepsilon+C\varepsilon^2 t^{\de_1},
\end{align}
where $C_i>0$ is exact from the choice of the initial data.

\end{lem}

\begin{proof}
%\texttt{Step 1, the wave component n.} Recall that since $n=n_0+A^\al\pa_\al \td{n}$, we then consider $\| \chi^{\frac12}(|\cdot|-t) \Lg |\cdot|-t\Rg \pa \Ga^I n_0\|_2$ and $\|\chi^{\frac12}(|\cdot|-t)\Lg |\cdot|-t\Rg \pa^2 \Ga^I \td n\|_2$ for any multi-index $I$ with $|I|\le k_2$. We first notice that 
%\begin{align}
%\Box \Ga^I n_0=0,
%\end{align}
%similar to the arguments in \eqref{3.34t}-\eqref{3.34y}.

To proceed, we apply the test function $\chi(\rho-t)\Lg\rho-t\Rg^2 \pa_t\Ga^I \phi$ to \eqref{3.0A} (with $|I|\le K-5$) and after integrating by parts we arrive at
 \begin{align}\label{3.45}
  &\frac{d}{dt}\left(
   \|\chi^{\frac12}(\rho-t)\Lg \rho-t\Rg \pa \Ga^I \phi \|_2^2+\|\chi^{\frac12}(\rho-t)\Lg \rho-t\Rg  \Ga^I \phi \|_2^2\right)\\+&\int_{\R^2} \left(\chi'(\rho-t)\Lg \rho-t\Rg^2+2\chi(\rho-t) \cdot (\rho-t)\right) (|T\Ga^I \phi|^2+|\Ga^I\phi|^2) \\
    \ls& \sum_{I_1+I_2\le I}\int_{\R^2}|\chi(\rho-t)\Lg \rho-t\Rg^2 \pa\Ga^{I_1} n \Ga^{I_2}\phi \pa_t \Ga^I \phi|\\
    \ls& \sum_{I_1+I_2\le I}\|\chi^{\frac12}(\rho-t)\Lg \rho-t\Rg\pa \Ga^{I_1}n\Ga^{I_2}\phi\|_2\cdot \|\chi^{\frac12}(\rho-t)\Lg \rho-t\Rg\pa_t \Ga^I \phi\|_2,
 \end{align}
 where $|T\Ga^I \phi|=\sum_{i=1}^2|T_i \Ga^I \phi|^2$ as $T_i$ are defined in \eqref{DefT}. It then can be observed that 
\begin{align}\label{3.34y}
    \chi'(\rho-t)\Lg \rho-t\Rg^2+2\chi(\rho-t)\cdot (\rho-t)\ge 0.
\end{align}
Therefore \eqref{3.45} leads to
\begin{equation}\label{3.36a}
 \begin{aligned}
   & \frac{d}{dt}\left(
   \|\chi^{\frac12}(\rho-t)\Lg \rho-t\Rg \pa \Ga^I \phi \|_2^2+\|\chi^{\frac12}(\rho-t)\Lg \rho-t\Rg  \Ga^I \phi \|_2^2\right)\\
    \ls& \sum_{I_1+I_2\le I}\|\chi^{\frac12}(\rho-t)\Lg \rho-t\Rg\pa \Ga^{I_1}n\Ga^{I_2}\phi\|_2\cdot \|\chi^{\frac12}(\rho-t)\Lg \rho-t\Rg\pa_t \Ga^I \phi\|_2,
 \end{aligned}
 \end{equation}
It then suffices to estimate $\|\chi^{\frac12}(\rho-t)\Lg \rho-t\Rg\pa \Ga^{I_1}n\Ga^{I_2}\phi\|_2$ by dividing into two cases:
\begin{align}
    &\sum_{I_1+I_2\le I}\|\chi^{\frac12}(\rho-t)\Lg \rho-t\Rg\pa \Ga^{I_1}n\Ga^{I_2}\phi\|_2\\ \ls& \|\chi^{\frac12}(\rho-t)\Lg \rho-t\Rg \pa \Ga^{\le [\frac{K-5}{2}]} n \Ga^{\le K-5}\phi\|_2+\|\chi^{\frac12}(\rho-t)\Lg \rho-t\Rg \pa \Ga^{\le K-5} n \Ga^{\le [\frac{K-5}{2}]}\phi\|_2\\
    \coloneqq &A_1+A_2.
\end{align}
$A_2$ is handled as follows:
\begin{align}
    A_2\ls \|\Lg \rho-t\Rg \pa \Ga^{\le K-5}n\|_2\|\Ga^{\le [\frac{K-5}{2}]}\phi\|_\infty\ls \varepsilon^2 \Lg t\Rg^{-1+\de_1},\quad \text{by \eqref{3.59a} and \eqref{A3}.  }
\end{align}
Similarly we can estimate $A_1:$
\begin{align}
    A_1\ls \|\chi^{\frac12}(\rho-t)\Lg \rho-t\Rg \pa \Ga^{\le [\frac{K-5}{2}]} n\|_2\| \Ga^{\le K-5}\phi\|_\infty\ls \varepsilon^2\Lg t\Rg^{-1+\de_1}.
\end{align}
%Recall that \eqref{3.17}-\eqref{3.17a}:
%\begin{align*}
%\begin{cases}
% &\Box \Gamma^I \td{n}=\sum_{I_1+I_2=I }C_{I;I_1,I_2}\Gamma^{I_1}\phi\Gamma^{I_2}\phi,\\
%    &\Box \pa\Gamma^I \td{n}=\sum_{I_1+I_2=I }\td{C}_{I;I_1,I_2}\pa\Gamma^{I_1}\phi\Gamma^{I_2}\phi.
%\end{cases}
%\end{align*}
%For the second term $\|\chi^{\frac12}(r-t)\Lg t+r\Rg \Box \Ga^I \td{n}\|_2$, we again recall \eqref{3.35s}. It then follows that
%\begin{equation}\label{3.44s}
%    \begin{aligned}
%    \|\chi^{\frac12}(r-t)\Lg t+r\Rg \Box \Ga^I \td{n}\|_2\ls & \sum_{I_1+I_2=I}\|  \Lg t+r\Rg \Ga^{J_1}\phi\Ga^{J_2}\phi\|_2\\
%    \ls&\|\Lg t+r\Rg \Ga^{\le [\frac{k_2}{2}]}\phi\|_{\infty}\|\Ga^{\le k_2}\phi\|_2\\
%    \le& C\varepsilon^2 \Lg t\Rg^{\de_1} \quad (\text{by \eqref{A3} and \eqref{3.32} }). 
%\end{aligned}
%\end{equation}
%Combining \eqref{3.44s},\eqref{3.44d},\eqref{3.36a} and \eqref{3.34u}, we thereby arrive at
%\begin{align}\label{3.33aa}
%    \|\chi^{\frac12}(r-t)\Lg r-t\Rg \pa \Ga^{\le k_2}n\|_2\le C_i+C\varepsilon^2t^{\de_1}
%\end{align}
We then conclude \eqref{3.33a} by carefully noticing that the initial condition holds (see \eqref{B1} and \eqref{B1a} where the extra weight is harmless)
{
\begin{align}
    \left\|\left(\chi^{\frac12}(\rho-t)\Lg \rho-t\Rg \pa \Ga^{\le K-5} \phi \right)\Big|_{t=0}\right\|_2+\left\|\left(\chi^{\frac12}(\rho-t)\Lg \rho-t\Rg  \Ga^{\le K-5} \phi \right)\Big|_{t=0}\right\|_2\le C_i \varepsilon.
\end{align}
}
%\cmtr{
%    It is worthwhile mentioning that we require that $[\frac{k_2}{2}]\le k$ for everything except $A_1$ where $k_2\le k$.
%}
\end{proof}

%\begin{rem}
%    The estimate \eqref{3.44d} in fact holds for a wider range of the multi-index:
%    The proof of \eqref{3.44e} follows from exactly the same arguments as in showing \eqref{3.44d}.
%\end{rem}
Before we head to show the uniform boundedness of the lower-order energy, it is of much use to show the following lemmas.
\begin{lem}\label{lem3.7a}
    Let $(\phi,n)$ be a couple of solutions to \eqref{eq:kg}. Assume \eqref{A2} and \eqref{A3} hold. Then there exists some constant $C>0$ such that
\begin{align}\label{3.52}
    &|\Ga^{\le K-7}\phi|\le C\varepsilon \frac{\Lg t-r\Rg }{\Lg t+r\Rg^2},\\
    &{\chi^{\frac12}(r-t)|\Ga^{\le K-7}\phi|\le C\varepsilon \Lg t+r\Rg^{-\frac{5}{4}+\frac{\de_1}{2} }}\label{3.52b}
\end{align}

\end{lem}

\begin{proof}
The first inequality follows directly from Lemma~\ref{lem2.7} and Remark~\ref{rem2.12}:
\begin{align}
    |\Ga^{\le K-7}\phi|\ls& \frac{\Lg r-t\Rg}{\Lg r+t\Rg}|\pa\Ga^{\le K-6}\phi|+|\pa \Ga^{\le K-7}n \Ga^{\le K-7}\phi|,\quad (\text{by \eqref{kgw} })\\ \le& C\varepsilon\frac{\Lg r-t\Rg}{\Lg r+t\Rg^2}+ C\varepsilon \Lg r-t\Rg^{-\frac{1+\de}{2}}\Lg t\Rg^{-\frac12}|\Ga^{\le K-7}\phi|,\quad (\text{ by \eqref{A2} and \eqref{A3}}).
\end{align}
Then we obtain the desired \eqref{3.52} using the smallness assumption. To show \eqref{3.52b}, we first recall by \eqref{t12b} and \eqref{3.33a} one can derive that
\begin{align}\label{3.63}
    \chi^{\frac12}(r-t)\Lg r-t\Rg |\Ga^{\le K-7}\phi|\ls \Lg t\Rg^{-\frac12}\|\na^{\le 1}(\chi^{\frac12}(|\cdot|-t) \Lg |\cdot|-t \Rg\Ga^{\le K-6}\phi)\|_2\le C_i\varepsilon \Lg t\Rg^{-\frac12}+C\varepsilon^2t^{-\frac12+\de_1}.
\end{align}
As a result, we can further get
\begin{align}
 \chi^{\frac12}(r-t)|\Ga^{\le K-7}\phi|  \ls &(\chi^{\frac12}(r-t)|\Ga^{\le K-7}\phi|)^{\frac12}|\Ga^{\le K-7}\phi|^{\frac12}\\
 \le& C\varepsilon \Lg t\Rg^{-\frac14+\frac{\de_1}{2} }\Lg r-t\Rg^{-\frac12}\cdot \Lg t+r\Rg^{-1}\Lg t-r\Rg^{\frac12}\quad \text{ by \eqref{3.63} and \eqref{3.52}}.
\end{align}
We thus conclude \eqref{3.52b} from here.
\end{proof}

%\begin{rem}
%    It is interesting to point out that the proof for Lemma~\ref{lem3.1a} and \eqref{3.59a} in fact imply 
%    \cmtb{\begin{align}\label{3.59b}
%           \|\Lg t-r\Rg \pa \Ga^{\le K-1} n \|_2\le C_i\varepsilon +C\varepsilon^2 \Lg t\Rg^{\de_1}. 
%    \end{align}}
%\end{rem}

With many estimates in hand, we present two uniform boundedness results of certain lower-order energy. As can be seen later, they play key roles in completing the bootstrap argument and in showing the linear scattering as well.

\begin{prop}[Uniform boundedness of the lower-order energy (wave)]\label{lem3.2}
Let $(\phi,n)$ be a couple of solutions to \eqref{eq:kg}. Assume \eqref{A2} and \eqref{A3} hold. Then we have for some constant $C>0$,
\begin{align}\label{3.34}
\|\pa \Gamma^{\le K-2}n\|_{2} \le C_i\varepsilon+C\varepsilon^2,
\end{align}
where $C_i>0$ is exact from the choice of the initial data. Moreover we have
\begin{align}\label{3.34b}
     \|\chi^{\frac12}(|\cdot|-t) \pa \Ga^{\le K-1} \td{n} \|_2\le C_i\varepsilon +C\varepsilon^2 .
\end{align}
\end{prop}
\begin{proof}
We show \eqref{3.34}. Firstly we recall that \eqref{3.9} and \eqref{3.31}, 
it then suffices to consider $\|\pa^2 \Ga^{\le K-3}\td{n}\|_2$. 
Let $J$ be a multi-index and $|J|=m\le K-2$. Then we have
\begin{align}\label{3.35}
\Box  \pa \Gamma^J \td{n}=\sum_{J_1+J_2=J}C_{J;J_1,J_2}\pa \Gamma^{J_1}\phi\Gamma^{J_2}\phi.
\end{align}
 Multiplying both sides of \eqref{3.35} by $\partial_t\pa\Gamma^J \td{n}$ and integrating, we obtain that
\begin{align}
\frac{d}{dt}\|\partial^2\Gamma^J \td{n}\|_{2}^2
\lesssim& \sum_{J_1+J_2= J}\int|\pa\Gamma^{J_1}\phi\Gamma^{J_2}\phi \partial_t\pa\Gamma^J \td{n}|
 \\ \lesssim & \Lg t\Rg^{-1}\|\Gamma^{\le [\frac{m}2]+1}\phi\|_{\infty}\left\|\frac{\Lg |\cdot|+t\Rg}{\Lg |\cdot|-t\Rg}\Gamma^{\le m+1}\phi\right\|_{2}
             \|\Lg |\cdot|-t\Rg  \partial^2\Gamma^J \td{n}\|_{2}
\\ \le& C\varepsilon^2 \Lg t\Rg^{-2+2\de_1}\qquad(\text{by \eqref{A3},\eqref{3.322} and \eqref{3.59a}}),
\end{align}
which shows \eqref{3.34} taking the initial condition {\eqref{3.44}} into account. We now proceed to show \eqref{3.34b}.
%Now we proceed to show \cmtb{\eqref{3.34b}}. 
Let $J$ be a multi-index and $|J|\le K-1$ and consider the following equation
\begin{align}\label{3.35b}
\Box  \Gamma^J \td{n}=\sum_{J_1+J_2= J}C_{J;J_1,J_2}\Gamma^{J_1}\phi\Gamma^{J_2}\phi.
\end{align}
Testing \eqref{3.35b} with $\chi(\rho-t)\pa_t \Ga^J \td{n}$ we arrive at (after integrating by parts and ignoring the negative parts)
\begin{align}
    \frac{d}{dt}\|\chi^{\frac12}(|\cdot|-t)\pa \Ga^J\td{n}\|_2^2\ls&\sum_{J_1+J_2=J}\int_{\R^2}|\chi(\rho-t)\Ga^{J_1} \phi \Ga^{J_2}\phi \pa_t \Ga^J \td{n}|\\
    \ls&\|\chi^{\frac12}(|\cdot|-t)\Ga^{\le [\frac{K-1}{2}]}\phi\|_{\infty}\|\Ga^{\le K-1}\phi\|_2\|\chi^{\frac12}(|\cdot|-t)\pa\Ga^J\td{n}\|_2\\
    \le& C\varepsilon^2 \Lg t\Rg^{-\frac54+\frac{3\de_1}{2}} \|\chi^{\frac12}(|\cdot|-t)\pa \Ga^J\td{n}\|_2,\quad \text{ by \eqref{3.52b}}.
\end{align}
\eqref{3.34b} thus follows from here by observing that $[\frac{K-1}{2}]\le  K-7$ by choosing $K\ge 12$.

\end{proof}

\begin{prop}[Uniform boundedness of the lower-order energy (Klein-Gordon)]
    \label{lem3.3}
Let $(\phi,n)$ be a couple of solutions to \eqref{eq:kg}. Assume \eqref{A2} and \eqref{A3} hold.
Then we have for some constant $C>0$,
\begin{align}\label{3.36}
\|\Gamma^{\le K-1}\phi\|_{2}+\|\partial\Gamma^{\le K-1}\phi\|_2\le C_i\varepsilon+C\varepsilon^2,
\end{align}
where $C_i>0$ is exact from the choice of the initial data.
\end{prop}

\begin{proof}
The proof basically follows from a usual energy estimate. Let $I$ be a multi-index with $|I|\le K-1$, we arrive at 
\begin{align}\label{3.37}
    \Box \Ga^I \phi =\sum_{I_1+I_2=I} C^\al_{I;I_1,I_2}\pa_\al\Ga^{I_1}n\Ga^{I_2}\phi.
\end{align}
Taking the $L^2$ inner product of the equation \eqref{3.37} with $\pa_t \Ga^I \phi$, we again get:
\begin{align}
    \frac{d}{dt}(\| \pa \Ga^I \phi\|_2^2+\|\Ga^I \phi\|_2^2)\ls&  \sum_{I_1+I_2=I}\int |\pa \Ga^{I_1}n\Ga^{I_2}\phi \pa_t\Ga^I \phi|.
\end{align}
To continue we again consider two cases:

\texttt{Case 1:} $|I_1|\ge |I_2|$.
\begin{align}
\text{RHS}\lesssim &\| \Lg \rho-t\Rg \pa\Gamma^{\le|I|}n\|_{2}\|\Lg \rho-t\Rg^{-1}\Gamma^{\le [\frac{|I|}2]}\Phi\|_{\infty}\|\partial_t\Ga^I\phi\|_{2}
\\ \le &C\varepsilon^2 t^{-2+\de_1}\|\pa_t\Ga^{I}\phi\|_2, \qquad \text{ by \eqref{3.52} and \eqref{3.59a}}.
\end{align}
Note that the estimate above is guaranteed by {$[\frac{K-1}{2}]\le K-7$} since $K\ge 12$.

\texttt{Case 2:} $|I_1|\le |I_2|$.
\begin{align}
\text{RHS}\lesssim \Lg t\Rg^{-1}\|\Lg \rho-t\Rg  \pa\Gamma^{\le[\frac{|I|}{2}]}n\|_{\infty}\left\|\frac{\Lg \rho+t\Rg}{\Lg \rho-t\Rg}\Gamma^{\le |I| }\Phi\right\|_{2}\|\partial_t\Ga^I\phi\|_{2}.
\end{align}
To proceed further, we notice that when $|x|\le \frac{t}{2}$ we have from \eqref{A2} that
\begin{align}\label{3.84a}
    |\Lg |x|-t\Rg \pa \Ga^{\le[\frac{K-1}{2}]}n| \ls \Lg t\Rg|\pa \Ga^{\le[\frac{K-1}{2}]}n|\ls\Lg t\Rg^{-\frac{\de}{2}} ,\qquad \text{ by \eqref{A2}}.
\end{align}
For the case $|x|\ge \frac{t}{2}$, one recall that by \eqref{t12b}
\begin{align}\label{3.84b}
|\Lg |x|-t\Rg  \pa\Gamma^{\le[\frac{K-1}{2}]}n| \ls& \Lg t\Rg^{-\frac12}\|\pa_{\rho}^{\le 1}(\Lg \rho-t\Rg \pa \pa_{\theta}^{\le1}\Ga^{\le [\frac{K-1}{2}]} n) \|_2\qquad \text{ by \eqref{t12b}}\\
\ls&\Lg t\Rg^{-\frac12}\|\Lg \rho-t\Rg\pa \Ga^{\le[\frac{K-1}{2}]+2}n\|_2\\
\ls& \Lg t\Rg^{-\frac12+\de_1},\qquad \text{ by \eqref{3.59a}}.
    %\\ \le &C\varepsilon \cmt{t^{-\frac{3}{2}+2\de_1}\|\pa_t\Ga^{I}\phi\|_2, \qquad \text{ by \eqref{A2} and \eqref{3.322}} what if |\pa n|\ls \Lg t\Rg^{-\frac12}?}
\end{align}
Combining the estimates in \eqref{3.84a} and \eqref{3.84b} we arrive at
\begin{align}
    \text{RHS}\le C\varepsilon^2 (\Lg t\Rg^{-1-\frac{\de}{2}+\de_1}+\Lg t\Rg^{-\frac32+2\de_1})\|\pa_t\Ga^I\phi\|_2,\qquad \text{ by \eqref{3.322}}.
\end{align}
We then conclude \eqref{3.36} from here. Note that we in fact have {$\frac{K-1}{2}+2\le K-1$ and $\frac{K-1}{2}\le K-5$} here and $\de_1<\frac{\de}{4}$ by definition. We therefore obtain \eqref{3.36}.
\end{proof}

\section {pointwise decay-in-time and scattering}\label{sec:close}
In this section we complete the proof of Theorem~\ref{thm} and Theorem~\ref{thm1} by showing the $a\ priori$ hypothesis and scattering. We first prove the $a\ priori$  assumptions:

\subsection{Assumption 1: wave component \eqref{A2}}
Recall $n=n_0+A_\al\pa_\al\td{n}$, then given multi-index $|\alpha|\le K-5$ and $|x|=r$. Unlike \eqref{A2} we propose to show the following instead:
\begin{align}
|\pa\Ga^\al n_0|+|\pa^2\Gamma^\alpha \td{n}|\le C\varepsilon \langle r-t\rangle^{-\frac{1+\de}{2}}\langle t+r\rangle^{-\frac12}.\label{4.2}
\end{align} 
Note that from Lemma~\ref{lemB2} and the initial condition \eqref{ID2}-\eqref{ID3} we have even
\begin{align}
    |\pa\Ga^\al n_0|\le C_i\varepsilon\Lg r-t\Rg^{-\frac32}\Lg t\Rg^{-\frac12}\le C_i\varepsilon\Lg t-r\Rg^{-1}\Lg t+r\Rg^{-\frac12},
\end{align}
where the constant $C_i>0$ results from the initial condition and should be clear in the context. Then by recalling \eqref{3.31} it suffices for us to show that 
\begin{align}\label{3.3}
|\partial^2 \Gamma^{\le |\alpha|} \td{n}|
\le C\varepsilon^2\langle r-t\rangle^{-\frac{1+\de}{2}}\langle t+r\rangle^{-\frac12}.
\end{align}
For this purpose, we assume that $t>2$ with no loss. Recall that \eqref{2.260}: $|\langle r-t\rangle\partial^2 u|\le |\partial \Gamma^{\le 1}u |+(t+r)|\Box u|$ for all $r>0$. We then consider two cases.

{\bf Case 1: $r\le \frac{t}{2}$ or $r>2t$}.
Noting that $\Box \td{n}=|\phi|^2$, we arrive at
\begin{align}
|\langle r-t\rangle \partial^2\Gamma^{\alpha} \td{n}|\lesssim &| \partial\Gamma^{\le |\alpha|+1} \td{n}|+(t+r)|\Box  \Gamma^{\alpha}\td{n}|
\\ \lesssim &\ \Lg t+r\Rg^{-\frac12}{\|\partial \Gamma^{\le |\alpha|+4} \td{n}\|_{2}}+\|\langle t+|\cdot|\rangle \Gamma^{\le [\frac {|\alpha|} 2]}\phi\|_{\infty}\|\Gamma^{\le |\alpha|}\phi \|_{\infty}\qquad (\text{by }\eqref{t12},\eqref{t12b})
\\ \lesssim &\ \varepsilon^2 \Lg t+r\Rg^{-\frac12}t^{\de_1}+\varepsilon^2\Lg t+r\Rg^{-1}\qquad(\text{by } \eqref{A3}).\label{3.4}
\end{align}
Then by \eqref{3.44e}, we obtain by taking $0<\de<\frac12$ and $\de_1<\frac{\de}{4}$ that
\begin{equation}
\label{3.4a}
\begin{aligned}
    |\pa^2\Ga^\al \td{n}|\lesssim &\varepsilon^2 \Lg t-r\Rg^{-\frac{1+\de}{2}} \Lg t-r\Rg^{-\frac{1-\de}{2}}\Lg t+r\Rg^{-\frac12}\Lg t\Rg^{\de_1}\quad (\mbox{by }\eqref{3.4})\\
    \lesssim&\varepsilon^2\Lg r-t\Rg^{-\frac{1+\de}{2}}\Lg r+t\Rg^{-\frac12}.
\end{aligned}   
\end{equation}
% {\color{blue}Then we can assume \eqref{A2}: $|\Gamma^{\le k_1}n|\le (C\varepsilon)^2\langle r-t\rangle^{-\frac12}\langle r+t\rangle^{-1}$.}

{\bf Case 2: $\frac t2\le r\le 2t$}. If $|t-r|\le 1$, it follows from \eqref{t12b} that
\begin{align}
    \Lg t-r\Rg^{1+\de} t |\pa^2\Ga^\al \td{n}|^2\lesssim t|\pa^2\Ga^\al\td{n}|^2\lesssim\|\Ga^{\le 1}\pa^2\Ga^\al \td{n}\|_2^2+\|\pa\Ga^{\le 1}\pa^2 \Ga^\al \td{n}\|_2^2\ls \|\pa^2 \Ga^{\le K-3}\td{n}\|_2^2.
\end{align}
Then by \eqref{3.34} we can conclude
\begin{equation}
\begin{aligned}
    |\pa^2\Ga^\al \td{n}|\le C \varepsilon^2\Lg r-t\Rg^{-\frac{1+\de}{2}}\Lg r+t\Rg^{-\frac12}.
\end{aligned}   
\end{equation}
It then remains to show \eqref{3.3} for the case when $|t-r|\ge 1$ and $\frac{t}{2}\le r\le 2t$. To begin with,  it follows by \eqref{t1} that
\begin{align}
&\langle r-t\rangle^{\frac{1+\de}{2}} t^{\frac12}|\partial^2 \Gamma^{\alpha} \td{n}|
 \\\lesssim& \| \Lg t-|\cdot|\Rg^{\frac{1+\de}{2}}\pa_r^{\le 1}\pa_\theta^{\le 1}\partial^2\Gamma^{\al } \td{n}\|_{2}%\qquad\text{(by \eqref{t1})}
 \\ \ls &\|\chi^{\frac12}(|\cdot|-t)\Lg t-|\cdot|\Rg^{\frac{1+\de}{2}} \pa^2 \Ga^{\le |\al|+2}\td{n}\|_2+\|(1-\chi^{\frac12}(|\cdot|-t))\Lg t-|\cdot|\Rg^{\frac{1+\de}{2}} \pa^2 \Ga^{\le |\al|+2}\td{n}\|_2\\
 \coloneqq& B_1+B_2.
%\\ \lesssim &\cmtr{\|\partial\Gamma^{\le |\alpha|+3}\td{n}\|_{2}}+\|\langle t+|\cdot|\rangle\Box\Gamma^{\le |\alpha|+2}\td{n}\|_{2},\qquad \text{(by \eqref{2.260})}.
%\\ \lesssim &\  \varepsilon^2+\|\langle t+|\cdot|\rangle\Box\Gamma^{\le |\alpha|+3}n^{\Delta}\|_{2}\qquad \text{(by \eqref{3.34})}.
\end{align}
Note that we can estimate $B_1$ as follows:
\begin{align}
    B_1\lesssim& \|\chi^{\frac12}(|\cdot|-t)\Lg t-|\cdot|\Rg \pa^2 \Ga^{\le K-3}\td{n}\|_2 \\
    \ls & \| \chi^{\frac12}(|\cdot|-t)\pa \Ga^{\le K-2}\td{n}\|_2+ \| \chi^{\frac12}(|\cdot|-t)
    (t+|\cdot|)\Box \Ga^{\le K-3}\td{n}\|_2\le C\varepsilon^2,    \quad \text{ by \eqref{3.34b} and \eqref{3.52b}}.
\end{align}
To control $B_2$, we notice that 
\begin{align}
    (1-\chi^{\frac12}(\rho-t))^2\ls 1-\chi(\rho-t)\ls \chi(t-\rho)+\eta(\rho,t),
\end{align}
where the cut-off function $\eta(\rho,t)$ is chosen to be supported in $\{|\rho-t|\le 2\}$ and its value is 1 in the region $\{|\rho-t|\le 1\}$. As a result, it follows that
\begin{align}
    B_2\ls& \|(1-\chi(\rho-t))^{\frac12}\Lg \rho-t\Rg^{-\frac{1}{2}+\frac{\de}{2}}\Lg \rho-t\Rg \pa^2\Ga^{\le K-3}\td{n}\|_2\\
\ls& \|\chi^{\frac12}(t-\rho)\Lg t-\rho\Rg^{-\frac{1}{2}+\frac{\de}{2}}\Lg t-\rho\Rg \pa^2\Ga^{\le K-3}\td{n}\|_2+\|\eta(\rho,t)\pa^2\Ga^{\le K-3}\td{n}\|_2\\
\ls&\varepsilon^2+\|\chi^{\frac12}(t-\rho)\Lg t-\rho\Rg^{-\frac{1}{2}+\frac{\de}{2}}\Lg t-\rho\Rg \pa^2\Ga^{\le K-3}\td{n}\|_2, \qquad \text{ by \eqref{3.34}}\\
    \ls& \varepsilon^2+\|\chi^{\frac12}(t-\rho)\Lg t-\rho\Rg^{-\frac12+\frac{\de}{2}} \pa\Ga^{\le K-2}\td{n}\|_2+\|\chi^{\frac12}(t-r)\Lg t-r\Rg^{-\frac12+\frac{\de}{2}} \Box \Ga^{\le K-3}\td{n}\|_2,\qquad \text{ by \eqref{2.260}}\\
    \ls&\varepsilon^2+\|\chi^{\frac12}(t-\rho)\Lg t-\rho\Rg^{-\frac12+\frac{\de}{2}} \pa\Ga^{\le K-2}\td{n}\|_2+\|\chi^{\frac12}(t-\rho)\Lg t-\rho\Rg^{-\frac12+\frac{\de}{2}}  \Ga^{\le [\frac{K-3}{2}]}\phi\|_\infty\|\Ga^{\le K-3}\phi\|_2\\
    \ls& \varepsilon^2 +\|\chi^{\frac12}(t-\rho)\Lg t-\rho\Rg^{-\frac12+\frac{\de}{2}} \pa\Ga^{\le K-2}\td{n}\|_2, \qquad \text{ by \eqref{A3} and \eqref{3.36},}
\end{align}
where the requirement {$[\frac{K-3}{2}]\le K-5$}  is guaranteed by our choice of $K$. It then suffices to control $\|\chi^{\frac12}(t-|\cdot|)\Lg t-|\cdot|\Rg^{-\frac12+\frac{\de}{2}} \pa\Ga^{\le K-2}\td{n}\|_2$. Letting $\beta$ be a multi-index and $|\beta|\le K-2$, we get
\begin{align}
\Box\Gamma^{\beta}\td{n}=\sum_{\beta_1+\beta_2= \beta}C_{\beta;\beta_1,\beta_2}\Gamma^{\beta_1}\phi\Gamma^{\beta_2}\phi.
\end{align}
Taking the test function $\left[\chi(t-\rho)\Lg t-\rho\Rg^{-1+\de}+1-\chi(t-\rho)\right]\pa_t \Ga^\be \td{n}$, it leads to that
\begin{align}
 \text{LHS}=&\frac12\frac{d}{dt}\int_{\R^2}\left[\chi(t-\rho))\Lg t-r\Rg^{-1+\de}+1-\chi(t-\rho)\right]|\pa \Ga^\be \td{n}|^2\\&+\frac12\int_{\R^2}\left(\chi'(t-\rho)-\chi'(t-\rho)\Lg t-\rho\Rg^{-1+\de}+(1-\de)\Lg t-\rho\Rg^{-2+\de}(t-\rho)\chi(t-\rho)\right)|T\Ga^{\be}\td{n}|^2;\\
 \text{RHS}\ls&\int_{\R^2}\left[\chi(t-\rho)\Lg t-\rho\Rg^{-1+\de}+1-\chi(t-\rho)\right]|\Ga^{\le K-2}\phi \Ga^{\le [\frac{K-2}{2}]}\phi \pa_t\Ga^{\le K-2}\td{n}|.
\end{align}
To proceed now we split RHS into 3 parts as follows:
\begin{align}
  \chi(t-\rho)\Lg t-\rho\Rg^{-1+\de}+1-\chi(t-\rho)\le \chi(t-\rho)\Lg t-\rho\Rg^{-1+\de}+\eta(t,\rho)+\chi(\rho-t).
\end{align}
It then follows that $\text{RHS}\ls D_1+D_2+D_3$, where
\begin{align}
D_1\coloneqq& \int_{\R^2}\chi(t-\rho)\Lg t-\rho\Rg^{-1+\de}|\Ga^{\le K-2}\phi \Ga^{\le [\frac{K-2}{2}]}\phi \pa_t\Ga^{\le K-2}\td{n}|\\
\ls& \left\|\Lg t-\rho\Rg^{\de}\frac{\Ga^{\le [\frac{K-2}{2}]}\phi}{\Lg t-\rho\Rg}\right\|_{\infty}\|\Ga^{\le K-2}\phi\|_2\|\pa\Ga^{\le K-2}\td{n}\|_2\\
\le& C\varepsilon^2\Lg t\Rg^{-2+\de+2\de_1},\qquad \text{ by \eqref{3.52}, \eqref{3.32} and \eqref{3.44e}}.
\end{align}
Similarly, we have
\begin{align}
    D_2\coloneqq&\int_{\R^2}\eta(t,\rho)|\Ga^{\le K-2}\phi \Ga^{\le [\frac{K-2}{2}]}\phi \pa_t\Ga^{\le K-2}\td{n}|\\ \ls& \left\|\frac{\Ga^{\le [\frac{K-2}{2}]}\phi}{\Lg t-\rho\Rg}\right\|_{\infty}\|\Ga^{\le K-2}\phi\|_2\|\pa\Ga^{\le K-2}\td{n}\|_2\\
    \le& C\varepsilon^2 \Lg t\Rg^{-2+2\de_1},\qquad \text{ by \eqref{3.52}, \eqref{3.32} and \eqref{3.44e}};
\end{align}
and we have
\begin{align}
    D_3\coloneqq&\int_{\R^2}\chi(\rho-t)|\Ga^{\le K-2}\phi \Ga^{\le [\frac{K-2}{2}]}\phi \pa_t\Ga^{\le K-2}\td{n}|\\
    \ls&\| \chi^{\frac12}(\rho-t)\Ga^{\le [\frac{K-2}{2}]}\phi\|_{\infty} \|\Ga^{\le K-2}\phi\|_2\|\pa\Ga^{\le K-2}\td{n}\|_2\\
    \le&C\varepsilon^2\Lg t\Rg^{-\frac54+\frac{\de_1}{2}+\de_1}, \qquad \text{ by \eqref{3.52b},\eqref{3.32} and \eqref{3.44e}}.
\end{align}
Here our requirements {$[\frac{K-2}{2}]\le K-7$ is guaranteed by choosing $K \ge 12$.   
Observing that 
\begin{align}
    \chi'(t-\rho)-\chi'(t-\rho)\Lg t-\rho\Rg^{-1+\de}+(1-\de)\Lg t-\rho\Rg^{-2+\de}(t-\rho)\chi(t-\rho)\ge 0
\end{align}
and the assumptions on the initial data 
\begin{align}
    \int_{\R^2}\left[\chi(t-\rho))\Lg t-\rho\Rg^{-1+\de}+1-\chi(t-\rho)\right]|\pa \Ga^{\le K-2} \td{n}|^2\ dx\Big|_{t=0}\le C_i\varepsilon^2,
\end{align}
we get the desired estimate $B_2\le C\varepsilon^2$ by noting 
\begin{align}
    \|\chi^{\frac12}(t-\rho)\Lg t-\rho\Rg^{-\frac12+\frac{\de}{2}}\pa\Ga^{\le K-2}\td{n}\|^2_2\ls \int_{\R^2}\left[\chi(t-\rho))\Lg t-\rho\Rg^{-1+\de}+1-\chi(t-\rho)\right]|\pa \Ga^{\le K-2} \td{n}|^2.
\end{align}
%This implies that
%\begin{align}
%\|\langle t+|\cdot|\rangle \Box\Gamma^{\le |\alpha|+3}n^{\Delta}\|_{2}\lesssim \|\Gamma^{\le |\alpha|+3}\phi\|_{2}\|\langle t+|\cdot|\rangle \Gamma^{\le [\frac{|\alpha|+3}2]}\phi\|_{\infty}
%\le (C\varepsilon)^2t^{-\frac12} 
%\qquad \text{(by \eqref{3.32} and \eqref{A3})}.
%\end{align}
%It follows that 
%\begin{align}
%\langle r-t\rangle^{\frac12} t|\partial \nabla \Gamma^{\alpha} n^{\Delta}|
% \le (C\varepsilon)^2,\qquad \text{for }r\ge\frac t2.
%\end{align}
Thus we have proved \eqref{3.3} and \eqref{4.2}. Then it suffices to choose {$C_1>C_{i}+\widetilde {C}_{n}\varepsilon$ in \eqref{A2}}, where $C_i$ is the exact known constant from the initial condition {(cf. \eqref{B1}-\eqref{B4}).}

\subsection{Assumption 2: Klein-Gordon component \eqref{A3}}

Let $\beta$ be a multi-index and {$|\beta|\le K-5$}. Then  $\Gamma^{\beta}\phi$ satisfies
% Recall the equation of $\Gamma^{\alpha}\phi$ 
\begin{equation}
\begin{cases}
\Box \Gamma^{\beta}\phi +\Gamma^{\beta}\phi=-\sum_{\beta_1+\beta_2=\beta}C_{\beta;\beta_1,\beta_2}\pa\Gamma^{\beta_1}n\Gamma^{\beta_2}\phi=:\widetilde F,\\
(\phi,\partial_t\phi)|_{t=0}=(\Phi_0,\Phi_1).
\end{cases}
\end{equation}
By \eqref{d-1}, we have 
\begin{align}\label{4.38}
\langle t+|x|\rangle|\Gamma^{\beta}\phi(t,x)|\lesssim {\sup_{0<s\le t}}\Lg s\Rg^{\de'}\|\langle s+|\cdot|\rangle\Gamma^{\le 4}\widetilde F(s,\cdot)\|_{2}+\|\langle |\cdot|\rangle\log\Lg |\cdot|\Rg \Gamma^{\le 5}\Gamma^{\beta}\phi(0,\cdot)\|_{2},
\end{align}
for some $0<\delta' \ll 1$.
For any $s\in(0,t)$ we have
\begin{align}
&\|\langle s+|\cdot|\rangle\Gamma^{\le 4}\widetilde F(s,\cdot)\|_{2}
\\ \lesssim& \| \Lg s-|\cdot|\Rg \pa \Gamma^{\le |\beta|+4}n\|_{2}\left\|\frac{\langle s+|\cdot|\rangle}{\Lg s-|\cdot|\Rg}\Gamma^{\le [\frac{|\beta|+4}2]}\phi\right\|_{\infty}+\|\Lg s-|\cdot|\Rg \pa\Gamma^{\le [\frac{|\beta|+4}2]}n\|_{\infty}\left\|\frac{\Lg s+|\cdot|\Rg}{\Lg s-|\cdot|\Rg}\Gamma^{\le |\beta|+4}\phi\right\|_{2}\\
\coloneqq& G_1+G_2.
%\\\lesssim &\ \varepsilon^2 +\|\langle s+|\cdot|\rangle^{1+\delta}\Gamma^{\le [\frac{|\beta|+4}2]}n\Gamma^{\le |\beta|+4}\phi\|_{L^2}
 %\qquad \text{(by \eqref{3.59a} \eqref{A3} and \eqref{3.32})}.\label{3.59}
\end{align}
{Note that we have
\begin{align}
    G_1\le C\varepsilon^2 \Lg s\Rg^{-1+\de_1}\qquad \text{(by \eqref{3.52} and \eqref{3.59a})};
\end{align}
on the other hand, $G_2$ can be handled as follows:
\begin{align*}
    G_2\ls&\Lg s\Rg^{\de_1}\|\Lg s-|\cdot|\Rg \pa\Gamma^{\le [\frac{|\beta|+4}2]}n\|_{\infty},\quad \text{ by \eqref{3.322}},\\
    \ls &\Lg s\Rg^{\de_1}(\| \Lg s-|\cdot|\Rg \pa\Ga^{\le [\frac{K-1}{2}]}n_0\|_{\infty}+\| \Lg s-|\cdot|\Rg\pa^2\Ga^{\le [\frac{K-1}{2}]}\td{n}\|_{\infty}),\\
    \le& C_i\varepsilon\Lg s\Rg^{-\frac12+\de_1}+ C\Lg s\Rg^{\de_1}\left(\|\pa \Ga^{\le[\frac{K-1}{2}]+1 }\td{n}\|_\infty+\|\Lg s+|\cdot|\Rg\Ga^{\le[\frac{K-1}{2}]}\phi\Ga^{\le[\frac{K-1}{4}]+1}\phi\|_\infty\right)
    ,%\quad 
    %text{ by \eqref{BB3} and \eqref{2.260},}
    \\ \le&C_i\varepsilon\Lg s\Rg^{-\frac12+\de_1}+C\varepsilon^2\Lg s\Rg^{-\frac12+2\de_1}+C\varepsilon^2\Lg s\Rg^{-1+\de_1},\quad 
    \text{ by \eqref{t12},\eqref{t12b}, \eqref{3.44e} and \eqref{A3}.}
\end{align*}}
Here the constant $C$ may vary from line to line and we need the index {$[\frac{K-1}{2}]\le K-7$,which is true since $K \ge 12$. %$\frac{K-1}{2}+4\le K$, which is true since $K \ge 12$.}
It then follows by taking $\de'=\frac{1}{2}-2\de_1$ and by \eqref{4.38} that 
\begin{align}
\langle t+r\rangle^{-1}|\Gamma^{\beta}\phi(t,x)|\le& C\varepsilon^2+
\|\langle |\cdot|\rangle\log\Lg |\cdot|\Rg \Gamma^{\le 5}\Gamma^{\beta}\phi(0,\cdot)\|_{2}
\\ \\ \le& C_i\varepsilon+\widetilde C_{\phi}\varepsilon^2, \qquad\text{(by \eqref{ID2} and \eqref{B3}).}
\end{align}
% {\color{red}
% Here we need $\|\langle |\cdot|\rangle^2\Gamma^{\le k+4}\phi(0,\cdot)\|_{2}\le \sum_{j=0}^{K-1}\|\langle x\rangle^{j+3}\nabla^j(\nabla\Phi_0, \Phi_1,\nabla n_0,n_1)\|_{2}$.
% }
Then it suffices to choose $C_2>{C_i}+\widetilde {C}_{\phi}\varepsilon$ in \eqref{A3},
where $C_i$ is the known constant from the initial condition {(cf. \eqref{B1}-\eqref{B4}).}
This concludes the proof of Theorem~\ref{thm}.

\subsection{Scattering}\label{scat}
In this subsection, we briefly discuss about the scattering of the wave-Klein-Gordon system \eqref{eq:kg} in $\R^2$. We show that the Klein-Gordon field $\phi$ scatters to a linear Klein-Gordon
equation in its energy space as $t\to+\infty$ and the wave field $n$ scatters ``weakly''. We now proceed to show the scattering of $\phi$ in the energy space {$\mathcal{X}_{K}= H^{K}\times H^{K-1} $ and $\mathcal{Y}_{K}= \dot{H}^{K}\times \dot{H}^{K-1}\cap \dot{H}^1\times L^2 $  } , namely there exists a couple of solutions $(\phi_l,n_l)$ to the linear system such that \eqref{Phiscat}, \eqref{Nscat1} and \eqref{Nscat2} hold:
\begin{align*}
   & \lim_{t\to \infty}\|(\phi,\partial_t\phi)-(\phi_l,\partial_t\phi_l)\|_{\mathcal{X}_{K}}= 0,\\
    &\lim_{t\to + \infty}\int_{\R^2} \na D^{K-1} (n-n_l)\cdot \na D^{K-1} f+\pa_t D^{K-1}(n-n_l)\cdot \pa_tD^{K-1} f=0,\\ 
    &\lim_{t\to + \infty}\int_{\R^2}\na(n-n_l)\cdot \na f+\pa_t(n-n_l)\pa_tf=0,
\end{align*}
where $(\phi,\partial_t\phi,n,\pa_t n)|_{t=0}=(\Phi_0,\Phi_1,N_0,N_1)$ and $(\phi_l,\partial_t\phi_l,n_l,\pa_t n_l)|_{t=0}=(\Phi_{l_0},\Phi_{l_1},N_{l_0},N_{l_1}) $ for some $(\Phi_{l_0},\Phi_{l_1},N_{l_0},N_{l_1})\in\mathcal{X}_{K}\times \mathcal{Y}_{K}$. Assume that $f$ is the linear homogeneous solution to the wave equation with arbitrary initial data $(F_0,F_1)\in \mathcal{Y}_{K}$; here we denote
$D=(D_i)_{i=1}^2$ to be the usual differential operator with $D_1=\pa_1,D_2=\pa_2$. 

\begin{proof}[Proof of Theorem~\ref{thm1}]
We show the linear scattering of $\phi$ and $n$ separately. To start with, we consider the Klein-Gordon field $\phi$.

\texttt{Linear scattering of the Klein-Gordon component.}
Recall that for the Cauchy problem for the inhomogeneous 2d Klein-Gordon equation:
\begin{align}
    \Box \phi+\phi = H\, ,\ \mbox{with } (\phi,\partial_t \phi)|_{t=0}=(\Phi_0,\Phi_1),
\end{align}
the solution $\vec{\phi}=(\phi,\partial_t \phi)$ is given by
\begin{align}
    \vec{\phi}(t,x)=N(t)\begin{pmatrix}
\Phi_0\\
\Phi_1
\end{pmatrix}+\int_0^t N(t-s)\begin{pmatrix}
0\\
H(s)
\end{pmatrix}\ ds,
\end{align}
where 
\begin{align}
    N(t)=\begin{pmatrix}
\cos(t \Lg\na\Rg) & \Lg \na\Rg^{-1}\sin(t\Lg\na\Rg)\\
-\Lg\na\Rg\sin(t\Lg\na\Rg) & \cos(t\Lg\na\Rg)
\end{pmatrix}.
\end{align}
More specifically $H(n,\phi)=\phi B^\al\pa_\al n$ in our setting.
 By a standard semi-group argument, it is known that the Klein-Gordon flow $N(t)$ is a unitary semi-group on $H^{J}\times H^{J-1}$ for $J\in \mathbb{N}^+$. More precisely one can verify on the Fourier side that for any $J\in \mathbb{N}^+$ one has
 \begin{align}
\|N(t)\|_{\mathcal{H}^J\to\mathcal{H}^J }=1,
 \end{align}
where $\mathcal{H}^J=H^J\times H^{J-1}$. Moreover, one recalls that 
\begin{align}
  \begin{pmatrix}
        \phi_l\\
        \partial_t\phi_l
    \end{pmatrix}=N(t)
    \begin{pmatrix}
        \Phi_{l_0}\\
        \Phi_{l_1}
    \end{pmatrix}.
\end{align}
Note that by \eqref{3.32}-\eqref{3.34}, we have 
\begin{align}
    &\|H(\phi,n)\|_{H^{K-1}}\ls \|\phi \pa n \|_{H^{K-1}}\\
    \ls& \|\Lg t-\rho\Rg \pa \Ga^{\le K-1}n\|_{2}\left\|\frac{\Ga^{\le [\frac{K-1}{2}]}\phi}{\Lg t-\rho\Rg}\right\|_{\infty}+\Lg t\Rg^{-1}\|\Lg t-\rho\Rg\pa\Ga^{\le [\frac{K-1}{2}] }n\|_{\infty}\left\| \frac{\Lg t+\rho\Rg}{\Lg t-\rho\Rg}\Ga^{\le K-1}\phi\right\|_2  \\
    \le &C\varepsilon^2(\Lg t\Rg^{-2+\de_1}+\Lg t\Rg^{-\frac32+2\de_1}),\qquad \text{ by \eqref{3.59a}, \eqref{3.52}, \eqref{t12}-\eqref{t12b} and \eqref{3.322}}.
\end{align}
Note that we have {[$\frac{K-1}{2}]\le K-7$}. Therefore it follows immediately that
\begin{align}
    \int_0^\infty \|H(\phi,n)\|_{H^{K-1}}\le C\varepsilon^2.
\end{align}
In particular, we set
\begin{align}
    \begin{pmatrix}
        \Phi_{l_0}\\\Phi_{l_1}
    \end{pmatrix}
    =\begin{pmatrix}
        \Phi_{0}\\\Phi_{1}
\end{pmatrix}+\int_0^\infty N(-s)
    \begin{pmatrix}
        0\\
        H(\phi,n)
    \end{pmatrix}\ ds,
\end{align}
then since $N(t)$ is unitary, we get
\begin{align}
    \left\|\begin{pmatrix}
        \phi\\ \partial_t \phi    \end{pmatrix}-N(t)
        \begin{pmatrix}
        \Phi_{l_0}\\\Phi_{l_1}    
        \end{pmatrix}
        \right\|_{\mathcal{H}^{K}}\le \int_t^\infty \left\| H(\phi,n)\right\|_{H^{K-1}}\ ds.
\end{align}
Therefore we have shown that $\phi$ scatters to $\phi_l$ as $t\to+\infty$:
\begin{align}
    \lim_{t\to \infty}\|(\phi,\partial_t\phi)-(\phi_l,\partial_t\phi_l)\|_{\mathcal{X}_K}= 0.
\end{align}

\texttt{Weak linear scattering of the Klein-Gordon component.} We show \eqref{Nscat1} and then \eqref{Nscat2} follows as a corollary. Recall that we have 
\begin{align}
    &\Box n=A^\al \pa_\al (|\phi|^2),\\
    &\Box  n_l=0.
\end{align}
Then we consider the inner product of $\vec{n}=(D^{K-1}n,\pa_t D^{K-1}n)$ and $\vec{n_l}=(D^{K-1}n_l,\pa_t D^{K-1}n_l)$ on $\mathcal{Z}\coloneqq\dot{H}^1\times L^2$. Here
$D=(D_i)_{i=1}^2$ is the usual differential operator with $D_1=\pa_1,D_2=\pa_2$. 
We denote the inner product to be $(\cdot,\cdot)_{\mathcal{Z}}$. Then we differentiate it in time to get
\begin{align}
    \frac{d}{dt}\int_{\R^2}\na D^{K-1} n\cdot \na D^{K-1} n_l+\pa_t D^{K-1}n\cdot \pa_t D^{K-1}n_l=-\frac12 \int_{\R^2} \pa_t D^{K-1}n_l(A^\al \pa_\al D^{K-1} |\phi|^2).
\end{align}
Integrating in the time slot  $(s,t)$ we thus arrive at
\begin{align}\label{4.61}
\big(\vec{n}(t),\vec{n_l}(t)\big)_{\mathcal{Z}}-\big(\vec{n}(s),\vec{n_l}(s)\big)_{\mathcal{Z}}=-\frac12 \int_s^t\int_{\R^2}\pa_tD^{K-1} n_l(A^\al \pa_\al D^{K-1}|\phi|^2)\ dx d\tau.
\end{align}
Recall that the Cauchy problem for the following homogeneous (inhomogeneous) 2d wave system:
\begin{align}
    \Box D^{K-1}n= G\, ,\ \mbox{with } (D^{K-1}n,\partial_t D^{K-1}n)|_{t=0}=(D^{K-1}N_0,D^{K-1}N_1),
\end{align}
the solution $\vec{n}=(D^{K-1}n,\partial_t D^{K-1}n)$ is given by
\begin{align}
    \vec{n}(t,x)=M(t)\begin{pmatrix}
D^{K-1}N_0\\
D^{K-1}N_1
\end{pmatrix}+\int_0^t M(t-\tau)\begin{pmatrix}
0\\
G(\tau)
\end{pmatrix}\ d\tau,
\end{align}
where the wave flow $M(t)$ is given by
\begin{align}
    M(t)=\begin{pmatrix}
\cos(t |\na|) & |\na|^{-1}\sin(t|\na|)\\
-|\na|\sin(t|\na|) & \cos(t|\na|)
\end{pmatrix}.
\end{align}
Similarly we have
\begin{align}
    \vec{n_l}(t,x)=M(t)\begin{pmatrix}
D^{K-1}N_{l_0}\\
D^{K-1}N_{l_1}
\end{pmatrix}.
\end{align}
It is known that $M(t)$ defines a unitary semi-group on $\mathcal{Y}_k$. Accordingly we can rewrite \eqref{4.61} as
\begin{equation}\label{4.67}
\begin{aligned}
    &\left( M(t)\begin{pmatrix}
        D^{K-1}N_0\\
        D^{K-1}N_1
    \end{pmatrix} +\int_0^t M(t-\tau)\begin{pmatrix}
        0\\
        G(\tau)
    \end{pmatrix},M(t)\begin{pmatrix}
       D^{K-1} N_{l_0}\\
       D^{K-1} N_{l_1}
    \end{pmatrix}
    \right)_{\mathcal{Z}}\\  -&\left( M(s)\begin{pmatrix}
         D^{K-1}N_0\\
        D^{K-1}N_1
    \end{pmatrix} +\int_0^s M(s-\tau)\begin{pmatrix}
        0\\
        G(\tau)
    \end{pmatrix},M(s)\begin{pmatrix}
        D^{K-1} N_{l_0}\\
       D^{K-1} N_{l_1}
    \end{pmatrix}
    \right)_{\mathcal{Z}}\\=&-\frac12 \int_s^t\int_{\R^2} \pa_t n_l(A^\al \pa_\al D^{K-1} |\phi|^2)\ dxd\tau.
\end{aligned}
\end{equation}
Noticing $M(t)$ and $M(s)$ are unitary in $\mathcal{Z}$, \eqref{4.67} is then equivalent to 
\begin{align}\label{4.68}
    \left(\int_s^t M(-\tau)\begin{pmatrix}
        0\\
        G(\tau)
    \end{pmatrix} \ d\tau,\begin{pmatrix}
        D^{K-1} N_{l_0}\\
       D^{K-1} N_{l_1}
    \end{pmatrix}\right)_{\mathcal{Z}}=-\frac12 \int_s^t\int_{\R^2} \pa_tD^{K-1} n_l(A^\al \pa_\al D^{K-1}|\phi|^2)\ dxd\tau.
\end{align}
We observe that the right hand side of \eqref{4.68} can be estimated as follows:
\begin{equation}\label{4.69}
\begin{aligned}
    &\text{RHS}\\ \ls& \int_s^t \int_{\R^2}(|D^{\le [\frac{K-1}{2}]}\phi \pa D^{\le K-1}\phi|+ |\pa D^{\le [\frac{K-1}{2}]}\phi D^{\le K-1}\phi|) |\pa_t D^{K-1}n_l|\ dxd\tau\\
    \ls&\int_s^t \left(\left\|\frac{D^{\le [\frac{K-1}{2}]}\phi}{\Lg \rho-\tau\Rg}\right\|_{\infty}\|D^{\le K}\phi\|_2+\Lg \tau\Rg^{-1}\left\|\frac{\Lg\rho+\tau\Rg}{\Lg \rho-\tau\Rg}D^{\le K-1}\phi  \right\|_2 \| \pa D^{\le [\frac{K-1}{2}]}\phi\|_{\infty} \right)\|\Lg \rho-\tau\Rg \pa_tD^{K-1} n_l\|_2\ d\tau\\
    \ls&\int_s^t \Lg \tau\Rg^{-2+2\de_1}(\|\Lg x\Rg \na D^{K-1} N_{l_0}\|_2+\|\Lg x\Rg D^{K-1} N_{l_1}\|_2),\quad\text{ by \eqref{3.52}, \eqref{3.34}, \eqref{DE-1} and \eqref{S4-4}}.
\end{aligned}
\end{equation}
Note that here we have indeed $[\frac{K-1}{2}]\le K-7$ by the choice $K \ge 12$. On top of this we observe that the right hand side of \eqref{4.68} goes to 0 as $s,t\to +\infty$ so long as $\begin{pmatrix}
        D^{K-1} N_{l_0}\\
       D^{K-1} N_{l_1}
\end{pmatrix}$ lies in a suitable space; one can take it to be $\text{ID}2$, namely the initial condition space for the wave component. It then follows from \eqref{4.69} that
\begin{align}
      \left(\int_t^{+\infty} M(-\tau)\begin{pmatrix}
        0\\
        G(\tau)
    \end{pmatrix} \ d\tau, \begin{pmatrix}
        D^{K-1} N_{l_0}\\
       D^{K-1} N_{l_1}
    \end{pmatrix}\right)_{\mathcal{Z}}\ls \Lg t\Rg^{-1+2\de_1}\left\|\begin{pmatrix}
        D^{K-1} N_{l_0}\\
       D^{K-1} N_{l_1}
\end{pmatrix}\right\|_{\mathcal{W}}\xrightarrow{t\to +\infty} 0 ,
\end{align}
where 
\begin{align}
   \left\|\begin{pmatrix}
        D^{K-1} N_{l_0}\\
       D^{K-1} N_{l_1}
\end{pmatrix}\right\|_{\mathcal{W}}=\|\Lg x\Rg \na D^{K-1} N_{l_0}\|_2+\|\Lg x\Rg D^{K-1} N_{l_1}\|_2.
\end{align}
Therefore we obtain the convergence $\int_t^{+\infty} M(-\tau)\begin{pmatrix}
        0\\
        G(\tau)
    \end{pmatrix} \ d\tau$ to 0 in the dual space $\mathcal{W'}$:
\begin{align}
    \lim_{t\to+\infty}\left\|\int_t^{+\infty}  M(-\tau)\begin{pmatrix}
        0\\
        G(\tau)
    \end{pmatrix} \ d\tau \right\|_{\mathcal{W}'}=0,\quad \text{with a rate of }\Lg t\Rg^{-1+2\de_1}.
\end{align}
On the other hand notice that $\mathcal{W}$ is dense in $\dot{H}^1\times L^2$, one has the weak convergence:
\begin{align}\label{4.72}
     \lim_{t\to+\infty} \left(\int_t^{+\infty} M(-\tau)\begin{pmatrix}
        0\\
        G(\tau)
    \end{pmatrix} \ d\tau, \begin{pmatrix}
        D^{K-1} F_{l_0}\\
       D^{K-1} F_{l_1}
\end{pmatrix}\right)_{\mathcal{Z}}= 0 ,
\end{align}
for any $(D^{K-1} F_{l_0},  D^{K-1} F_{l_1})\in \dot{H}^1\times L^2$. In particular, we set
\begin{align}\label{4.73}
    \begin{pmatrix}
       D^{K-1} N_{l_0}\\ D^{K-1}N_{l_1}
    \end{pmatrix}
    =\begin{pmatrix}
       D^{K-1} N_{0}\\ D^{K-1}N_{1}
\end{pmatrix}+\int_0^{+\infty} M(-\tau)
    \begin{pmatrix}
        0\\
        G(\phi)
    \end{pmatrix}\ d\tau.
\end{align}
Rearranging \eqref{4.72} and \eqref{4.73} we arrive at
\begin{equation}
    \begin{aligned}
&\left( \begin{pmatrix}
       D^{K-1} N_{l_0}\\ D^{K-1}N_{l_1}
    \end{pmatrix}-\begin{pmatrix}
       D^{K-1} N_{0}\\ D^{K-1}N_{1}
\end{pmatrix}-\int_0^t M(-\tau)
    \begin{pmatrix}
        0\\
        G(\phi)
    \end{pmatrix}\ d\tau,\begin{pmatrix}
        D^{K-1} F_{l_0}\\
       D^{K-1} F_{l_1}
\end{pmatrix}
\right)_{\mathcal{Z}}\\
    =&\left(\int_t^{+\infty} M(-\tau)\begin{pmatrix}
        0\\
        G(\tau)
    \end{pmatrix} \ d\tau, \begin{pmatrix}
        D^{K-1} F_{l_0}\\
       D^{K-1} F_{l_1}
\end{pmatrix}\right)_{\mathcal{Z}}\xrightarrow{t\to+\infty}0.
\end{aligned}
\end{equation}
Noting that $M(t)$ is unitary in $\mathcal{Z}$, we have actually
\begin{equation}
    \begin{aligned}
0=&\lim_{t\to+\infty}\left( M(t)\begin{pmatrix}
       D^{K-1} N_{l_0}\\ D^{K-1}N_{l_1}
    \end{pmatrix}- M(t)\begin{pmatrix}
       D^{K-1} N_{0}\\ D^{K-1}N_{1}
\end{pmatrix}-\int_0^t M(t-\tau)
    \begin{pmatrix}
        0\\
        G(\phi)
    \end{pmatrix}\ d\tau, M(t)\begin{pmatrix}
        D^{K-1} F_{l_0}\\
       D^{K-1} F_{l_1}
\end{pmatrix}
\right)_{\mathcal{Z}}\\
=&\lim_{t\to+\infty}\int_{\R^2}\na D^{K-1} (n-n_l)\cdot \na D^{K-1}f+\pa_t D^{K-1}(n-n_l)\cdot \pa_t D^{K-1}f\ dx,
\end{aligned}
\end{equation}
where $f$, $n_l$ are the free waves with initial data $(F_{l_0},F_{l_1})$ and $(N_{l_0},N_{l_1})$ respectively. This proves \eqref{Nscat1} and \eqref{Nscat2} follows afterwards.

\end{proof}
\begin{rem}
The scattering results for the Klein-Gordon-Zakharov equations have been obtained in \cite{D,OTT95}, with high regular initial data. It is worth pointing out here that there are different
difficulties arising in seeking scattering results for initial data of low regularity data (cf. \cite{GNW14} on Klein-Gordon-Zakharov equations). Nevertheless, the strong scattering of the wave field $n$ cannot be shown via our method due to the critical decay-in-time as $$\|\pa(|\phi|^2)\|_{H^m}(s)\ls \Lg s\Rg^{-1}\implies \int_1^t\|\pa(|\phi|^2)\|_{H^m}\ ds\sim \log(t).$$
In fact, whether the wave
field $n$ scatters strongly (linearly or nonlinearly) is unknown. As a consequence we leave it open here.
\end{rem}

\begin{rem}\label{rem4.2}
    We note that the weak linear scattering we proved above is necessary to the strong linear scattering. Nevertheless, it is also worth pointing out here that the weak convergence immediately leads to a uniform boundedness in the energy class by the classic Uniform Boundedness Principle. More specifically speaking, one has
    \begin{align}
        \|(n,\pa_tn)\|_{\mathcal{Y}_K}<+\infty,
    \end{align}
    which is also necessary to the strong linear scattering.
\end{rem}

\section{The \eqref{eq:kg2} case}\label{sec:bcase}
In this section we discuss more about the global stability result of \eqref{eq:kg2} system. Recall that \eqref{eq:kg2} is of the following form:
\begin{align*}
    \begin{cases}
&\Box n = A^{\alpha\beta}\partial_{\alpha\beta}(|\phi|^2),
\\
&\Box \phi + \phi = B n \phi,
\end{cases}
\end{align*}
with  the initial data being $(\phi,\partial_t \phi, n,\pa_{t}n)|_{t=0}=(\Phi_0, \Phi_1, N_0, N_1)$. We decompose $n=n_0+A^{\al\be} \pa_{\al\be} N$ where
\begin{align}
\begin{cases}
\Box n_0=0,\\
(n_0,\partial_t n_0)|_{t=0} =(N_0,N_1),
\end{cases}
\qquad \text{and}\qquad\
\begin{cases}
\Box N=|\phi|^2,\\
(N,\partial_t N) |_{t=0}=(0,0).
\end{cases}
\end{align}
Different from the set-up in the \eqref{eq:kg1} case here we require $N_0=\pa_i M_0$ and $N_1=\pa_j M_1$ for some $i,j=1,2$ to apply our analysis tools. In fact similar set-up has been adopted by work \cite{DM21} which dates back to a well-known energy conservation law \eqref{1.19} (cf. \cite{GNW14}) in the area of Klein-Gordon-Zakharov system. Without loss of generality we write $n_0=Dm_0$, $n=Dm_0+A^{\al\be}\pa_{\al\be}N$ and 
\begin{align}
\begin{cases}
\Box m_0=0,\\
(m_0,\partial_t m_0)|_{t=0} =(M_0,M_1),
\end{cases}
\end{align}
where $D=(D_i)_{i=1}^2$ is the usual differential operator with $D_1=\pa_1,D_2=\pa_2$.
 \begin{thm}\label{thm2}
Consider the coupled wave-Klein-Gordon system as in \eqref{eq:kg2} and let $K$ be an integer no less than 12. Then for any small $0<\de<\frac12$, there exists small $\varepsilon_0>0$ such that if $\varepsilon\in (0,\varepsilon_0)$ and all initial data $(\Phi_0,\Phi_1,M_0,M_1)$ satisfying the smallness condition below: 
\begin{align}\label{ID2b}
    \|(\Phi_0,\Phi_1)\|_{\text{ID1}}+\|(M_0,M_1)\|_{\text{ID2}}<\varepsilon,
\end{align}
where the initial-data-norm is given by
\begin{equation}\label{ID3b}
\begin{split}
   \|(\Phi_0,\Phi_1)\|_{\text{ID1}}= &\|\Lg x \Rg^{K}\Lg \na\Rg^{K+1}\Phi_0\|_2+\|\Lg x \Rg^{K}\Lg \na\Rg^{K}\Phi_1\|_2+\|\Lg x\Rg^{K+1}\log(1+\Lg x\Rg) \Lg \na \Rg^{K}\Phi_0\|_2\\ &+\|\Lg x\Rg^{K+1}\log(1+\Lg x\Rg) \Lg \na \Rg^{K-1}\Phi_1\|_2,
    \\ \|(M_0,M_1)\|_{\text{ID2}}= &\sum_{j=0}^{K-5}\Big(\|\Lg x\Rg^{j+\frac52} \na^{j+3}(\na M_0,M_1)\|_1+\|\Lg x\Rg^{j+\frac52} \na^{j}(\na M_0,M_1)\|_1\Big)\\&+\sum_{j=0}^K\|\Lg x\Rg^{j+1}\na^j(\na M_0,M_1)\|_2.
    \end{split}
\end{equation}
%   \begin{align}\label{ID1}
% \sum_{j=0}^{K}\|\langle x\rangle^{j+1}\nabla^j(\nabla \Phi_0,\Phi_1)\|_{L^2}
% +\sum_{j=0}^{k_1+3}\|\langle x\rangle^{j+3}\nabla^j(\nabla \Phi_0,\Phi_1)\|_{L^2}
% +\sum_{j=0}^{K}\|\langle x\rangle^{j+1}\nabla^j(\nabla n_0,n_1)\|_{L^2}\le \varepsilon
% \end{align}
% or 
% \begin{align}
% \sum_{j=0}^{K+1}\|\langle x\rangle^{j+2}\nabla^j \Phi_0\|_{L^2}
% +\sum_{j=0}^{K}\|\langle x\rangle^{j+3}\nabla^j \Phi_1\|_{L^2}
% +\sum_{j=0}^{K}\|\langle x\rangle^{j}\nabla^j n_0\|_{L^2}
% +\sum_{j=0}^{K-1}\|\langle x\rangle^{j+1}\nabla^j n_1\|_{L^2}\le \varepsilon.
% \end{align}
We then can conclude the following:

\texttt{(i).} The Cauchy problem for the \eqref{eq:kg2} system admits a couple of global solutions $(\phi,n)$ in time with the following energy estimates:  
\begin{align}
\| \partial \Gamma^{\le K}\phi\|_{2}+\| \Gamma^{\le K}\phi\|_{2}+\|\Ga^{\le K}n\|_2\le C\varepsilon \Lg t\Rg^{\de}
\end{align}
for some constant $C>0$. Moreover, we have the following estimate for the lower-order energy:
\begin{align}
\| \partial \Gamma^{\le K-1}\phi\|_{2}+\| \Gamma^{\le K-1}\phi\|_{2}+\|\Gamma^{\le K-2}n\|_{2}\le C\varepsilon.
\end{align}

\texttt{(ii).} Such global solutions admit the following pointwise decay-in-time estimate: 
\begin{align}
    |\phi(t,x)|\le C_1 \varepsilon \Lg t\Rg^{-1},\quad |n(t,x)|\le C_2 \varepsilon \Lg t\Rg^{-\frac12} \Lg t-|x|\Rg^{-\frac{1+\delta}{2}},
\end{align}
for some positive constants $C_1,C_2>0$. 
\end{thm}

\begin{proof}
    The proof of Theorem~\ref{thm2} follows from a very similar bootstrap framework to the proof of Theorem~\ref{thm}, we therefore omit the details. In fact we can get the scattering of \eqref{eq:kg2} system too and we only state the result. 
\end{proof}

\begin{thm}
 Assume the same setting as in Theorem~\ref{thm2} is adopted and in addition we denote the energy space {$\mathcal{X}_{K}=H^{K}\times H^{K-1}$ and $\mathcal{Y}_{K-1}=\dot{H}^{K-1}\times \dot{H}^{K-2}\cap\dot{H}^1\times L^2$ }where $H^m, \dot{H}^m$ are standard Sobolev spaces in $\R^2$ that are defined in \eqref{eq:hsdot}. We also denote
$D=(D_i)_{i=1}^2$ to be the usual differential operator with $D_1=\pa_1,D_2=\pa_2$. Then the following statements hold:
 
 \texttt{(i)}. The solution $\phi$ scatters to a free solution in $\mathcal{X}_K$ as $t\to+\infty$, namely, there exists $(\Phi_{l_0},\Phi_{l_1})\in \mathcal{X}_K$ such that 
   \begin{align}
    \lim_{t\to+ \infty}\|(\phi,\partial_t\phi)-(\phi_l,\partial_t\phi_l)\|_{\mathcal{X}_K}= 0,
\end{align}
where $\phi_l$ is the linear homogeneous solution to the Klein-Gordon equation with initial data $(\Phi_{l_0},\Phi_{l_1})$.

\texttt{(ii)}. The solution $n$ scatters to a free solution weakly in $\mathcal{Y}_{K-1}$ as $t\to+ \infty$, i.e. there exists $(N_{l_0},N_{l_1})\in \mathcal{Y}_{K-1}$ such that for any $(F_0,F_1)\in \mathcal{Y}_{K-1}$ we have
\begin{align}
    &\lim_{t\to + \infty}\int_{\R^2} \na D^{K-2} (n-n_l)\cdot\na D^{K-2} f+\pa_t \na D^{K-2}(n-n_l)\cdot \pa_t\na D^{K-2} f=0,\\ 
\mbox{and}\ &\lim_{t\to + \infty}\int_{\R^2}\na(n-n_l)\cdot \na f+\pa_t(n-n_l)\pa_tf=0,
\end{align}
where $n_l$ is the linear homogeneous solution to the wave equation with initial data $(N_{l_0},N_{l_1})$ and $f$ is the linear homogeneous solution to the wave equation with initial data $(F_0,F_1)$.

\end{thm}

\appendix

\section{Estimate of $\|\widetilde{\Ga}\Gamma^{\beta}n_0\|_{2}^2$.}

In this section we give an approach to estimate $\|\wtd{\Ga}\Gamma^{\beta}n_0\|_{2}^2$ as in Lemma~\ref{lem3.7}. It is clear that $\Box \wtd{\Ga}\Ga^\beta n_0=0.$ Therefore it is known that we can write $\wtd{\Ga}\Ga^\beta n_0$ in the following mild form: 
\begin{align}\label{A.1}
   \wtd{\Ga} \Ga^\beta n_0(x)= (\cos t|\na|\  v_0)(x)+(\frac{\sin t|\na|}{|\na|}v_1)(x),
\end{align}
where $(v_0,v_1)=(\wtd{\Ga}\Ga^\beta n_0,\partial_t\wtd{\Ga}\Ga^\beta n_0)|_{t=0}$.

\begin{lem}\label{lemA1}
  Assume  $\wtd{\Ga}\Ga^\be n_0$ is defined as in \eqref{A.1} above, then it holds that
  \begin{align}\label{S4-4}
  \|\wtd{\Ga}\Ga^\be n_0\|_2\ls \|\Lg x\Rg \na v_0\|_2+ \log (t+2) \| \Lg x\Rg v_1\|_2.
\end{align}
\end{lem}
Before we give a proof to it, it is convenient to recall a useful lemma as below.
\begin{lem}[Log-type Hardy]
Let $B(0,1)\subset \R^2$, $h\in C_c^{\infty}(B(0,1))$. %and $\eta(x)=\frac1{|x|\log |x|}$. 
Then
\begin{align}\label{Hardy}
  \left \|\frac h{|x|\log |x|}\right\|_{L^2(B(0,1))}\lesssim \|\nabla h\|_{L^2(B(0,1))}.
\end{align}

\end{lem}
\begin{proof}
The proof is standard. First we choose
$\varphi(r)=-\frac{0.1}{r\log r}$.  Then
\begin{align*}
|h\varphi +\pa_{r}h|^2&=h^2\varphi^2 +\varphi\pa_{r}(h^2)+(\pa_r h)^2\ge 0,
\\ \implies\ \int_{B(0,1)} |\pa_{r}h|^2\ dx&\geq \int_{B(0,1)} (-\varphi^2+\frac1r\pa_{r}(r\varphi))h^2\ dx \\&=\int_{B(0,1)} (-\frac{0.01}{r^2(\log r)^2}+\frac{0.1}{r^2(\log r)^2})h^2\ dx
  \gtrsim\int_{B(0,1)} \frac{h^2}{r^2(\log r)^2}\ dx.
\end{align*}
This implies \eqref{Hardy}.
\end{proof}

\begin{proof}[Proof of Lemma~\ref{lemA1}]
    The proof for the 3D case is given in \cite{CX22}. We sketch the proof here for the sake of completeness. It is clear from the Fourier side that
\begin{align}
    \|\cos tD\ v_0\|_{L^2_x(\R^2)}\sim \|\cos(t|\xi|) \hat{v_0}(\xi)\|_{L^2_\xi(\R^2)}\lesssim\|\hat{v_0}\|_{L^2_\xi(\R^2)}\lesssim \|v_0\|_{L^2_x(\R^2)}. 
\end{align}
Similarly we can estimate the other term by splitting the region into 3 parts:
\begin{align}
    \left\| \frac{\sin tD}{D}v_1\right\|_2^2\sim\int_{\R^2}\frac{(\sin(t|\xi|))^2}{|\xi|^2}|\hat{v_1}(\xi)|^2\ d\xi
    =\int_{0<|\xi|<\frac1t}
    +\int_{\frac1t\le|\xi|<\frac12}+\int_{\frac12\le|\xi|<\infty}\frac{(\sin(t|\xi|))^2}{|\xi|^2}|\hat{v_1}(\xi)|^2\ d\xi.
\end{align}
%where
%\begin{align}
%    I_1=\int_{0<|\xi|<\frac1t}\frac{(\sin(t|\xi|))^2}{|\xi|^2}|\hat{v_1}(\xi)|^2\ d\xi,  I_1=\int_{0<|\xi|<\frac1t}\frac{(\sin(t|\xi|))^2}{|\xi|^2}|\hat{v_1}(\xi)|^2\ d\xi, I_1=\int_{0<|\xi|<\frac1t}\frac{(\sin(t|\xi|))^2}{|\xi|^2}|\hat{v_1}(\xi)|^2\ d\xi.
%\end{align}
The estimates follow by direct computation after applying \eqref{Hardy}:
\begin{align}
&\int_{\frac1t\le|\xi|<\frac12}\frac{(\sin(t|\xi|))^2}{|\xi|^2}|\hat{v_1}(\xi)|^2\ d\xi\\ \lesssim& \int_{\frac1t\le|\xi|<\frac12}\frac{|\hat{v_1}(\xi)\phi(\xi)|^2}{|\xi|^2}\ d\xi \lesssim\  \int_{\frac1t\le|\xi|<\frac12}\frac{|\hat{v_1}(\xi)\phi(\xi)|^2}{|\xi|^2 |\log |\xi||^2} |\log |\xi||^2\ d\xi\\
   \lesssim&|\log t|^2  \int_{\frac1t\le|\xi|<\frac12}\frac{|\hat{v_1}(\xi)\phi(\xi)|^2}{|\xi|^2 |\log |\xi||^2}\ d\xi\lesssim |\log t|^2 \int_{|\xi|<2}|\nabla(\hat{v_1}(\xi)\phi(\xi))|^2\ d\xi \qquad(\text{by } \eqref{Hardy})
   \\ \lesssim& |\log t|^2\|\hat{v_1}\|_{H^1}^2\lesssim |\log t|^2 \|\Lg x\Rg v_1\|_2^2,
\end{align}
 where we choose $\phi\in C_c^{\infty}$ such that $\phi=1$ when $|\xi|\le \frac23$ and $\phi=0$ when $|\xi|>1$. The estimates for $I_1$ and $I_3$ follow from similar arguments and we leave for the readers. These then conclude \eqref{S4-4}.
 
 %By this choice $|\na \phi|\lesssim \frac{1}{1-\frac2t}\lesssim 1$.}

\end{proof}

%\begin{rem}
%   Note that compared to the $X(\partial)$ trick (conformal energy estimates), this propagator estimate needs an additional assumption:
%\begin{align}\label{4.48}
%    \|v_0\|_2+\|\Lg x\Rg v_1\|_2\lesssim\sum_{j=0}^K\|\Lg x\Rg^{j+1}\na^j(\na n_0,n_1)\|_2\le C\varepsilon.
%\end{align}
%One clearly sees the difference between \eqref{3.12a} and \eqref{4.48}.
%\end{rem}

\section{Estimate of the homogeneous wave component}
In this section we prove Lemma~\ref{lemhom}. %and Lemma~\ref{cor:gradinh}.

%Let $v(t,x)$ be a solution of the following system:
%\begin{align}\label{vhom}
% \begin{cases}
% \Box v=0;\\
% v(0,x)=v_{0}(x),\ \pa_{t}v(0,x)=v_{1}(x).
% \end{cases}
%\end{align}
%It is known that the solution to such system can be written as the following mild form: 
%\begin{align*}
%    v(t,x)=&\cos tD\ (v_0)(x)+\frac{\sin tD}{D}(v_1)(x)\\=&\frac1{2\pi}\pa_{t}\int_{|x-y|<t}\frac{v_{0}(y)dy}{\sqrt{t^2-|x-y|^2}}+\frac1{2\pi}\int_{|x-y|<t}\frac{v_{1}(y)dy}{\sqrt{t^2-|x-y|^2}},
%\end{align*}
%where $D=|\nabla|$.

\begin{lem}[Homogeneous estimate]\label{lemBB2} Assume $v$ is defined as in \eqref{vhom}-\eqref{vhom1} and assume $v_0, v_1\in \mathcal{S}(\R^2)$, then the following estimate holds: for any $t\ge 2$ and $x\in\R^2$,
 \begin{align}\label{BB2}
   | v(t,x)|\lesssim t^{-\frac12}\Lg t-|x|\Rg^{-{\frac12}}(\|\Lg y\Rg^{\frac12}\Lg \nabla\Rg v_{1}\|_{L^1}+\|\Lg y\Rg^{\frac12}\Lg \nabla\Rg^{ 2}v_{0}\|_{L^1}).
 \end{align}
\end{lem}

\begin{proof}
  For $|x-y|<t$, we have $t+|x-y|\sim t$. We first consider the following integral $\int_{|x-y|<t}\frac{v_1(y)\ dy}{\sqrt{t^2-|x-y|^2}}$. We denote $r=|x|$ and then consider several cases:
  
  {\bf Case 1.} $|x-y|<t-1$.
  We have 
  \begin{align}
      |t-r|\lesssim \Lg t-|x-y|\Rg\Lg y\Rg\lesssim (t-|x-y|)\Lg y\Rg
      \quad \Rightarrow \quad  
      \frac1{\sqrt{t^{2}-|x-y|^2}}\lesssim t^{-\frac12}\Lg t-r\Rg^{-\frac12}\Lg y\Rg^{\frac12}.
  \end{align}
 This implies  
 \begin{align}
     \int_{|x-y|<t-1}\frac{v_{1}(y)dy}{\sqrt{t^2-|x-y|^2}}
     \lesssim  t^{-\frac12}\Lg t-|x|\Rg^{-\frac12}\|\Lg y\Rg^{\frac12}\nabla^{\leq 1}v_{1}\|_{L^1}.
 \end{align}
 {\bf Case 2.} $t-1<|x-y|<t$. In this case we have $t-|x|\le |y|+1$ and $|x|-t\le |y|$, therefore $|t-r|\ls \Lg y\Rg$.
 \begin{align}
    &\int_{t-1<|x-y|<t}\frac{v_{1}(y)dy}{\sqrt{t^2-|x-y|^2}}
    \\ =&\int_{t-1}^{t}\int_{|\omega|=1}\frac{v_{1}(x+\rho\omega)}{\sqrt{t^2-\rho^2}}\rho d\rho d\omega=-\int_{t-1}^{t}\int_{|\omega|=1}v_{1}(x+\rho\omega)\pa_{\rho}(\sqrt{t^2-\rho^2})d\rho d\omega
     \\=&\int_{|\omega|=1}v_{1}(x+(t-1)\omega) \sqrt{2t-1}  d\omega +\int_{t-1}^{t}\int_{|\omega|=1}\sqrt{t^2-\rho^2}(\nabla v_{1})(x+\rho\omega)\cdot\omega d\rho d\omega\\
     \coloneqq&J_1+J_2.
 \end{align}
 For $J_{1}$, note that by $\Lg t-r\Rg\ls \Lg x+(t-1) \omega\Rg$ we have 
 \begin{align}
   J_{1}\lesssim &\Lg t-r\Rg^{-\frac12} \int_{|\omega|=1}\left|v_1(x+(t-1)\omega)\Lg x+(t-1)\omega\Rg^{\frac12}\sqrt{2t-1}\right|\ d\omega\\
\ls&\Lg t-r\Rg^{-\frac12} \int_{|\omega|=1}\Big|\int_{t-1}^{\infty}\frac{d}{d\rho}\left(v_{1}(x+\rho\omega)\Lg x+\rho\omega\Rg^{\frac12}\sqrt{2\rho+1}\right)d\rho \Big| d\omega
   %\\ \lesssim & \Lg t-r\Rg^{-\frac12} \int_{|\omega|=1}\left|\int_{t-1}^{\infty}\left((\nabla v_{1})(x+\rho \omega)\cdot\omega\sqrt{2\rho+1}+v_{1}(x+\rho\omega)\frac{1}{\sqrt{2\rho+1}}\right)d\rho\right| d\omega
   \\ \lesssim &\Lg t-r\Rg^{-\frac12}  \ t^{-\frac12}\|\Lg y\Rg^{\frac12}\Lg\nabla\Rg v_{1}\|_{L^1}.
 \end{align}
 For $J_{2}$, we have 
 \begin{align}
    J_{2}\lesssim & \ t^{-\frac12}\int_{t-1}^{t}\int_{|\omega|=1}|(\nabla v_{1})(x+\rho\omega)|\rho d\rho d\omega \\ \lesssim &\  t^{-\frac12}\Lg t-r\Rg^{-\frac12}\|\Lg y\Rg^{\frac12}\nabla v_{1}\|_{L^1} \quad (\text{since }|t-r|\lesssim |t-|x-y||+|y|\lesssim \Lg y\Rg ).
 \end{align}
  Therefore, 
  \begin{align}
    \int_{t-1<|x-y|<t}\frac{v_{1}(y)dy}{\sqrt{t^2-|x-y|^2}}
    \lesssim  t^{-\frac12}\Lg t-r\Rg^{-\frac12}\|\Lg y\Rg^{\frac12}\Lg \nabla\Rg v_{1}\|_{L^1}.
 \end{align}
 Now we consider the second part. Let $z=\frac{y-x}t$, then 
 \begin{align}
   &\pa_{t}\int_{|x-y|<t}\frac{v_{0}(y)dy}{\sqrt{t^2-|x-y|^2}}
   =\pa_{t}\int_{|z|<1}\frac{t v_{0}(x+tz)dz}{\sqrt{1-|z|^2}}
  \\ =& \int_{|z|<1}\frac{ v_{0}(x+tz)+t(\nabla v_{0})(x+tz)\cdot z}{\sqrt{1-|z|^2}}
  \\ =&\ t^{-1}\int_{|x-y|<t}\frac{ v_{0}(y)}{\sqrt{t^2-|x-y|^2}} dy+\int_{|x-y|< t}\frac{ \nabla v_{0}(y)\cdot\frac{y-x}{t}}{\sqrt{t^2-|x-y|^2}} dy \qquad (\text{let } y=x+t z).
 \end{align}
 Similarly, we have 
 \begin{align}
    & t^{-1}\int_{|x-y|<t}\frac{ v_{0}(y)}{\sqrt{t^2-|x-y|^2}} dy \lesssim \ t^{-\frac32}\Lg t-r\Rg^{-\frac12}\|\Lg y\Rg^{\frac12}\nabla^{\leq 1}v_0\|_{L^1};\\
    & \int_{|x-y|< t-1}\frac{ \nabla v_{0}(y)\cdot\frac{y-x}{t}}{\sqrt{t^2-|x-y|^2}} dy
    \lesssim \int_{|x-y|< t-1}\frac{ |\nabla v_{0}(y)|}{\sqrt{t^2-|x-y|^2}} dy
    \lesssim t^{-\frac12}\Lg t-r\Rg^{-\frac12}\|\Lg y\Rg^{\frac12}\Lg\nabla\Rg^2 v_0\|_{L^1}.
 \end{align}
 The remaining terms can be estimated as follows:
 \begin{align}
   & \int_{ t-1<|x-y|< t}\frac{ \nabla v_{0}(y)\cdot\frac{y-x}{t}}{\sqrt{t^2-|x-y|^2}} dy\\=& \int_{|\omega|=1}\int_{ t-1}^{ t}\frac{ \nabla v_{0}(x+\rho\omega)\cdot\frac{\rho\omega}{t}}{\sqrt{t^2-\rho^2}} \rho d\rho d\omega
   =-\int_{|\omega|=1}\int_{ t-1}^{ t}  \nabla v_{0}(x+\rho\omega)\cdot\frac{\rho\omega}{t} \pa_{\rho}(\sqrt{t^2-\rho^2} ) d\rho d\omega\\=&\int_{|\omega|=1} \nabla v_{0}(x+(t-1)\omega)\cdot\frac{(t-1)\omega}{t}\sqrt{2t-1}  d\omega\\&+\int_{t-1}^{t}\int_{|\omega|=1}\left[((\omega\cdot\nabla)^2v)(x+\rho\omega)+\frac1\rho (\nabla v_0)(x+\rho\omega)\cdot\omega\right]\frac{\sqrt{t^2-\rho^2}}{t}\rho d\rho d\omega\\ \coloneqq&J_3+J_4.
 \end{align}
 Here we can estimate $J_3$ as below:
 \begin{align}
     J_{3}\lesssim& \Lg t-r\Rg^{-\frac12}\int_{|\omega|=1}|\Lg x+(t-1)\omega\Rg^{\frac12}(\omega\cdot\nabla) v_0(x+(t-1)\omega)\sqrt{2t-1}| d\omega\\ \lesssim & \Lg t-r\Rg^{-\frac12}\int_{|\omega|=1}|\int_{t-1}^{\infty}\frac d{d\rho}\left(\Lg x+\rho\omega\Rg^{\frac12}(\omega\cdot\nabla) v_0(x+\rho\omega)\sqrt{2\rho+1}\right) d\rho| d\omega\\ \lesssim &\Lg t-r\Rg^{-\frac12}\int_{|\omega|=1}|\int_{t-1}^{\infty} \Lg x+\rho\omega\Rg^{-\frac12}|(\omega\cdot \nabla) v_0(x+\rho\omega)|\frac{\sqrt{2\rho+1}}{\rho} \rho d\rho| d\omega
     \\&+\Lg t-r\Rg^{-\frac12}\int_{|\omega|=1}|\int_{t-1}^{\infty} \Lg x+\rho\omega\Rg^{\frac12}\left((\omega\cdot\nabla)^2 v_0(x+\rho\omega)\sqrt{2\rho+1}+(\omega\cdot\nabla) v_0(x+\rho\omega)\frac{1}{\sqrt{2\rho+1}}\right)  d\rho| d\omega 
     \\ \lesssim &\ t^{-\frac12}\Lg t-r\Rg^{-\frac12}\|\nabla v_0\|_{L^1}+t^{-\frac12}\Lg t-r\Rg^{-\frac12}\|\Lg y\Rg^{\frac12}\nabla^2 v_0\|_{L^1}+t^{-\frac32}\Lg t-r\Rg^{-\frac12}\|\Lg y\Rg^{\frac12}\nabla v_0\|_{L^1}.
 \end{align}
 $J_{4}$ can be handled similarly: 
 \begin{align}
     J_{4}\lesssim  t^{-\frac12}\Lg t-r\Rg^{-\frac12}\|\Lg y\Rg^{\frac12}\Lg \nabla\Rg^{ 2}v_0\|_{L^1}.
 \end{align}
 Therefore we conclude \eqref{BB2}
\end{proof}
%\begin{rem}
 %   Note that in the first inequality above, we require that \begin{align}
 %    \lim_{\rho \to \infty} v_1(x+\rho\omega)\cdot\sqrt{2\rho+1}=0,
 %\end{align}
 %however it is generally not true even when $\|\Lg y\Rg^{\frac12}\na^{\leq 1}v_1\|_{L^1(\R^2)}\lesssim 1$. But this issue can be resolved by approaching $v_1$ with a sequence of smooth compactly supported functions. 
%end{rem} 

\section{Verification of the smallness conditions of the Initial data}
In this section we verify the smallness condition \eqref{ID2} is consistent with the vector fields:
\begin{prop}\label{prop}
Assume $(\Phi_0,\Phi_1,N_0,N_1)$ satisfy the condition \eqref{ID2}, then for any $J\le K$ we have
    \begin{align}
     &\|\Ga^{\le J}\phi|_{t=0}\|_2\lesssim \big(\|\Lg x\Rg^J \Lg \na\Rg^J  \Phi_0\|_2+\|\Lg x\Rg^J  \Lg \na \Rg^{J-1}\Phi_1\|_2\big)\big(1+\sum_{j=0}^{J-2}\|\Lg x\Rg^{j+1}\na^j (\na N_0,N_1)\|_2\big),\label{B1}\\
    % &\|\Ga^{\le J} n|_{t=0}\|_2\lesssim \sum_{j=0}^{J-1}\|\Lg x\Rg^{j+1}\na^j (\na N_0,N_1)\|_2.\label{B2}
    &\|\Ga^{\le J} n|_{t=0}\|_2\lesssim \big(\|\Lg x\Rg^{J-1} \Lg \na\Rg^{J-1}  \Phi_0\|_2+\|\Lg x\Rg^{J-1}  \Lg \na \Rg^{J-2}\Phi_1\|_2\big)\big(1+\sum_{j=0}^{J-1}\|\Lg x\Rg^{j+1}\na^j (\na N_0,N_1)\|_2\big).\label{B2}
    \end{align}
 In particular, the following also holds:
 \begin{align}
          &\|\pa \Ga^{\le J}\phi|_{t=0}\|_2\lesssim \big(\|\Lg x\Rg^J \Lg \na\Rg^{J+1}  \Phi_0\|_2+\|\Lg x\Rg^J  \Lg \na \Rg^{J}\Phi_1\|_2\big)\big(1+\sum_{j=0}^{J-1}\|\Lg x\Rg^{j+1}\na^j (\na N_0,N_1)\|_2\big),\label{B1a}\\
         % &\|\wtd{\Ga}\Ga^{\le J-1} n|_{t=0}\|_2\lesssim \sum_{j=0}^{J-1}\|\Lg x\Rg^{j+1}\na^j (\na N_0,N_1)\|_2.\label{B2a}
     &\|\wtd{\Ga}\Ga^{\le J-1} n|_{t=0}\|_2\lesssim \big(\|\Lg x\Rg^{J-1} \Lg \na\Rg^{J-1}  \Phi_0\|_2+\|\Lg x\Rg^{J-1}  \Lg \na \Rg^{J-2}\Phi_1\|_2\big)\big(1+\sum_{j=0}^{J-1}\|\Lg x\Rg^{j+1}\na^j (\na N_0,N_1)\|_2\big).\label{B2a}
 \end{align}
Similarly, we also have{
\begin{equation}
\label{B3}
\begin{aligned}
     &\|\Lg x\Rg \log(\Lg x\Rg) \Ga^{\le K}\phi|_{t=0}\|_2\\
     \lesssim&\big(\|\Lg x\Rg^{K+1}\log (\Lg x\Rg) \Lg \na\Rg^{K}  \Phi_0\|_2+\|\Lg x\Rg^{K+1}\log (\Lg x\Rg)   \Lg \na \Rg^{K-1}\Phi_1\|_2\big)\big(1+\sum_{j=0}^{K-2}\|\Lg x\Rg^{j+1}\na^j (\na N_0,N_1)\|_2\big),
\end{aligned}
\end{equation}}
and{
\begin{equation}\label{B4}
    \begin{aligned}
       &\|(\Lg x\Rg^{\frac32}  \Lg \na\Rg^3\wtd{\Ga}^{\le K-4} n_0)|_{t=0}\|_1+\|(\Lg x\Rg^{\frac32}  \Lg \na\Rg^2 \pa_t\wtd{\Ga}^{\le K-4} n_0)|_{t=0}\|_1\\
       \ls& \sum_{j=0}^{K-5}(\|\Lg x\Rg^{j+\frac52} \na^{j+3}(\na N_0,N_1)\|_1+\|\Lg x\Rg^{j+\frac52} \na^{j}(\na N_0,N_1)\|_1).
    \end{aligned}
\end{equation}}
All the constants in the inequalities are exact and computable; for simplicity we identify all quantities related to the initial conditions by $C_i\varepsilon$, where $C_i$ is a fixed exact positive constant that can be computed directly.  

\end{prop}
\begin{proof}
    The proof relies on an induction that is very similar to \cite{CX22}, therefore we only sketch the idea. In fact the wave filed $n$ is simpler to deal with because of the decomposition $n=n_0+A^\al \pa_\al n$ so that the initial conditions $(N_0,N_1)$ can be taken care of essentially through the homogeneous equation up to some higher order perturbation. For the Klein-Gordon field $\phi$, with no loss, we only consider the ``worst'' scenario as follows:
\begin{align}
    \|\Lg x\Rg^J \partial_t^J\phi|_{t=0}\|_2\lesssim \|\Lg x\Rg^J \partial_t^{J-2}\Delta \phi|_{t=0}\|_2+\|\Lg x\Rg^J \partial_t^{J-2}\phi|_{t=0}\|_2+\|\Lg x\Rg^J\partial_t^{J-2}(\phi \pa n)|_{t=0}\|_2.
\end{align}
Note that by induction assumption we shall have
\begin{align}
    \|\Lg x\Rg^J \partial_t^{J-2}\Delta\phi|_{t=0}\|_2+\|\Lg x\Rg^J \partial_t^{J-2}\phi|_{t=0}\|_2\lesssim
\| \Lg x\Rg^{J}\Lg \na\Rg^J\Phi_0\|_2+\| \Lg x\Rg^{J}\Lg \na\Rg^{J-1}\Phi_1\|_2.
\end{align}
For the nonlinear part, we also only present the worst case:
\begin{align}
    \|\Lg x\Rg^J\partial_t^{J-2}(\phi \pa n)|_{t=0}\|_2\lesssim&\|\Lg x\Rg^J(\partial_t^{\le  J-2}\phi)( \partial_t^{\le \frac{J}{2}}n)|_{t=0}\|_2+\|\Lg x\Rg^J(\partial_t^{\le J-1}n)( \partial_t^{\le \frac{J-2}{2}}\phi)|_{t=0}\|_2\\
    \lesssim&\|\Lg x\Rg^J\partial_t^{\le J-2}\phi |_{t=0}\|_2\|\partial_t^{\le \frac{J}{2}} n|_{t=0}\|_\infty+\|\Lg x\Rg^J\partial_t^{\le \frac{J-2}{2}}\phi |_{t=0}\|_\infty\|\partial_t^{\le {J-1}} n|_{t=0}\|_2.
\end{align}
Then by induction and the standard $H^2(\R^2)\hookrightarrow L^\infty(\R^2)$ embedding we have
\begin{align}
&\|\Lg x\Rg^J\partial_t^{\le J-2}\phi |_{t=0}\|_2\|\partial_t^{\le \frac{J}{2}} n|_{t=0}\|_\infty\\
\lesssim& (\| \Lg x\Rg^{J}\Lg \na\Rg^{J-2}\Phi_0\|_2+\| \Lg x\Rg^{J}\Lg \na\Rg^{J-3}\Phi_1\|_2)\sum_{j=0}^{\frac{J}{2}+1}\|\Lg x\Rg^{j+1}\na^j(\na N_0,N_1)\|_2
\end{align}
and 
\begin{align}
&\|\Lg x\Rg^J\partial_t^{\le \frac{J-2}{2}}\phi |_{t=0}\|_\infty\|\partial_t^{\le {J-1}} n|_{t=0}\|_2\\
\lesssim& (\| \Lg x\Rg^{J}\Lg \na\Rg^{\frac{J+2}{2}}\Phi_0\|_2+\| \Lg x\Rg^{J}\Lg \na\Rg^{\frac{J}{2}}\Phi_1\|_2)\sum_{j=0}^{J-2}\|\Lg x\Rg^{j+1}\na^j(\na N_0,N_1)\|_2.
\end{align}
The other cases follow from similar arguments and we omit the details.

\end{proof}

\section*{Acknowledgement}
X. Cheng was partially supported by the Shanghai ``Super Postdoc" Incentive Plan (No. 2021014), the International
Postdoctoral Exchange Fellowship Program (No. YJ20220071) and China Postdoctoral Science Foundation (Grant No. 2022M710796, 2022T150139). 

% \section*{Declarations}
% \noindent \textbf{Data availability} The author declares that data sharing is not applicable to this article since no datasets were generated or
% analyzed during the current study.

% \bigskip

% \noindent\textbf{Conflict of interest} The author declares no conflict of interest.

\bibliographystyle{abbrv}

\end{document}